\documentclass{amsart}
\makeatletter
\def\@tocline#1#2#3#4#5#6#7{\relax
  \ifnum #1>\c@tocdepth 
  \else
    \par \addpenalty\@secpenalty\addvspace{#2}%
    \begingroup \hyphenpenalty\@M
    \@ifempty{#4}{%
      \@tempdima\csname r@tocindent\number#1\endcsname\relax
    }{%
      \@tempdima#4\relax
    }%
    \parindent\z@ \leftskip#3\relax \advance\leftskip\@tempdima\relax
    \rightskip\@pnumwidth plus4em \parfillskip-\@pnumwidth
    #5\leavevmode\hskip-\@tempdima
      \ifcase #1
       \or\or \hskip 1em \or \hskip 2em \else \hskip 3em \fi%
      #6\nobreak\relax
    \dotfill\hbox to\@pnumwidth{\@tocpagenum{#7}}\par
    \nobreak
    \endgroup
  \fi}
\makeatother

\usepackage[english]{babel}
\usepackage{amsmath,amssymb,amscd,amsfonts,amsthm}

\usepackage{enumitem}

\usepackage[msc-links, alphabetic, abbrev]{amsrefs}

\usepackage{hyperref}

\usepackage[table]{xcolor}
\usepackage{colortbl}
\usepackage{rotating}
\usepackage{booktabs}
\usepackage{longtable}

\definecolor{light-gray}{gray}{0.85}

\usepackage{wrapfig}

\usepackage{comment,stackrel}

\usepackage[all]{xy}

\newtheorem{theorem}{Theorem}[section]
\newtheorem{family}{}
\newtheorem{proposition}[theorem]{Proposition}
\newtheorem{prop-defn}[theorem]{Proposition-Definition}

\newtheorem{conjecture}[theorem]{Conjecture}

\theoremstyle{definition}
\newtheorem{definition}[theorem]{Definition}
\newtheorem{remark}[theorem]{Remark}
\newtheorem{notation}[theorem]{Notation}
\newtheorem{example}[theorem]{Example}

\newcommand\cC{\mathcal{C}}
\newcommand\cD{\mathcal{D}}
\newcommand\cE{\mathcal{E}}
\newcommand\cF{\mathcal{F}}

\newcommand\cH{\mathcal{H}}

\newcommand{\cK}{\mathcal{K}}
\newcommand\cL{\mathcal{L}}

\newcommand\cO{\mathcal{O}}

\newcommand\cX{\mathcal{X}}
\newcommand\cY{\mathcal{Y}}

\renewcommand\AA{\mathbb{A}}

\newcommand\CC{\mathbb{C}}

\newcommand\GG{\mathbb{G}}

\newcommand\LL{\mathbb{L}}

\newcommand\NN{\mathbb{N}}

\newcommand\PP{\mathbb{P}}
\newcommand\QQ{\mathbb{Q}}
\newcommand\RR{\mathbb{R}}

\newcommand\ZZ{\mathbb{Z}}

\newcommand\bfe{\mathbf{e}}

\newcommand\rA{\mathrm{A}}
\newcommand\rB{\mathrm{B}}
\newcommand\rC{\mathrm{C}}

\newcommand\rE{\mathrm{E}}

\newcommand\rG{\mathrm{G}}
\newcommand\rH{\mathrm{H}}
\newcommand\rI{\mathrm{I}}

\newcommand\rL{\mathrm{L}}

\newcommand\rN{\mathrm{N}}
\newcommand\rO{\mathrm{O}}

\newcommand\rR{\mathrm{R}}
\newcommand\rS{\mathrm{S}}
\newcommand\rT{\mathrm{T}}

\newcommand\rV{\mathrm{V}}
\newcommand\rW{\mathrm{W}}

\newcommand\rmd{\mathrm{d}}
\newcommand\rme{\mathrm{e}}

\newcommand\rmi{\mathrm{i}}

\newcommand\rml{\mathrm{l}}
\newcommand\rmm{\mathrm{m}}

\newcommand\rmo{\mathrm{o}}
\newcommand\rmp{\mathrm{p}}

\newcommand\rms{\mathrm{s}}
\newcommand\rmt{\mathrm{t}}

\newcommand\rmv{\mathrm{v}}

\newcommand{\frakg}{\mathfrak{g}}

\newcommand{\frakt}{\mathfrak{t}}

\newcommand{\phiv}{\varphi}
\newcommand{\epsi}{\varepsilon}
\newcommand{\inv}{^{-1}}
\newcommand{\al}{\alpha}
\newcommand{\into}{\hookrightarrow} 
\newcommand{\onto}{\twoheadrightarrow} 
\DeclareMathOperator{\Hom}{Hom} 
\DeclareMathOperator{\Spec}{Spec} 
\DeclareMathOperator{\Pic}{Pic} 
\DeclareMathOperator{\coker}{coker}

\renewcommand{\setminus}{\smallsetminus}

\newcommand{\git}{/ \! \! /} 
\newcommand{\bmu}{\boldsymbol{\mu}} 
\newcommand{\maximallymutable}{\mathbf{L}} 
\newcommand{\Tmaximallymutable}{\mathbf{L}^\mathrm{T}} 
\newcommand{\chenruan}{\mathrm{H}_{\mathrm{CR}}^\bullet} 
\newcommand{\juniorchenruan}{\widetilde{\mathrm{H}}_\mathrm{CR}^{<2}} 
\newcommand{\bfone}{\mathbf{1}} 
\DeclareMathOperator{\BBox}{Box} 
\newcommand{\Nef}{\mathrm{Nef}} 
\newcommand{\NE}{\mathrm{NE}} 

\def\ambientfactor#1#2{\frac{\prod\limits_{\substack{a \leq 0 \\ \langle a \rangle = \langle #1 \rangle}} (#2 + az)}{\prod\limits_{\substack{a \leq #1 \\ \langle a \rangle = \langle #1 \rangle}} (#2 + az)}} 

\def\ambientfactorwithoutnumerator#1#2{\frac{1}{\prod\limits_{\substack{0 < a \leq #1 \\ \langle a \rangle = \langle #1 \rangle}} (#2 + az)}} 

\def\bundlefactorwithoutdenominator#1#2{\prod\limits_{\substack{0 < a \leq #1 \\ \langle a \rangle = \langle #1 \rangle}} (#2 + az)}

\def\bundlefactorwithdenominator#1#2{\frac{\prod\limits_{\substack{a \leq #1 \\ \langle a \rangle = \langle #1 \rangle}} (#2 + az)}{\prod\limits_{\substack{a \leq 0 \\ \langle a \rangle = \langle #1 \rangle}} (#2 + az)}}

\title[Quantum periods of del Pezzo surfaces with $\frac{1}{3}(1, 1)$ singularities]{On the quantum periods of del Pezzo surfaces with $\frac{1}{3}(1, 1)$ singularities}

\author[A.~Oneto]{Alessandro Oneto}
\address{Department of Mathematics, Stockholm University, SE-106 91, Stockholm,  Sweden}
\email{oneto@math.su.se}

\author[A.~Petracci]{Andrea Petracci}
\address{Department of Mathematics, Imperial College London, 180 Queen's Gate, London SW7 2AZ, United Kingdom}
\email{a.petracci13@imperial.ac.uk}


\subjclass{X} 

\begin{document}

\maketitle

\begin{abstract}

In earlier joint work with our collaborators Akhtar, Coates, Corti, Heuberger, Kasprzyk, Prince and Tveiten, we gave a conjectural classification of a broad class of orbifold del~Pezzo surfaces, using Mirror Symmetry.  We proposed that del~Pezzo surfaces $X$ with isolated cyclic quotient singularities such that $X$ admits a $\QQ$-Gorenstein toric degeneration correspond under Mirror Symmetry to maximally mutable Laurent polynomials $f$ in two variables, and that the quantum period of such a surface $X$, which is a generating function for Gromov--Witten invariants of $X$, coincides with the classical period of its mirror partner $f$.

In this paper, we prove a large part of this conjecture for del~Pezzo surfaces with $\frac{1}{3}(1,1)$ singularities, by computing many of the quantum periods involved. Our tools are the Quantum Lefschetz theorem and the Abelian/non-Abelian Correspondence; our main results are contingent on, and give strong evidence for, conjectural generalizations of these results to the orbifold setting.
\end{abstract}

\tableofcontents

\addtocontents{toc}{\protect\setcounter{tocdepth}{1}} 

\section{Introduction} \label{sec:intro}

In \cite{conjectures}, together with our coauthors Akhtar, Coates, Corti, Heuberger, Kasprzyk, Prince and Tveiten, we gave a conjectural classification of a broad class of del~Pezzo surfaces with isolated cyclic quotient singularities, using Mirror Symmetry.  We proposed that del~Pezzo surfaces $X$ with isolated cyclic quotient singularities such that $X$ admits a $\QQ$-Gorenstein toric degeneration correspond under Mirror Symmetry to maximally mutable Laurent polynomials $f$ in two variables, and that the \emph{quantum period} of such a surface $X$ coincides with the \emph{classical period} of its mirror partner $f$.  The quantum period of $X$ here is a generating function for genus-zero Gromov--Witten invariants of $X$ that depends on certain natural parameters, which correspond to Reid's ``junior classes'' in the Chen--Ruan cohomology of $X$; on the other hand the Laurent polynomial $f$, being maximally mutable, also depends on certain parameters. We conjectured that the quantum period of $X$ and the classical period of $f$ coincide after an affine-linear identification of the parameter spaces involved.  This is Conjecture~B in~\cite{conjectures}; it is restated as Conjecture~\ref{conj:conj_B} below.

There are 26 families of del~Pezzo surfaces with $\frac{1}{3}(1,1)$ points that admit a $\QQ$-Gorenstein degeneration to a toric surface. In this paper, for 25 of these surfaces we compute the restriction of the quantum period to a nonempty affine subspace of the parameter space and we check that this matches with the classical period along an appropriate subspace of the space of maximally mutable Laurent polynomials.  When combined with \cite{singularity_content, minimal_polygons, alessio_liana, al_ketil}, this represents a substantial step towards Conjecture~B of \cite{conjectures}  for del~Pezzo surfaces with $\frac{1}{3}(1,1)$ points: it establishes a weaker form of Conjecture~B for 25 of the 26 surfaces involved. Our main results are stated in Theorem \ref{thm:main_result}.

Computing the quantum period of orbifolds is a hard problem in Gromov--Witten theory, and our computations are at the limit of the currently available techniques.  Our calculations depend on -- and provide strong evidence for -- natural conjectural generalisations of the Quantum Lefschetz theorem and the Abelian/non-Abelian Correspondence to the orbifold setting.  These are stated in \S\ref{sec:quantum_Lef_principle} and \S\ref{sec:setup_abelian} below.  When combined with toric mirror theorems, these generalizations allow the computation of quantum periods for many orbifolds that are either 
\begin{itemize}
\item[(a)] complete intersections in toric Deligne--Mumford stacks; or
\item[(b)] zero loci of regular sections of homogeneous vector bundles on Deligne--Mumford quotient stacks $[V \git G]$, where $V$ is a representation of a reductive group $G$.
\end{itemize}
In \cite{alessio_liana} Corti and Heuberger have given models of the form (a) or (b) for 25 of the 26 del~Pezzo surfaces with $\frac{1}{3}(1,1)$ points that admit a $\QQ$-Gorenstein degeneration to a toric surface, and we use these models to compute the quantum periods.  For the missing surface an Italian-style birational description is known, but no explicit description as a complete intersection in a toric stack or a reasonably good GIT quotient has yet been found. This gap prevents us from computing its quantum period. We hope to return to this case in a future paper.

\subsection*{Highlights} In \S\ref{sec:background}  we give an overview of the paper in more detail.  The impatient reader may wish, however, to skip ahead to the following sections.
\begin{itemize}[leftmargin=*]
\item In \S\ref{subsec:example_2} and \S\ref{subsec:example_3} we compute the quantum periods of two surfaces that are complete intersections in toric orbifolds; these are applications of our conjectural generalisation of the Quantum Lefschetz theorem from \S\ref{sec:quantum_Lef_principle}.
\item In \S\ref{subsec:sample_abelian_nonabelian} we compute the quantum period of a complete intersection in a weighted Grassmannian; this is an application of our conjectural generalisation of the Abelian/non-Abelian Correspondence from \S\ref{sec:setup_abelian}.
\item On page \pageref{subsec:X_4,1/3} we compute the quantum period of a del~Pezzo surface
  that does not admit a $\QQ$-Gorenstein degeneration to a toric surface.  This surface is a complete intersection in a toric orbifold, and as such has a Hori--Vafa mirror, but this mirror model does not admit a torus chart.
\end{itemize}

\subsection*{Plan of the paper}
In Section \ref{sec:background} we provide the basic notions on quantum periods of del Pezzo surfaces, present the conjectural picture for Mirror Symmetry for del Pezzo surfaces with isolated cyclic quotient singularities (Conjecture \ref{conj:conj_B}), and state our main result (Theorem \ref{thm:main_result}).
In Section \ref{sec:quantum period} we briefly recall Gromov--Witten theory for stacks (\S\ref{sec:GW_for_stacks}), Givental's symplectic formalism (\S\ref{sec:Givental_formalism}), and the definition of the quantum period of a Fano orbifold (\S\ref{sec:quantum_period_Fano_orbifold}).
In Section \ref{sec:toric} we focus on toric stacks: we recall the formalism of stacky fans (\S\ref{sec:stacky_fans}) and the toric mirror theorem (\S \ref{sec:mirror_theorem}); in \S\ref{subsec:example_1} we compute the quantum periods of the blow-up of $\PP(1,1,3)$ at one point.
In Section \ref{sec:toric_complete_intersections} we deal with complete intersections in toric stacks: in \S\ref{sec:quantum_Lef_principle} we give a brief survey of the Quantum Lefschetz theorem and we state a conjectural generalisation, in \S\ref{sec:Lefschetz_toric_orbifolds} we discuss how to apply Quantum Lefschetz to compute the quantum period of toric complete intersections, and we give two examples in \S\ref{subsec:example_2} and \S\ref{subsec:example_3}.
Section \S\ref{sec:abelian_nonabelian} is devoted to the Abelian/non-Abelian Correspondence for stacks: we state a conjecture in \S\ref{sec:setup_abelian} and present a sample computation in \S\ref{subsec:sample_abelian_nonabelian}.
The results of our calculations are collected in Section~\ref{sec:results}, and are summarized in Table~\ref{our_list} on page~\pageref{our_list}.

\subsection*{Conventions} We work over the field $\CC$ of complex numbers. 
A \emph{Fano variety} is a projective normal variety over $\CC$ such that the anticanonical divisor $-K_X$ is $\QQ$-Cartier and ample.  A \emph{del~Pezzo surface} is a Fano variety of dimension $2$. The \emph{Fano index} of a Fano variety $X$ is the largest positive integer $f$ such that the equality $-K_X = fH$ holds in the divisor class group $\mathrm{Cl}(X)$, for some Weil divisor $H$ on $X$.

Calligraphic letters, i.e. $\cX$ and $\cY$, denote separated Deligne--Mumford stacks of finite type over $\CC$ and roman letters, i.e. $X$ and $Y$, denote their coarse moduli spaces.

\subsection*{Acknowledgements} This project started during the PRAGMATIC 2013 Research School ``Topics in Higher Dimensional Algebraic Geometry'' held in Catania, Italy, in September 2013. We thank the organisers Alfio Ragusa, Francesco Russo, and Giuseppe Zappal\`a. We are very grateful to Alessio Corti for introducing us to this subject and supporting us during the last two years and to Tom Coates for countless invaluable comments and suggestions. We thank  Alexander Kasprzyk and Andrew Strangeway for many useful conversations. Our computations rely heavily on the use of the computer algebra software Magma \cite{magma}; we thank John Cannon and the Magma team at the University of Sydney for providing licences.

\addtocontents{toc}{\protect\setcounter{tocdepth}{2}} 

\section{Background} \label{sec:background}

\subsection{Quantum periods and del~Pezzo surfaces}

Let $X$ be a Fano variety with quotient singularities. The \emph{quantum period} of $X$ is a generating function $G_X$ for some genus-zero Gromov--Witten invariants of the unique well-formed orbifold $\cX$ having $X$ as coarse moduli space (see \S \ref{sec:quantum_period_Fano_orbifold}).
Let $\juniorchenruan(\cX)$ denote the subspace of the Chen--Ruan cohomology of $\cX$ generated by the identity classes of the twisted sectors with age $<1$. 
The quantum period is a family of power series parameterised by $\juniorchenruan(\cX)$, i.e.
\begin{equation*}
G_X \colon \juniorchenruan(\cX) \longrightarrow \QQ [ \! [ t ] \!].
\end{equation*}
For example, if $X$ has canonical singularities, then $\juniorchenruan(\cX) = \{ 0 \}$ and the quantum period $G_X$ is just a power series in $t$. We refer the reader to \S\ref{sec:quantum period} for the precise definitions of Chen--Ruan cohomology and quantum period.

If the quantum period of a Fano variety $X$ is
\begin{equation*}
G_X = \sum_{\delta \in \NN} c_\delta t^\delta
\end{equation*}
with $c_\delta \colon \juniorchenruan(\cX) \to \QQ$, then we define the \emph{regularised quantum period} of $X$ to be the power series 
\begin{equation*}
\widehat{G}_X = \sum_{\delta \in \NN} \delta ! c_\delta t^\delta.
\end{equation*}
The regularised quantum period plays an important r{\^o}le in Mirror Symmetry of Fano varieties \cite{mirror_symmetry_and_fano_manifolds, quantum_periods_3folds, conjectures}.

\bigskip

A \emph{del~Pezzo surface with isolated quotient singularities of type $\frac{1}{3}(1,1)$} (or, more briefly, a del~Pezzo surface with $\frac{1}{3}(1,1)$ singularities or points) is a normal projective surface $X$ over $\CC$ such that:
\begin{itemize}
\item the anticanonical divisor $-K_X$ is $\QQ$-Cartier and ample; and \item the complement of the smooth locus consists of finitely many points such that each of them has an analytic neighborhood isomorphic to an analytic neighborhood of the origin in the quotient $\CC^2 / \bmu_3$, where the cyclic group $\bmu_3$ acts on $\CC^2$ with weights $(1,1)$, i.e. the third root of unity $\zeta_3$ acts by mapping the point $(x,y) \in \CC^2$ into the point $(\zeta_3 x, \zeta_3 y)$.
\end{itemize}
Del~Pezzo surfaces with $\frac{1}{3}(1,1)$ points have recently been classified in \cite{fujita, alessio_liana}. There are 29 deformation families of such surfaces. Three of them have Fano index greater than 
1: the weighted projective plane $\PP(1,1,3)$ and two surfaces denoted by $B_{1,16/3}$ and $B_{2,8/3}$. The remaining 26 families have Fano index equal to 1 and they are denoted by $X_{k,d}$, 
where $k$ is the number of singular points and $d = K_X^2$ is the degree. For many of these surfaces, Corti and Heuberger exhibit explicit models, which are essential for our computations of the quantum periods.

In this paper we compute the quantum periods for 26 out of the 29 families with the following methods.
\begin{itemize}[leftmargin=0.5cm]
\item $6$ surfaces are toric. Using the mirror theorem for toric stacks (see \S\ref{sec:mirror_theorem}) we compute the full quantum periods, except on $X_{6,2}$ for which the large Picard rank increases the computational complexity and allows us to compute a specialisation of the quantum period only. An example is given in \S\ref{subsec:example_1}.
\item $19$ surfaces are complete intersections in toric orbifolds. Using a conjectural generalisation of the Quantum Lefschetz theorem (see \S\ref{sec:quantum_Lef_principle}), we compute the restriction of the quantum period to a non-empty affine subspace of $\juniorchenruan(\cX)$. In \S\ref{subsec:example_2} and \S\ref{subsec:example_3} we give two examples of these computations.
\item The surface $X_{1,7/3}$ is described as a complete intersection inside a weighted Grassmannian. In \S\ref{subsec:sample_abelian_nonabelian}, combining conjectural generalisations of the Quantum Lefschetz theorem (see \S \ref{sec:quantum_Lef_principle}) and the Abelian/non-Abelian Correspondence (see \S \ref{sec:setup_abelian}),  we compute a restriction of the quantum period to a non-empty affine subspace of $\juniorchenruan(\cX)$.
\item For the surfaces $X_{5, 2/3}$, $X_{5, 5/3}$, $X_{6,1}$, since we do not know any useful model for computations in Gromov--Witten theory, we have not been able to compute any restriction of the quantum period.
\end{itemize}
Although our computations rest on conjectural generalisations of the Quantum Lefschetz theorem and of the Abelian/non-Abelian Correspondence, we are confident that the results of our computations are correct because, even though partial, they match perfectly with the framework of Mirror Symmetry for orbifold del~Pezzo surfaces, as formulated in \cite{conjectures}. Our complete results are reported in \S\ref{sec:results}.

\subsection{Mirror Symmetry for orbifold del~Pezzo surfaces}

In \cite{conjectures} the authors state a conjecture which, roughly speaking, predicts that the regularised quantum period of an orbifold del~Pezzo surface $X$ with $\QQ$-Gorenstein rigid cyclic singularities coincides with the classical period of a certain family of `special' Laurent polynomials supported on the polygon corresponding to some toric degeneration of $X$. Since $\frac{1}{3}(1,1)$ is the simplest non-trivial example of $\QQ$-Gorenstein rigid surface singularity, our computations provide some evidence for this Mirror Symmetry type conjecture, which we now explain.

A \emph{Fano polygon} in a rank-$2$ lattice $N$ is a convex polytope $P \subseteq N_\RR$ such that the origin is in the strict topological interior of $P$ and the vertices of $P$ are primitive lattice vectors. From a Fano polygon $P$ one may construct a toric del~Pezzo surface $X_P$ from the face fan of $P$. On the set of Fano polygons there is an equivalence relation called \emph{mutation} \cite{mutations}. In \cite[Conjecture A]{conjectures} it is conjectured  that mutation equivalence classes of Fano polygons are in a one-to-one correspondence with del~Pezzo surfaces that admit a $\QQ$-Gorenstein degeneration\footnote{We refer the reader to \cite{kollar_shepherd_barron, hacking_compact, conjectures} for the notion of $\QQ$-Gorenstein deformation of surfaces.} to a toric surface and have isolated cyclic quotient singularities that are rigid with respect to $\QQ$-Gorenstein deformations. To a Fano polygon $P$ this correspondence associates a generic $\QQ$-Gorenstein deformation of the toric surface $X_P$.

 The \emph{maximally mutable Laurent polynomials} (see \cite{al_ketil}) of a Fano polygon $P$  are the Laurent polynomials $f \in \QQ [ N ]$ such that the Newton polygon of $f$ is $P$ and they stay Laurent after every mutation of $P$ and the corresponding operation on $f$. Let $\maximallymutable(P)$ be the affine space of maximally mutable Laurent polynomials of $P$ and let $\Tmaximallymutable(P) \subseteq \maximallymutable(P)$ be the affine subspace made up of those with T-binomial edge coefficients (see \cite{al_ketil, conjectures} for definitions).

If $f \in \QQ[ x^{\pm 1}, y^{\pm 1}]$ is a Laurent polynomial, then the \emph{classical period} of $f$ is the power series
\begin{align*}
\pi_f (t) &= \left( \frac{1}{2\pi {\rmi}} \right)^2 \int_{|x | = | y | =1}\frac{1}{1-tf(x,y)} \frac{\rmd x}{x} \wedge \frac{\rmd y}{y}  \\ 
&= \sum_{\delta \in \NN} \mathrm{coeff}_1 (f^\delta) t^\delta \in \QQ [ \! [ t ] \! ],
\end{align*}
where $\mathrm{coeff}_1(f^\delta)$ is the coefficient of $1$ in the Laurent polynomial $f^\delta$.

Now we will consider the following setup:

\begin{itemize}
\item[($\star$)] $P$ is a Fano polygon; $X_P$ is the toric del~Pezzo surface corresponding to the face-fan of $P$; $X$ is a generic $\QQ$-Gorenstein deformation of $X_P$; $\cX$ is the unique well-formed orbifold (see \S \ref{sec:quantum_period_Fano_orbifold} for details) such that its coarse moduli space is $X$.
\end{itemize}

\begin{conjecture}[Conjecture~B in \cite{conjectures}] \label{conj:conj_B} 
Let $P, X, \cX$ be as in $(\star)$.
Then there exists an affine-linear isomorphism $\Phi \colon \Tmaximallymutable (P) \longrightarrow \juniorchenruan(\cX)$ such that
\begin{equation*}
\forall f \in \Tmaximallymutable(P), \qquad \widehat{G}_X ( \Phi(f); t) = \pi_f(t),
\end{equation*}
where $\widehat{G}_X$ is the regularised quantum period of $X$ and $\pi_f$ is the classic period of $f$.
\end{conjecture}

Corti and Heuberger \cite{alessio_liana} have proved that, out of the 29 del~Pezzo surfaces with $\frac{1}{3}(1,1)$ points, only 26 surfaces admit a $\QQ$-Gorenstein degeneration to a toric surface. Indeed, the surfaces $X_{4, 1/3}$, $X_{5, 2/3}$ and $X_{6,1}$ do not have any $\QQ$-Gorenstein degeneration to a toric surface.

The Fano polygons $P$ such that the corresponding surface $X$, according to ($\star$), has only $\frac{1}{3}(1,1)$ points have been classified up to mutation by Kasprzyk, Nill and Prince \cite{minimal_polygons}. There are 26 mutation equivalence classes of such polygons and they correspond to the del Pezzo surfaces mentioned above.
The spaces $\Tmaximallymutable(P)$, for such polygons $P$, have been computed by Kasprzyk and Tveiten \cite{al_ketil}.

Combining these results with our calculations in Section \S\ref{sec:results} yields:

\begin{theorem} \label{thm:main_result} 
Let $P$, $X$ and $\cX$ satisfy $(\star)$. Suppose that $X$ has only $\frac{1}{3}(1,1)$ singularities and is not $X_{5,5/3}$.

If natural generalisations of the Quantum Lefschetz theorem (Conjecture \ref{conj:tom_quantum_lefschetz}) and of the Abelian/non-Abelian Correspondence (Conjecture \ref{conj:abelian}) hold, then there exist a non-empty affine subspace $W \subseteq \Tmaximallymutable (P)$ and an injective affine-linear map $\Phi \colon W \longrightarrow \juniorchenruan(\cX)$ such that 
\begin{equation*}
\forall f \in W, \qquad \widehat{G}_X ( \Phi(f); t) = \pi_f(t).
\end{equation*}
\end{theorem}


\section{The quantum period}\label{sec:quantum period}

\subsection{Gromov--Witten theory for smooth proper Deligne--Mumford stacks} \label{sec:GW_for_stacks}

Gromov--Witten theory for smooth proper Deligne--Mumford has been developed by Chen--Ruan \cite{chen_ruan_cohomology} in the symplectic setting and by Abramovich--Graber--Vistoli \cites{algebraic_orbifold_quantum_products, gromov_witten_for_DM_stacks} in the algebraic setting. Here we recall just the basic definitions, following the concise expositions of \cite{iritani_periods}*{\S 2.1} and \cite{mirror_theorem}.

Let $\cX$ be a proper smooth Deligne--Mumford stack over $\CC$ and $X$ be its coarse moduli space, which is assumed to be projective. Let $\rI \cX$ be the inertia stack, which is the fibre product $\cX \times_{\Delta, \cX \times \cX, \Delta} \cX$ of the diagonal morphisms $\Delta \colon \cX \to \cX \times \cX$. For every $\CC$-scheme $S$, an $S$-valued point of $\rI \cX$ is a pair $(x,g)$, where $x$ is a $S$-valued point of $\cX$ and $g \in \mathrm{Aut}_{\cX(S)}(x)$ is a stabilizer of $x$. Let $\mathrm{Box}(\cX)$ be the set of the connected components of $\rI \cX$ and let
\begin{equation*}
\rI \cX = \coprod_{b \in \mathrm{Box}(\cX)} \cX_b
\end{equation*}
be the decomposition of $\rI \cX$ into connected components. In $\mathrm{Box}(\cX)$ there is a special element $0 \in \mathrm{Box}(\cX)$ which corresponds to the trivial stabilizer $g=1$ and whose associated connected component $\cX_0$ of $\rI \cX$ is just $\cX$. For every  $b \in \mathrm{Box}(\cX)$, let $\mathrm{age}(b) \in \QQ_{\geq 0}$ be the \emph{age} of the component $\cX_b$ (see \cite[\S 3.2]{chen_ruan_cohomology}, where it is called degree shifting number, or \cite[\S 7.1]{gromov_witten_for_DM_stacks}). Let $\rH^\bullet_\mathrm{CR}(\cX)$ be  the even part of the \emph{Chen--Ruan orbifold cohomology group} of $\cX$, i.e. the $\QQ$-graded vector space over $\QQ$ given by
\begin{equation}
\rH^p_\mathrm{CR}(\cX) := \bigoplus\limits_{\substack{b \in \mathrm{Box}(\cX) \text{ s.t.} \\ p - 2 \mathrm{age}(b) \in 2 \ZZ}} \rH^{p - 2 \mathrm{age}(b)} (\cX_b; \QQ)
\end{equation}
for every $p \in \QQ$. In all spaces considered below there will be only even cohomology classes. We see that $\rH^\bullet_\mathrm{CR}(\cX)$ coincides, as a vector space, with the even degree cohomology  $\rH^\mathrm{even}(\rI \cX;\QQ)$ of the inertia stack, but the gradings on these two vector spaces are different.

We have an involution $\mathrm{inv} \colon \rI \cX \to \rI \cX$ given by $(x,g) \mapsto (x,g \inv)$. This induces an involution $\mathrm{inv}^\star \colon \rH^\bullet_\mathrm{CR}(\cX) \to  \rH^\bullet_\mathrm{CR}(\cX)$. The \emph{orbifold Poincar\'e pairing} $(\cdot,\cdot)_\mathrm{CR}$ is the bilinear form on $\rH^\bullet_\mathrm{CR}(\cX)$ defined by 
\begin{equation}
(\al, \beta)_\mathrm{CR} := \int_{\rI \cX} \al \cup \mathrm{inv}^\star \beta.
\end{equation}
It is symmetric and non-degenerate, because $\rI \cX$ is smooth.

For $d \in \rH_2(X;\ZZ)$ and $n \geq 0$, let $\cX_{0,n,d}$ be the \emph{moduli stack of stable maps}\footnote{This is the same as the stack of twisted stable maps $\cK_{0,n}(\cX,d)$ in \cite{abramovich_vistoli_stable_maps}.} to $\cX$ of genus $0$, with $n$ marked points and degree $d$.
This is equipped with a virtual fundamental class $[\cX_{0,n,d}]^\mathrm{vir} \in \rH_\bullet (\cX_{0,n,d}; \QQ)$ and evaluation maps
\begin{equation}
\rme \rmv_i \colon \cX_{0,n,d} \to \rI^\mathrm{rig} \cX
\end{equation}
to the rigidified inertia stack $\rI^\mathrm{rig} \cX$ (see \cite[\S 3.4]{gromov_witten_for_DM_stacks}), for $i=1,\dots,n$. Since the stacks $\rI^\mathrm{rig} \cX$ and $\rI \cX$ have the same coarse moduli space, there are canonical isomorphisms between their cohomology groups with rational coefficients. Thus, we can think of elements of $\rH^\bullet_\mathrm{CR}(\cX)$ as cohomology classes on $\rI^\mathrm{rig} \cX$. For $i=1,\dots, n$, let $\psi_i \in \rH^2(\cX_{0,n,d};\QQ)$ be the first Chern class of the $i$th universal cotangent line bundle $\cL_i \in \mathrm{Pic}(\cX_{0,n,d})$; the fibre of $\cL_i$ at a stable map $f \colon (\cC; x_1, \dots, x_n) \to \cX$ is the cotangent space $\rT^*_{x_i} C$ at the $i$th marked point of the coarse curve $C$ of $\cC$. \emph{Gromov--Witten invariants} of $\cX$ are
\begin{equation*}
\left\langle \al_1 \psi^{k_1}, \dots, \al_n \psi^{k_n} \right\rangle_{0,n,d} := \int_{[\cX_{0,n,d}]^\mathrm{vir}} \prod_{i=1}^n \left( \mathrm{ev}^\star_i(\al_i) \cup \psi_i^{k_i} \right),
\end{equation*}
where $\al_1, \dots, \al_n \in \rH^\bullet_\mathrm{CR}(\cX)$ and  $k_1, \dots, k_n$ are non-negative integers.
Roughly speaking, if $k_1 = \cdots = k_n = 0$, this is the `virtual number' of possibly-nodal $n$-pointed orbifold curves in $\cX$ of genus $0$ and degree $d$ which are incident at the $i$th marked point, $1 \leq i \leq n$, to a chosen generic cycle Poincar\'e-dual to $\al_i$ and which have isotropy at the $i$th marked point specified by $\al_i$. If any of the $k_i$ are non-zero then we count only curves which in addition satisfy certain constraints on their complex structure.

\subsection{Givental's symplectic formalism} \label{sec:Givental_formalism}

Let $\cX$ be a proper smooth Deligne--Mumford stack over $\CC$ with projective coarse moduli space $X$.
Let $\mathrm{Eff}(\cX) \subseteq \rH_2(X;\ZZ)$ denote the submonoid generated by the homology classes in $X$ represented by images of representable maps from complete stacky curves to $\cX$. If $R$ is a commutative ring, then the \emph{Novikov ring} $\boldsymbol{\Lambda}(R)$ on $R$ is the completion of the group $R$-algebra $R[\mathrm{Eff}(\cX)]$ with respect to an additive valuation defined by a polarization on $X$ which we choose once and for all (see \cite[Definition 2.5.4]{tseng_orbifold_quantum_riemann_roch}). If $d \in \mathrm{Eff}(\cX)$ we denote by $Q^d$ the corresponding element in the Novikov ring $\boldsymbol{\Lambda}(R)$.

Following Givental \cite{givental_lagrangian_cone} and Tseng \cite{tseng_orbifold_quantum_riemann_roch}, we consider the infinite dimensional $\CC$-vector space
\begin{equation} \label{eq:symplectic_vector_space_H}
\cH_\cX := \rH^\bullet_\mathrm{CR}(\cX) \otimes_\QQ \boldsymbol{\Lambda} \left( \CC ( \! ( z \inv ) \!) \right),
\end{equation}
where $z$ is a formal variable and $\CC ( \! ( z \inv ) \!)$ is the fraction field $\CC [ \! [ z \inv, z ]$ of the power series ring $\CC [ \! [ z \inv ] \! ]$, equipped with the symplectic form
\begin{equation*}
\Omega(f,g) = - \mathrm{Res}_{z = \infty} \big( f(-z), g(z) \big)_\mathrm{CR} \rmd z \qquad \qquad \text{for } f,g \in \cH_\cX.
\end{equation*}
In the symplectic vector space $(\cH_\cX, \Omega)$ there is a Lagrangian submanifold $\cL_\cX$, which is a formal germ of a cone with vertex at the origin and which encodes all genus-zero Gromov--Witten invariants of $\cX$. We will not give a precise definition of $\cL_\cX$ here, referring the reader to \cite[\S 3.1]{tseng_orbifold_quantum_riemann_roch}, \cite[Appendix B]{computing_twisted} and \cite[\S 2]{mirror_theorem}. $\cL_\cX$ is called the \emph{Givental cone} of $\cX$ and determines and is determined by Givental's \emph{J-function}:
\begin{equation} \label{eq:J-function}
J_\cX (\gamma,z) = z + \gamma + \sum_{d \in \mathrm{Eff}(\cX)} \sum_{n=0}^\infty \sum_{k=0}^\infty \sum_{\epsilon = 1}^N \frac{Q^d}{n!}
\left\langle \gamma, \dots, \gamma, \phi^\epsilon \psi^k \right\rangle_{0, n+1, d} \phi_\epsilon z^{-k-1},
\end{equation}
where:
\begin{itemize}
\item $\gamma$ runs in the even part $\rH^\bullet_{\rC \rR} (\cX)$ of the Chen--Ruan orbifold cohomology  of $\cX$;
\item $\{ \phi_1, \dots \phi_N \}$ and $\{ \phi^1, \dots \phi^N \}$ are homogeneous bases of the $\QQ$-vector space $\rH^\bullet_\mathrm{CR}(\cX)$ which are dual with respect to the orbifold Poincar\'e pairing $( \cdot, \cdot )_\mathrm{CR}$.
\end{itemize}
The cone $\cL_\cX$ determines the J-function because $J_\cX(\gamma,-z)$ is the unique point on $\cL_\cX$ of the form $-z + \gamma + \rO(z \inv)$, where $\rO(z \inv)$ denotes a power series in $z \inv$. Conversely, the J-function determines $\cL_\cX$ and all genus-zero Gromov--Witten invariants of $\cX$ by \cite[Proposition 2.1]{gholampour_tseng_computations} and topological recursion relations.

\subsection{The quantum period of a Fano orbifold} \label{sec:quantum_period_Fano_orbifold}

An \emph{orbifold} is defined to be a separated smooth connected Deligne--Mumford stack $\cX$ of finite type over $\CC$ such that the stabiliser of the generic point is trivial. Following \cite[Definition 5.11]{working_weighted}, we say that an orbifold $\cX$ is \emph{well-formed}\footnote{A well-formed orbifold is called a canonical smooth Deligne--Mumford stack by Fantechi--Mann--Nironi \cite[Definition 4.4]{Fantechi_toric_stacks}.} if the natural morphism $\cX \to X$ to the coarse moduli space is an isomorphism in codimension $1$. In other words, an orbifold is well-formed if the stacky locus has codimension at least $2$.

If $\cX$ is a well-formed orbifold and its coarse moduli space $X$ is a scheme, then $X$ is a Cohen--Macaulay $\QQ$-factorial normal variety with quotient singularities such that $\mathrm{Pic}(\cX) \simeq \mathrm{Pic}(X_\mathrm{sm}) \simeq \mathrm{Cl}(X)$, where $X_\mathrm{sm}$ is the smooth locus of $X$ and $\mathrm{Cl}(X)$ is the divisor class group of $X$. Conversely, a normal separated variety with quotient singularities is the coarse moduli space of a unique well-formed orbifold, by \cite[(2.8) and (2.9)]{vistoli_intersection} and \cite[\S 4.1]{Fantechi_toric_stacks}. In other words, there is a one-to-one correspondence between well-formed orbifolds with schematic coarse moduli space and normal separated varieties with quotient singularities.

When $X$ is a normal separated variety with quotient singularities, we denote by $\cX$ the unique well-formed orbifold such that $X$ is its coarse moduli space.

\begin{definition}
 A well-formed orbifold $\cX$ is called a \emph{Fano orbifold} if its coarse moduli space $X$ is a projective variety such that its anticanonical class $- K_X$ is an ample $\QQ$-Cartier divisor.
 \end{definition}
 
\noindent There is a one-to-one correspondence between Fano orbifolds and normal projective varieties with quotient singularities such that $\cO_X(-mK_X)$ is a very ample line bundle on $X$, for some $m \geq 1$.

 \medskip
 
The quantum period of a Fano orbifold $\cX$ is a generating function of certain genus-zero Gromov--Witten invariants of $\cX$.

\begin{definition} \label{def:quantum_period}
Let $\cX$ be a Fano orbifold and let $b_1, \dots, b_r \in \mathrm{Box}(\cX)$ be the indices of the connected components of the inertia stack $\rI \cX$ such that $0 < \mathrm{age}(b_i) < 1$. For $i=1, \dots, r$, let
\begin{equation*}
\mathbf{1}_{b_i} \in \rH^0(\cX_{b_i};\QQ) \subseteq \rH^{2 \mathrm{age}(b_i)}_\mathrm{CR}(\cX)
\end{equation*}
be the identity cohomology class of the component $\cX_{b_i}$. If $d \in \mathrm{Eff}(\cX)$, $n \in \NN$ and $1 \leq i_1, \dots, i_n \leq r$, then set 
\begin{align*}
\delta_{d,i_1,\dots, i_n} &:= -K_X \cdot d + \sum_{j=1}^n \left( 1 - \mathrm{age} \left(b_{i_j} \right) \right) \in \QQ.
\end{align*}
The \emph{quantum period} of $\cX$ is:
\begin{align*}
 G_\cX & (x_1, \dots, x_r ; t ) = \ 1 \ + \\
&+ \sum_{d \in \mathrm{Eff}(\cX)} \sum_{n =0}^\infty \sum_{1 \leq i_1, \dots, i_n \leq r}  \left\langle \mathbf{1}_{b_{i_1}}, \dots, \mathbf{1}_{b_{i_n}}, \frac{ \phi_{\rmv \rmo \rml}}{1 - \psi} \right\rangle_{0,n+1,d} \frac{x_{i_1} \cdots x_{i_n}}{n!} t^{\delta_{d, i_1, \dots, i_n}},
\end{align*}
where $\phi_\mathrm{vol} \in \rH^{2 \dim \cX}(\cX; \QQ)$ is the cohomology class of a point, $t,x_1,\dots, x_r$ are formal variables and  $\frac{1}{1-\psi}$ denotes the series $\sum_{k \geq 0} \psi^k$.
\end{definition}

The quantum period comes from a specialisation of a component of the J-function. Indeed, $G_\cX$ is obtained from the component of the J-function $J_\cX$ along the unit class $\mathbf{1}_0 \in \rH^0(X;\QQ) \subseteq \rH^0_\mathrm{CR}(\cX)$ by applying the following substitutions:
\begin{itemize}
\item replacing the Novikov variable $Q^d$  by $t^{-K_X \cdot d}$,
\item setting $ z = 1$,
\item setting $
\gamma = t^{1 - \mathrm{age}(b_1)} x_1 \mathbf{1}_{b_1} + \cdots + t^{1 - \mathrm{age} (b_r)} x_r \mathbf{1}_{b_r}$.
\end{itemize}

\begin{notation}
If $d \in \mathrm{Eff}(\cX)$, $n \in \NN$ and $1 \leq i_1, \dots, i_n \leq r$, then set
\begin{align*}
\rG \rW_{d,i_1,\dots, i_n} &:= \left\langle \mathbf{1}_{b_{i_1}}, \dots, \mathbf{1}_{b_{i_n}}, \phi_{\rmv \rmo \rml} \psi^{\delta_{d, i_1, \dots, i_n} - 2} \right\rangle_{0,n+1,d} \in \QQ.
\end{align*} 
\end{notation}

\begin{proposition} \label{prop:quantum_period}
If $\cX$ is a Fano orbifold, then $G_\cX \in \QQ[x_1, \dots, x_r] [ \! [ t ] \!]$ and the following formula holds:
\begin{equation} \label{eq:quantum_period_accorciato}
 G_\cX(x_1, \dots, x_r;t) = 1 + \sum\limits_{\substack{d \in \mathrm{Eff}(\cX), \\ n \in \NN, \\ 1 \leq i_1, \dots, i_n \leq r}} \frac{\rG \rW_{d,i_1,\dots,i_n}}{n!} x_{i_1} \cdots x_{i_n} t^{\delta_{d,i_1,\dots,i_n}}.
\end{equation}
Moreover:
\begin{itemize}
\item[(i)] the coefficient of $t$ in $G_\cX$ is zero;
\item[(ii)] if $f$ is the Fano index of $X$, then in the specialisation
$
G_\cX (0, \dots, 0;t)$
only powers of $t^f$ appear, i.e. $G_\cX(0,\dots, 0;t) \in \QQ[ \! [ t^f ] \!]$.
\end{itemize}

\end{proposition}

\begin{proof}
Notice
\begin{equation*}
\left\langle \mathbf{1}_{b_{i_1}}, \dots, \mathbf{1}_{b_{i_n}}, \frac{ \phi_{\rmv \rmo \rml}}{1 - \psi} \right\rangle_{0,n+1,d} = \sum_{k \in \NN} \left\langle \mathbf{1}_{b_{i_1}}, \dots, \mathbf{1}_{b_{i_n}}, \phi_{\rmv \rmo \rml} \psi^k \right\rangle_{0,n+1,d}.
\end{equation*}
If the Gromov--Witten invariant $ \left\langle \mathbf{1}_{b_{i_1}}, \dots, \mathbf{1}_{b_{i_n}}, \phi_{\rmv \rmo \rml} \psi^k \right\rangle_{0,n+1,d}$ is non-zero, then $\deg (\phi_\mathrm{vol} \psi_{n+1}^k) = 2 \dim \cX + 2k$
must be equal to the real virtual dimension of the corresponding component of $\cX_{0, n+1, d}$, which is
\begin{equation*}
2 \left[ -K_\cX \cdot d + \dim \cX - \mathrm{age}(b_{i_1}) - \cdots - \mathrm{age}(b_{i_n}) + (n+1) - 3 \right].
\end{equation*}
Thus
\begin{align*}
k &= -K_\cX \cdot d - \mathrm{age}(b_{i_1}) - \cdots - \mathrm{age}(b_{i_n}) + n - 2 \\
&= \delta_{d, i_1, \dots, i_n} - 2
\end{align*}
is uniquely determined by $d$, $n$ and $i_1, \dots, i_n \in \{ 1, \dots, r \}$. This shows that the formula \eqref{eq:quantum_period_accorciato} holds.

We prove that $\delta_{d, i_1, \dots, i_n}$ is an integer greater than $1$ whenever there exist $d \in \rH_2(X,\ZZ)$, $n \in \NN$ and $1 \leq i_1, \dots i_n \leq r$ such that $\delta = \delta_{d, i_1, \dots, i_n}$ and $\mathrm{GW}_{d, i_1, \dots, i_n} \neq 0$. In these circumstances there must exist a genus-zero $(n+1)$-pointed stable curve
\begin{equation*}
\xymatrix{
\cC \ar[d] \ar[r]^\phiv & \cX \ar[d]^\pi \\
C \ar[r]^{\overline{\phiv}} & X
}
\end{equation*}
such that $\overline{\phiv}_\star [C] = d$ and the marking gerbes $\Sigma_1, \dots, \Sigma_n, \Sigma_{n+1} \subseteq \cC$ give geometric points in the components $b_{i_1}, \dots, b_{i_n}, 0 \in \mathrm{Box}(\cX)$ of $\rI \cX$, respectively. By orbifold Riemann--Roch \cite[Theorem 7.2.1]{gromov_witten_for_DM_stacks}, we see that
\begin{align*}
\delta_{d, i_1, \dots, i_n} &= \deg_\cC \phiv^\star \rT_\cX - \mathrm{age}(b_{i_1}) - \cdots - \mathrm{age}(b_{i_n}) + n \\
&= \chi (\cC, \phiv^\star \rT_\cX) - \mathrm{rk}(\phiv^\star \rT_\cX) \chi(\cC, \cO_\cC) + n
\end{align*}
is an integer. Moreover, since Gromov-Witten invariants with negative gravitational descendants are zero by definition, $\delta_{d, i_1, \dots, i_n} \geq 2$. This proves (i).

Now we have to prove the finiteness of the sum \eqref{eq:quantum_period_accorciato}. More specifically we have to prove that, for every integer $\delta \geq 2$, the coefficient
\begin{equation} \label{eq:coefficient_of_t^delta}
\sum\limits_{\substack{d \in \mathrm{Eff}(\cX), \\ n \in \NN, \\ 1 \leq i_1, \dots, i_n \leq r, \\ \text{s.t. } \delta_{d, i_1, \dots, i_n} = \delta}} \frac{\rG \rW_{d,i_1,\dots,i_n}}{n!} x_{i_1} \cdots x_{i_n} 
\end{equation}
of $t^\delta$ is a polynomial in the variables $x_1, \dots, x_r$. This is a dimensional argument, as follows.

For each geometric point $p \colon \Spec \CC \to \cX$, denote by $e_p$ the exponent of the automorphism group of $p$. Call $e$ the least common multiple of the $e_p$ for all geometric points of $\cX$. By \cite[Lemma 2.1.2]{gromov_witten_for_DM_stacks}, the line bundle $(\det (\Omega^1_\cX)^\vee)^{\otimes e}$ on $\cX$ is a pull-back to $\cX$ of a line bundle $H$ on $X$. Since $\cX$ is a Fano orbifold, $H$ is an ample line bundle on $X$. In the divisor class group of $X$ we have the equality $H = - e K_X$. Following \cite{abramovich_vistoli_stable_maps}, for every $h,n \in \NN$ we denote by $\cK_{0,n}(\cX,h)$ the moduli stack of genus-zero $n$-marked stable maps $\phiv \colon \cC \to \cX$ such that $\deg_C \bar{\phiv}^\star H = h$, where $C$ is the coarse moduli space of $\cC$ and $\bar{\phiv} \colon C \to X$ is the morphism of schemes induced by $\phiv$.

Fix an integer $\delta \geq 2$. Let $a = \max_{1 \leq i \leq r} \mathrm{age}(b_i)$. If $d \in \mathrm{Eff}(\cX)$, $n \in \NN$ and $i_1, \dots, i_n \in \{ 1, \dots, r \}$ are such that $\delta_{d, i_1, \dots, i_n} = \delta$, then $\delta \geq n (1-a)$ and $\delta \geq -K_X \cdot d $, so $n \leq \delta / (1-a)$ and $H \cdot d \leq e \delta$. Therefore, the coefficient \eqref{eq:coefficient_of_t^delta}
of $t^\delta$ involves some intersection products on various connected components of the proper stack
\begin{equation*}
\coprod\limits_{\substack{n \leq \delta / (1-a), \\ h \leq e \delta}} \cK_{0,n} (\cX, h).
\end{equation*}
This shows that the sum \eqref{eq:coefficient_of_t^delta} is a polynomial in the variables $x_1, \dots, x_r$ with rational coefficients.

Now we prove assertion (ii). Let $\cL \in \Pic(\cX)$ be such that $\omega_\cX^\vee = \cL^{\otimes f}$. If $\left\langle \phi_\mathrm{vol} \psi^{-K_X \cdot d - 2} \right\rangle_{0,1,d} \neq 0$, then there exists a $1$-pointed stable curve $\phiv \colon \cC \to \cX$ such that $\bar{\phiv}_\star [C] = d$ and the marking gerbe $\Sigma_1 \subseteq \cC$ gives a geometric point in the trivial component of $\rI \cX$. Applying orbifold Riemann--Roch for $-K_X = \omega_\cX^\vee$ and $\cL$ gives that the anticanonical degree
\begin{align*}
-K_X \cdot d &= \deg_\cC \phiv^\star \omega_\cX^\vee \\
&= f \cdot \deg_\cC \phiv^\star \cL \\
&= f \cdot \left( \chi(\cC, \phiv^\star \cL) - \chi(\cC, \cO_\cC) \right)
\end{align*}
is divided by $f$. This concludes the proof of (ii).
\end{proof}

\begin{remark}
Let $\cX$ be a Fano orbifold and let $\mathbf{1}_{b_1}, \dots, \mathbf{1}_{b_r} \in \rH_\mathrm{CR}^\bullet(\cX)$ be the identity cohomology classes of the components of $\rI \cX$ with age between $0$ and $1$. Let $e$ be the least common multiple of the exponents of the automorphism groups of all geometric points of $\cX$. Let $\phi_\mathrm{vol}$ be the cohomology class of a point. For every $n \in \NN$ and $d \in \mathrm{Eff}(\cX)$, consider the $\QQ[ t^{1/e}]$-valued multilinear symmetric $n$-form on $\juniorchenruan(\cX)$ defined by
\begin{equation*}
(\mathbf{1}_{b_{i_1}}, \dots, \mathbf{1}_{b_{i_n}} ) \mapsto \left\langle \mathbf{1}_{b_{i_1}}, \dots, \mathbf{1}_{b_{i_n}}, \frac{ \phi_{\rmv \rmo \rml}}{1 - \psi} \right\rangle_{0,n+1,d} \prod_{j=1}^n t^{1 - \mathrm{age}(b_{i_j})}
\end{equation*}
for any $1 \leq i_1, \dots, i_n \leq r$. It induces an element $\Xi_{d,n} \in \mathrm{Sym}^n \juniorchenruan(\cX)^\vee \otimes_\QQ \QQ[t^{1/e}]$. The quantum period of $\cX$ can be written as
\begin{equation*}
G_\cX (t) = 1 + \sum\limits_{\substack{d \in \mathrm{Eff}(\cX), \\ n \in \NN}} \frac{\Xi_{d,n}}{n!} t^{-K_X \cdot d}.
\end{equation*}
Proposition \ref{prop:quantum_period} shows that $G_\cX \in (\mathrm{Sym}^\bullet \juniorchenruan(\cX)^\vee) [ \! [ t ] \! ]$. In this way, we can consider the quantum period as a family of power series $G_\cX \colon \juniorchenruan(\cX) \to \QQ[ \! [ t ] \! ]$.
\end{remark}

\begin{example}
If $\cX$ is a Fano orbifold such that its coarse moduli space $X$ has canonical singularities, then there are no connected components of $\rI \cX$ with positive age smaller than $1$ by the Reid--Tai criterion \cite[Theorem 3.21]{kollar_singularities_mmp}. Therefore, in this case, $\juniorchenruan(\cX) = 0$ and the quantum period of $\cX$ is
\begin{equation*}
G_\cX (t) = 1 + \sum\limits_{\substack{d \in \mathrm{Eff}(\cX) \\ \text{s.t. } -K_X \cdot d \geq 2}}  \left\langle \phi_\mathrm{vol} \psi^{-K_X \cdot d - 2} \right\rangle_{0,1,d} t^{-K_X \cdot d}.
\end{equation*} 
In particular, if $X$ is a smooth Fano variety then the formula above for the quantum period of $X$ agrees with \cite[Definition 4.2]{mirror_symmetry_and_fano_manifolds} and \cite[\S B]{quantum_periods_3folds}.
\end{example}

\begin{example}
Let $X$ be a del Pezzo surface with isolated quotient singularities of type $\frac{1}{3}(1,1)$ and let $\cX$ be the Fano orbifold associated to $X$. Assume that $X$ has $r$ singular points. Then, the inertia stack $\rI \cX$ has $1 + 2r$ connected components:
\begin{itemize}
\item the trivial connected component of age $0$, which is isomorphic to $\cX$;
\item $r$ connected components of age $2/3$, which are isomorphic to $\rB \boldsymbol{\mu}_3$;
\item $r$ connected components of age $4/3$, which are isomorphic to $\rB \boldsymbol{\mu}_3$.
\end{itemize}
For $i=1, \dots, r$, let $\mathbf{1}_i$ be the unit cohomology class of the $i$th connected component of age $2/3$. Then the quantum period of $\cX$ is:
\begin{align*}
G_\cX &(x_1, \dots, x_r;t) = \\
&= 1 + \sum\limits_{\substack{d \in \mathrm{Eff}(\cX), \\ n \in \NN, \\
 1 \leq i_1, \dots, i_n \leq r}}
 \left\langle \mathbf{1}_{i_1}, \dots, \mathbf{1}_{i_n}, \phi_{\rmv \rmo \rml} \psi^{-K_X \cdot d + \frac{n}{3}-2} \right\rangle_{0,n+1,d} \frac{x_{i_1} \cdots x_{i_n}}{n!} t^{- K_X \cdot d +  \frac{n}{3} }.
\end{align*}
\end{example}

\section{Toric stacks} \label{sec:toric}

\subsection{Stacky fans and extended stacky fans}
\label{sec:stacky_fans}

Here we briefly recall the theory of toric stacks \cites{BCS_toric_stacks, Fantechi_toric_stacks} and the combinatorial machinery developed in \cite{mirror_theorem}, which will allow us to produce a point of the Givental cone for a toric stack. We will present the case of toric well-formed orbifolds only.

Let $X$ be a simplicial toric variety which is proper over $\CC$. In other words, according to \cite{fulton_toric_varieties}, $X$ comes from a finitely generated free abelian group $N$ of finite rank\footnote{The general theory of \cite{BCS_toric_stacks} allows $N$ to be a finitely generated abelian group. Nevertheless, since we consider well-formed orbifolds rather than more general Deligne--Mumford stacks, we consider only the case in which $N$ is torsion free.} and a complete \emph{simplicial} fan $\Sigma$  in $N_\RR$. Let $\rho \colon \ZZ^n \to N$ be the linear map which maps the $i$th standard basis element of $\ZZ^n$ to the primitive generator $\rho_i$ of the $i$th ray of the fan $\Sigma$. So $n$ is the number of rays of $\Sigma$.  Let $\LL$ be the kernel of $\rho$.
The exact sequence 
\begin{equation} \label{eq:fan_sequence_general}
0 \longrightarrow \LL \longrightarrow \ZZ^n \overset{\rho}\longrightarrow N
\end{equation}
is called the \emph{fan sequence}.  Set $M := \Hom_\ZZ (N, \ZZ)$. Since the cokernel of $\rho$ is finite,  the dual map $\rho^* \colon M \to \ZZ^{* n}$, which is obtained from $\rho$ by applying $\Hom_\ZZ (-, \ZZ)$, is injective. The cokernel of $\rho^*$ is denoted by $\LL^\vee$ and is called the \emph{Gale dual} of $\rho$. We get a short exact sequence, which is called the \emph{divisor sequence}:
\begin{equation} \label{eq:divisor_sequence_general}
0 \longrightarrow M \longrightarrow \ZZ^{*n} \overset{D}\longrightarrow \LL^\vee \longrightarrow 0.
\end{equation}
We see that $\LL^\vee$ is an extension of the dual $\LL^* = \Hom_\ZZ (\LL, \ZZ)$ by a finite group which is isomorphic to $\coker \rho$. In particular, if $\rho$ is surjective, then $\LL^\vee = \LL^*$. It is well known that in \eqref{eq:divisor_sequence_general} the group $\ZZ^{*n}$ is identified with the group of torus-invariant Weil divisors of $X$ and  the group $\LL^\vee$ is canonically isomorphic to the divisor class group $\mathrm{Cl}(X)$: the image $D_i \in \LL^\vee$ of the $i$th standard basis element of $\ZZ^{* n}$ is the class of the $i$th toric divisor of $X$. The anticanonical class of $X$ is given by $-K_X = D_1 + \cdots + D_n \in \LL^\vee$. We have $\rN^1(X) = \LL^\vee \otimes_\ZZ \RR$ and the nef cone of $X$ is
\begin{equation} \label{eq:nef_cone_general}
\Nef (X) = \bigcap_{\sigma \in \Sigma} \mathrm{cone} \left\langle D_i \mid i \notin \sigma \right\rangle \subseteq \LL^\vee \otimes_\ZZ \RR = \rN^1(X),
\end{equation}
where $\mathrm{cone} \langle T \rangle = \{ \sum_j a_j t_j \mid t_j \in T, a_j \geq 0 \}$ is the cone spanned by a subset $T$ of a real vector space. Moreover, $\rA_1(X)_\QQ = \rH_2^\mathrm{alg}(X;\QQ) =\rN_1(X)_\QQ \simeq \LL \otimes_\ZZ \QQ$. The bilinear form $\LL^\vee \times \LL \to \ZZ$, which is induced by the duality pairing of $\ZZ^n$, induces the pairing $\rN^1(X) \times \rN_1(X) \to \RR$  between numerical classes of divisors and numerical classes of curve cycles. The Mori cone $\mathrm{NE}(X)$ is the dual cone of $\Nef(X)$ in $\LL \otimes_\ZZ \RR$.

Applying the exact functor $\Hom_\ZZ (-, \CC^*)$ to \eqref{eq:divisor_sequence_general}, we get a homomorphism of algebraic groups from from $\GG := \Hom (\LL^\vee, \CC^*)$ to the torus $(\CC^*)^n$. Since $(\CC^*)^n$ acts diagonally on $\AA^n_\CC$, there is an induced action of $\GG$ on $\AA^n_\CC$. Let $x_1, \dots, x_n$ be the standard coordinates on $\AA^n_\CC$. Consider the ideal $\mathrm{Irr}_\Sigma$ of $\CC[x_1, \dots, x_n]$ generated by the monomials $\prod_{i \colon \rho_i \notin \sigma} x_i$ as $\sigma \in \Sigma$ and the quasi-affine variety $U_\Sigma = \AA^n_\CC \setminus \rV (\mathrm{Irr}_\Sigma)$. The quotient stack $\cX := [U_\Sigma / \GG ]$ is called the \emph{toric stack} associated to the triple $(N, \Sigma, \rho)$, which is called a \emph{stacky fan}. By \cites{BCS_toric_stacks, Fantechi_toric_stacks}, $\cX$ is a proper well-formed orbifold and its coarse moduli space is $X$.

By \cite[Proposition 4.7]{BCS_toric_stacks}, the connected components of the inertia stack $\rI \cX$ are indexed by the finite set
\begin{equation*}
\mathrm{Box}(\Sigma) := N \cap \bigcup_{\sigma \in \Sigma} \left\{ \left. \sum_{i \colon \rho_i \in \sigma} a_i \rho_i \right| 0 \leq a_i <1  \right\}.
\end{equation*} 
The element $b = \sum_{\rho_i \in \sigma} a_i \rho_i \in \mathrm{Box}(\Sigma)$, for some $\sigma \in \Sigma$ and $0 < a_i < 1$, corresponds to the subvariety of $X$ defined by the homogeneous equations $x_i = 0$ for $\rho_i \in \sigma$. Its age is $\sum a_i$.

\bigskip

Now we describe the formalism of extended stacky fans according to \cite{jiang}. We choose a finite set $S$ with a map $S \to N$. We allow $S$ to be empty. We label the finite set $S$ by $\{ 1, \dots, m \}$ and write $s_j \in N$ for the image of the $j$th element of $S$. Following \cite{jiang}*{Definition 2.1}, we consider the $S$-\emph{extended stacky fan} $(N,\Sigma,\rho^S)$, where $\rho^S \colon \ZZ^{n+m} \to N$ is defined by
\begin{equation} \label{eq:extended_fan_map_general}
\rho^S (e_i) = \begin{cases}
\rho_i & i = 1, \dots, n \\
s_{i-n} & i = n+1, \dots, n+m
\end{cases}
\end{equation}
and $e_i$ is the $i$th standard basis vector for $\ZZ^{n+m}$. This gives the \emph{extended fan sequence}
\begin{equation} \label{eq:extended_fan_sequence_general}
0 \longrightarrow \LL^S \longrightarrow \ZZ^{n+m} \overset{\rho^S}\longrightarrow N 
\end{equation}
and by Gale duality the \emph{extended divisor sequence}
\begin{equation} \label{eq:extended_divisor_sequence_general}
0 \longrightarrow M \longrightarrow \ZZ^{n+m} \overset{D^S}\longrightarrow \LL^{S \vee} \longrightarrow 0.
\end{equation}
The inclusion $\ZZ^n \to \ZZ^{n+m}$ of the first $n$ factors induces an exact sequence
\begin{equation*}
0 \longrightarrow \LL \longrightarrow \LL^S \longrightarrow \ZZ^m,
\end{equation*}
which splits over $\QQ$ via the map $\QQ^m \to \LL^S \otimes_\ZZ \QQ$ that sends the $j$th standard basis vector to
\begin{equation*}
e_{j+n} - \sum_{i : \rho_i \in \sigma(j)} s^i_j e_i \in \LL^S \otimes_\ZZ \QQ \subseteq \QQ^{n+m}
\end{equation*}
where $\sigma(j) \in \Sigma$ is the minimal cone containing $s_j$ and the positive numbers $s^i_j$ are determined by $\sum_{i : \rho_i \in \sigma(j)} s^i_j \rho_i = s_j$. Thus we obtain an isomorphism:
\begin{equation} \label{eq:isomorphism_splitting}
\LL^S \otimes_\ZZ \QQ \simeq (\LL \otimes_\ZZ \QQ) \oplus \QQ^m.
\end{equation}
Therefore, an element $\lambda \in \LL^S \otimes_\ZZ \QQ \subseteq \QQ^{m+n}$ correspond to the pair $(d,k)$, where
\begin{align*}
d &= \sum_{i=1}^n \left( \lambda_i + \sum_{j=1}^m s^i_j \lambda_{n+j} \right) e_i \in \LL \otimes_\ZZ \QQ \subseteq \QQ^n,  \\
k &= \sum_{j=1}^m \lambda_{n+j} e_j \in \QQ^m. 
\end{align*}
The \emph{extended Mori cone} is the subset of $\LL^S \otimes_\ZZ \RR$ given by $\mathrm{NE}^S (\cX) = \mathrm{NE}(\cX) \times (\RR_{\geq 0})^m$ via the isomorphism \eqref{eq:isomorphism_splitting}. The extended Mori cone can be thought of as the cone spanned by the `extended degrees' of certain stable maps $ f \colon \cC \to \cX$: see \cite{mirror_theorem}*{\S 4} for details. The dual of $\mathrm{NE}^S (\cX)$ in $\LL^{S \vee} \otimes_\ZZ \RR$ is called the \emph{extended nef cone} of $\cX$ and is denoted by $\mathrm{Nef}^S (\cX)$. There is an equality of cones in $\LL^{S \vee} \otimes_\ZZ \RR$
\begin{equation} \label{eq:extended_nef_cone_general}
\mathrm{Nef}^S(\cX) = \bigcap_{\sigma \in \Sigma} \mathrm{cone} \left\langle \{ D^S_i \mid 1 \leq i \leq n, \rho_i \notin \sigma \} \cup \{ D^S_{n+j} \mid 1 \leq j \leq m \} \right\rangle,
\end{equation}
where $D^S_i \in \LL^{S \vee} \otimes_\ZZ \RR$ is the image of the $i$th element of the standard basis of $\ZZ^{n+m}$ via the map $D^S$.

For a cone $\sigma \in \Sigma$, denote by $\Lambda_\sigma^S \subseteq \LL^S \otimes_\ZZ \QQ$ the subset consisting of elements
\begin{equation*}
\lambda = \sum_{i=1}^{n+m} \lambda_i e_i \in \LL^S \otimes_\ZZ \QQ \subseteq \QQ^{n+m}
\end{equation*}
such that $\lambda_{n+j} \in \ZZ$, $1 \leq j \leq m$, and $\lambda_i \in \ZZ$ whenever $\rho_i \notin \sigma$ and $i \leq n$. Set $\Lambda^S := \bigcup_{\sigma \in \Sigma} \Lambda_\sigma^S$ and $\Lambda \rE^S := \Lambda^S \cap \rN \rE^S$. The \emph{reduction function} is $v^S \colon \Lambda^S \to \mathrm{Box}(\Sigma)$ defined by
\begin{equation*}
v^S (\lambda) = \sum_{i=1}^n \lceil \lambda_i \rceil \rho_i + \sum_{j=1}^m \lceil \lambda_{n+j} \rceil s_j = \sum_{i=1}^n \langle - \lambda_i \rangle \rho_i.
\end{equation*}
If $\lambda \in \Lambda^S_\sigma$, then $v^S(\lambda) \in \mathrm{Box}(\Sigma) \cap \sigma$.

\subsection{The mirror theorem for toric stacks} \label{sec:mirror_theorem}

By using the combinatorial objects associated to extended stacky fans (as in \S\ref{sec:stacky_fans}),  we give the definition of I-function for a toric stack. Let $\cX$ be a toric orbifold as above and let $(N, \Sigma, \rho^S)$ be an $S$-extended stacky fan defining $\cX$. Then the \emph{$S$-extended I-function} \cite{mirror_theorem} of $\cX$ is:
\begin{equation} \label{eq:I-function}
I^S(\tau, \xi, z) := z \rme^{\sum_{i=1}^n u_i \tau_i / z} \sum_{\lambda \in \Lambda \rE^S} \tilde{Q}^\lambda \rme^{\lambda \tau} \left( \prod_{i=1}^{n+m}  \ambientfactor{\lambda_i}{u_i} \right) \mathbf{1}_{v^S(\lambda)},
\end{equation}
where:
\begin{itemize}
\item $\tau = (\tau_1, \dots, \tau_n)$ are formal variables;
\item $\xi = (\xi_1, \dots, \xi_m)$ are formal variables;
\item for $1 \leq i \leq n$, $u_i \in \rH^2(\cX;\QQ)$ is the first Chern class of the the line bundle corresponding to the $i$th toric divisor $D_i$;
\item for $n+1 \leq i \leq n+m$, $u_i$ is defined to be zero;
\item for $\lambda \in \Lambda \rE^S$,
\begin{equation*}
\tilde{Q}^\lambda = Q^d \xi_1^{k_1} \cdots \xi_m^{k_m} \in \boldsymbol{\Lambda} [ \! [ \xi_1, \dots, \xi_m ] \! ],
\end{equation*}
where $d \in \LL \otimes_\ZZ \QQ$ and $k \in \NN^m$ are such that that $\lambda$ corresponds to $(d,k)$ via \eqref{eq:isomorphism_splitting} and $Q^d$ denotes the representative of $d \in \mathrm{Eff}(\cX)$ in the Novikov ring $\boldsymbol{\Lambda}$;
\item for $\lambda \in \Lambda \rE^S$, $\rme^{\lambda \tau} := \prod_{i=1}^n \rme^{(u_i \cdot d) \tau_i}$;
\item for $\lambda \in \Lambda \rE^S$, $\mathbf{1}_{v^S(\lambda)} \in \rH^0 (\cX_{v^S(\lambda)} ; \QQ) \subseteq \rH^{2 \mathrm{age}(v^S(\lambda))}_\mathrm{CR}(\cX)$ is the identity class supported on the component of inertia associated to $v^S(\lambda) \in \mathrm{Box}(\Sigma)$.
\end{itemize}
The I-function $I^S (\tau, \xi,z)$ is a formal power series in $Q, \xi, \tau$ with coefficients in $\rH^\bullet_\mathrm{CR} (\cX) \otimes_\QQ \CC ( \! ( z \inv ) \! )$.

\begin{theorem}[Mirror theorem for toric stacks \cite{mirror_theorem}] \label{thm:mirror_theorem}
Let $X$ be a projective simplicial toric variety, associated to the fan $\Sigma$ in the lattice $N$, and let $\cX$ be the corresponding toric well-formed orbifold constructed above. Let $S$ be a finite set equipped with a map to $N$. Then the $S$-extended I-function $I^S(\tau, \xi, -z)$ lies in the Givental cone $\cL_\cX$ for all values of the parameters $\tau$ and $\xi$.
\end{theorem}

The mirror theorem relates the combinatorics of toric geometry (namely the I-function) with Gromov--Witten theory (namely the Givental cone $\cL_\cX$). In \S \ref{subsec:example_1} we will show an example in which the mirror theorem is applied to compute the quantum period of a toric Fano orbifold.

\begin{remark} \label{rmk:version_I_function}
The formula \eqref{eq:I-function} for the extended I-function is given in \cite{mirror_theorem}. In our calculations we will use a slightly different version, which provides the same amount of information in Theorem \ref{thm:mirror_theorem}. Let $p_1, \dots, p_\ell$ be an integral basis of $\rH^2(\cX;\QQ)$. Then we use the formal variables $\tau = (\tau_1, \dots, \tau_\ell)$ and the exponentials appearing in \eqref{eq:I-function} are replaced by $\rme^{\sum_{i=1}^\ell p_i \tau_i /z}$ and by $\rme^{\lambda \tau} := \prod_{i=1}^\ell \rme^{(p_i \cdot d) \tau_i}$.
\end{remark}

\subsection{A toric example: the blow-up of $\PP(1,1,3)$ at one point.} \label{subsec:example_1}

Let $X$ be the blow-up of $\PP(1,1,3)$ at a smooth point. It has one singular point and degree
$22/3$, hence in our list it is called $X_{1,22/3}$. The well-formed orbifold $\cX$ having $X$ as coarse moduli space is a toric stack and is defined by the following stacky fan.
 \begin{minipage}{10cm}
Let $P$ be the polygon n.25 in the list below, let $N = \ZZ^2$, and let $\Sigma$ be the spanning fan of $P$. The rays of $\Sigma$, which is a complete fan in $\RR^2$, are the columns of the matrix
\begin{equation*}
\rho = \begin{bmatrix}
1 & 0 & -1 & -2  \\
-1 & 1 & 2 & 1
\end{bmatrix}.
\end{equation*}
\end{minipage}
\begin{minipage}{2cm}
 \includegraphics[scale=0.7]{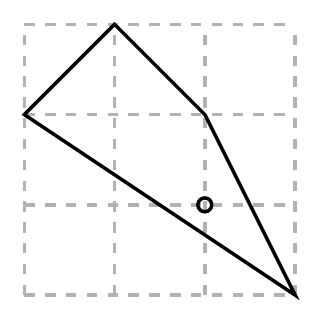}
\end{minipage}

\noindent It is easy to observe that the rays $\rho_1,\rho_3$ and $\rho_4$ would define $\PP(1,1,3)$ and that the toric divisor corresponding to $\rho_2$ is obtained after blowing-up $\PP(1,1,3)$ at one torus-invariant point.

A basis of $\LL = \ker (\rho \colon \ZZ^4 \to N)$ is given by the vectors 
\begin{equation} \label{eq:basis_of_L_example1}
\begin{bmatrix}
3 \\ 0 \\ 1 \\ 1
\end{bmatrix},
\begin{bmatrix}
-1 \\ 1 \\ -1 \\ 0
\end{bmatrix}.
\end{equation}
We use this basis to identify $\LL$ with $\ZZ^2$.
The fan sequence \eqref{eq:fan_sequence_general} is
\begin{equation*}
0 \longrightarrow \LL \simeq \ZZ^2 \longrightarrow \ZZ^4 \overset{\rho}\longrightarrow N = \ZZ^2 \longrightarrow 0.
\end{equation*}
Since $\rho$ is surjective, the divisor sequence \eqref{eq:divisor_sequence_general} is
\begin{equation*}
0 \longrightarrow M = \ZZ^2 \longrightarrow \ZZ^4 \overset{D}\longrightarrow \LL^\vee \simeq \ZZ^2 \longrightarrow 0,
\end{equation*}
where $\LL^\vee$ is identified with $\ZZ^2$ via the dual basis $\{ p_1, p_2 \}$ of \eqref{eq:basis_of_L_example1} and the linear map $D$ is given by the matrix
\begin{equation*}
D = \begin{bmatrix}
3 & 0 & 1 & 1 \\
-1 & 1 & -1 & 0
\end{bmatrix}.
\end{equation*}
The anticanonical class of $X$ is given by the sum of the divisor classes of irreducible torus-invariant divisors: $-K_X = 5p_1 - p_2$ in $\mathrm{Cl}(X)$.
 
The irrelevant ideal is $\mathrm{Irr}_\Sigma =  ( x_3 x_4, x_1 x_4, x_1 x_2, x_2 x_3 )$. Set $U_\Sigma = \AA^4_\CC \setminus \rV(\mathrm{Irr}_\Sigma)$. The matrix $D$ induces a group homomorphism from the $2$-dimensional torus $\GG_\rmm^2$ to the group of invertible diagonal matrices inside $\rG \rL_4$, therefore we have an action of $\GG_\rmm^2$ on the affine space $\AA^4_\CC$. The toric variety $X$ is the geometric quotient of $U_\Sigma$ with respect to this action, i.e. $X =  U_\Sigma / \GG_\rmm^2$, and $\cX$ is the quotient stack $[U_\Sigma / \GG_\rmm^2 ]$. Using \eqref{eq:nef_cone_general} we get that the nef cone of $X$ is $\Nef(X) =  \mathrm{cone}\langle 3 p_1 - p_2,p_1\rangle \subseteq \LL^\vee \otimes_\ZZ \RR$. The anticanonical class $-K_X$ is in the interior of the nef cone, i.e. $X$ is Fano.

Now we analyze the Chen--Ruan cohomology of $\cX$. The inertia stack $\rI \cX$ has three connected components: the component with age $0$, which is isomorphic to $\cX$, and two components isomorphic to $\rB \bmu_3$ corresponding to the non-trivial stabilizers of the singular point, which have ages $\frac{2}{3}$ and $\frac{4}{3}$. A basis of the rational cohomology of $X$ is given by $\{ \bfone_0, p_1, p_2, \rmp \rmt \}$.  Therefore, if we denote by $\bfone_{2/3}$ and $\bfone_{4/3}$ the cohomology classes of the non-trivial components of $\rI \cX$, we have that $\{ \bfone_0, \bfone_{2/3}, p_1, p_2, \bfone_{4/3}, \rmp \rmt \}$ is a basis of $\chenruan(\cX)$.  The set of the connected components of $\rI \cX$ is in a canonical one-to-one correspondence with $ \mathrm{Box}(\Sigma) = \{ (0,0), (-1,1), (-2,2) \}$; the zero vector corresponds to the trivial component of $\rI \cX$, whereas the vectors $(-1,1)$ and $(-2,2)$ correspond to the non-trivial components of $\rI \cX$ with ages $\frac{2}{3}$ and $\frac{4}{3}$, respectively.

Since we are interested in the part of $\chenruan(\cX)$ of degree smaller than $2$, we `extend' with the vector $(-1,1)$. In other words, we consider the map $S = \{1 \} \to N$ with $s_1 = (-1,1) \in N$ and the corresponding extended stacky fan, which is the one with extended ray map \eqref{eq:extended_fan_map_general}
\begin{equation*}
\rho^S = \begin{bmatrix}
1 & 0 & -1 & -2  & -1\\
-1 & 1 & 2 & 1 & 1
\end{bmatrix}.
\end{equation*}
A basis of $\LL^S = \ker (\rho^S \colon \ZZ^5 \to \ZZ^2 = N)$ is given by the vectors 
\begin{equation} \label{eq:basis_L_extended_example1}
\begin{bmatrix}
3 \\ 0 \\ 1 \\ 1 \\ 0
\end{bmatrix},
\begin{bmatrix}
-1 \\ 1 \\ -1 \\ 0 \\ 0
\end{bmatrix},
\begin{bmatrix}
1 \\ 0 \\ 0 \\ 0 \\ 1
\end{bmatrix}.
\end{equation}
We use this basis to identify $\LL^S$ with $\ZZ^3$ and we call $l_1, l_2, k$ the coordinates with respect to this basis. The extended fan sequence \eqref{eq:extended_fan_sequence_general} is
\begin{equation*}
0 \longrightarrow \LL^S \simeq \ZZ^3 \longrightarrow \ZZ^5 \overset{\rho^S}\longrightarrow N = \ZZ^2 \longrightarrow 0,
\end{equation*}
where the inclusion $\LL^S \otimes_\ZZ \RR \into \RR^5$ is given by
\begin{equation*}
(l_1, l_2, k) \mapsto \begin{pmatrix}
3 l_1 - l_2 + k \\
l_2 \\
l_1 - l_2 \\
l_1 \\
k
\end{pmatrix}.
\end{equation*}
The extended divisor sequence \eqref{eq:extended_divisor_sequence_general} is
\begin{equation*}
0 \longrightarrow M = \ZZ^2 \longrightarrow \ZZ^5 \overset{D^S}\longrightarrow \LL^{S \vee} \simeq \ZZ^3 \longrightarrow 0,
\end{equation*}
where $\LL^{S \vee}$ is identified with $\ZZ^3$ via the dual basis of \eqref{eq:basis_L_extended_example1} and the linear map $D^S$ is given by the matrix
\begin{equation*}
D^S = \begin{bmatrix}
3 & 0 & 1 & 1 & 0 \\
-1 & 1 & -1 & 0 & 0 \\
1 & 0 & 0 & 0 & 1
\end{bmatrix}.
\end{equation*}
By \eqref{eq:extended_nef_cone_general} the extended nef cone is
\begin{equation*}
\Nef^S(\cX)=\mathrm{cone} \left\langle
\begin{bmatrix}
3 \\ -1 \\ 1
\end{bmatrix},
\begin{bmatrix}
3 \\ 0 \\ 1
\end{bmatrix},
\begin{bmatrix}
0 \\ 0 \\ 1
\end{bmatrix}
\right\rangle.
\end{equation*}
The extended Mori cone $\NE^S \subseteq \LL^S \otimes_\ZZ \RR$ is the dual of $\Nef^S$, so it is defined by the inequalities $3 l_1 - l_2 + k \geq 0$, $3 l_1 + k \geq 0$, $k \geq 0$. 

We will not write down a description of $\Lambda_\sigma^S$ for every cone $\sigma \in \Sigma$. We just mention, for example, that if $\sigma$ is the cone spanned by $\rho_2$ and $\rho_4$ then $\Lambda_{\sigma}^S$ is defined by the conditions $3 l_1 - l_2 + k \in \ZZ$, $l_1 - l_2 \in \ZZ$, $k \in \ZZ$. After few computations one finds that 
\begin{equation*}
\Lambda \rE^S= \left\{ \left. \begin{pmatrix}
3 l_1 - l_2 + k \\
l_2 \\
l_1 - l_2 \\
l_1 \\
k
\end{pmatrix} \in \RR^5 \right| 
\begin{matrix}
3 l_1 \in \ZZ, l_2 \in \ZZ, k \in \ZZ, \\
3 l_1 - l_2 + k \geq 0, \\
3 l_1 + k \geq 0, \\
k \geq 0
\end{matrix}
\right\}.
\end{equation*}
The extended reduction function  $v^S \colon \Lambda \rE^S \to \mathrm{Box}(\Sigma)$ is defined by
\begin{align*}
v^S (l_1, l_2, k) &= \lceil 3l_1 - l_2 +k \rceil \begin{pmatrix}
1 \\ -1
\end{pmatrix} + \lceil l_2 \rceil \begin{pmatrix}
0 \\ 1
\end{pmatrix} + \lceil l_1 - l_2 \rceil \begin{pmatrix}
-1 \\ 2
\end{pmatrix} + \\
&\quad \quad + \lceil l_1 \rceil \begin{pmatrix}
-2 \\ 1
\end{pmatrix} + \lceil k \rceil \begin{pmatrix}
-1 \\ 1
\end{pmatrix}  \\
&= 3 ( \lceil l_1 \rceil - l_1) \begin{pmatrix}
-1 \\ 1
\end{pmatrix}.
\end{align*}
Since $s_1 = \frac{1}{3} \rho_3 + \frac{1}{3} \rho_4$, we have that
the image of $(l_1, l_2, k) \in \LL^S \otimes_\ZZ \QQ$ in $\LL \otimes_\ZZ \QQ$ via the splitting \eqref{eq:isomorphism_splitting} is
\begin{equation} \label{eq:d_example_1}
d = \begin{bmatrix}
3 l_1 - l_2 + k \\ l_2 \\ l_1 - l_2 + \frac k 3 \\ l_1 + \frac k 3
\end{bmatrix} = \left( l_1 + \frac k 3 \right) \begin{bmatrix}
3 \\ 0 \\ 1 \\ 1
\end{bmatrix} + l_2 \begin{bmatrix}
-1 \\ 1 \\ -1 \\ 0
\end{bmatrix}.
\end{equation}
Therefore, the $S$-extended Novikov variable corresponding to $(l_1, l_2, k) \in \NE^S (\cX)$ is $\tilde{Q}^{(l_1, l_2,k)} = Q^d \xi^k$, where $d \in \NE(X)$ is given by \eqref{eq:d_example_1}.
To match notation with \eqref{eq:I-function}, set $\bfone_{(0,0)} = \mathbf{1}_0$, $\bfone_{(-1,1)} = \bfone_{2/3}$, $\bfone_{(-2,2)} = \bfone_{4/3}$. The Chern classes of the toric divisors are
\begin{align*}
u_1 &= 3p_1 - p_2, \\
u_2 &= p_2, \\
u_3 &= p_1 - p_2, \\
u_4 &= p_1.
\end{align*}
Following Remark \ref{rmk:version_I_function}, by \eqref{eq:I-function} the $S$-extended I-function $I^S(\tau_1, \tau_2, \xi;z)$ of $\cX$ is
\begin{align*}
I^S (\tau_1, \tau_2, \xi; z) &= z \rme^{(\tau_1 p_1 + \tau_2 p_2)/z} \times \\
&\quad \quad \times \sum_{(l_1, l_2, k) \in \Lambda \rE^S} \tilde{Q}^{(l_1, l_2,k)} \rme^{\tau_1 \left(l_1 + \frac{k}{3} \right) + \tau_2 l_2} \square_{(l_1, l_2, k)} \mathbf{1}_{v^S(l_1,l_2,k)},
\end{align*}
where
\begin{align*}
\square_{(l_1, l_2, k)} &= \ambientfactorwithoutnumerator{3l_1 - l_2 +k}{3p_1 - p_2} 
\ambientfactor{l_2}{p_2} \times \\
&\qquad \qquad \times \ambientfactor{l_1-l_2}{p_1-p_2} 
\ambientfactor{l_1}{p_1} \frac{1}{k! z^k}
\end{align*}

Now we want to study the asymptotic behaviour of $I^S$ with respect to the variable $z$. Note that if $(l_1, l_2, k) \in \Lambda \rE^S$ and $\deg_z \square_{(l_1, l_2, k)} \geq -1$ then either $(l_1, l_2, k) = (0,0,0)$ or $(l_1, l_2, k) = (- \frac 1 3 , 0, 1)$. Since $\tilde{Q}^{(- \frac 1 3 , 0, 1)} = \xi$ and $v^S(- \frac 1 3 , 0, 1) = (-1,1)$, we obtain
\begin{align*}
I^S(\tau_1, \tau_2, \xi ; z) &= z \rme^{(\tau_1 p_1 + \tau_2 p_2)/z} \left( \bfone_0 + z \inv \xi \bfone_{2/3} + \rO(z^{-2}) \right) \\
&= z \bfone_0 + \tau_1 p_1 + \tau_2 p_2 + \xi \mathbf{1}_{2/3} + \rO(z \inv),
\end{align*}
where $\rO(z \inv)$ denotes term of the form $\sum_{n=1}^\infty c_n z^{-n}$ with $c_n$ independent of $z$.

Since the J-function is the only point of the Givental cone of $\cX$ with the asymptotic expansion $z \bfone_0 + F(t) + O(z \inv)$, by the mirror theorem (Theorem \ref{thm:mirror_theorem}) we have that
\begin{equation*}
J (\tau_1 p_1 + \tau_2 p_2 + \xi \bfone_{2/3};z) = I^S (\tau_1, \tau_2, \xi; z)
\end{equation*}
for every $\tau_1, \tau_2, \xi \in \CC$. To obtain the quantum period of $\cX$, we have to set $z=1$, $\tau_1 = \tau_2 = 0$, $\xi = t^{\frac 1 3} x $, replace the Novikov variable $Q^d$ with $t^{-K_X \cdot d}$, and take the component along $\bfone_0$ of the $J$-function. In particular
\begin{equation*}
\tilde{Q}^{(l_1, l_2, k)} = Q^{(l_1 + \frac k 3, l_2)} \xi^k \mapsto t^{5 (l_1 + \frac k 3) - l_2} \left( x t^{\frac{1}{3}} \right)^k = x^k t^{5 l_1 - l_2 + 2k}.
\end{equation*}
Thus the quantum period of $X$ is
\begin{equation*}
G_X(x;t) = \sum\limits_{\substack{l_1, l_2, k \in \ZZ, \\  l_1 \geq l_2 \geq 0, \ k \geq 0, \\
3 l_1 + k \geq l_2}} \frac{1} { (3 l_1 - l_2 + k)! l_2! (l_1 - l_2)! l_1 ! k!  }x^k t^{5 l_1 - l_2 + 2k}.
\end{equation*}
The regularised quantum period of $X$ is
\begin{equation*}
\widehat{G}_X(x;t) = 1 + 2xt^2 + (12+6x^2) t^4 + 20 t^5 + (120 x + 20 x^3) t^6 + \dots
\end{equation*}
and coincides with the classical period of the $1$-parameter family of Laurent polynomial 
\begin{equation*}
f_a(x,y) = \frac{y^2}{x} + \frac{y}{x^2} + \frac{x}{y} + y + a \frac{y}{x}
\end{equation*}
associated to the 
polygon n.25, after identifying the parameter $x$ in $\widehat{G}_X$ with the parameter $a$.

\section{Toric complete intersections} \label{sec:toric_complete_intersections}

\subsection{The quantum Lefschetz principle}\label{sec:quantum_Lef_principle}
The Gromov--Witten invariants of a complete intersection are governed by the so-called `quantum Lefschetz principle', which was formulated and proven by Coates and Givental \cite{quantum_lefschetz_tom_givental} in the case of smooth projective varieties (see also \cite{lee} and \cite{kim}). It has been shown that this principle fails for some positive line bundles on some orbifolds \cite{quantum_lefschetz_can_fail}, but there is evidence that the quantum Lefschetz principle holds in cases which are sufficient for us. We will be more precise below.

Firstly we have to define twisted Gromov--Witten invariants. Let $\cY$ be a proper smooth Deligne--Mumford stack over $\CC$ with projective coarse moduli space $Y$ and let $\cE$ be a vector bundle on $\cY$. If $n \in \NN$ and $d \in \mathrm{Eff}(\cY)$, the universal family over the moduli space of stable maps
\begin{equation*}
\xymatrix{\cC_{0,n,d} \ar[r]^{\mathrm{ev}} \ar[d]_\pi & \cY \\
\cY_{0,n,d}
}
\end{equation*}
induces an element $\cE_{0,n,d} := \pi_! \mathrm{ev}^\star \cE$ of the K-theory of $\cY_{0,n,d}$. The torus $\CC^*$ acts on $\cE$ rotating the fibres and leaving the base $\cY$ invariant. This action induces an action of $\CC^*$ on $\cE_{g,n,d}$. Let $\bfe$ be the $\CC^*$-equivariant Euler class, which is invertible over the field of fractions $\QQ(\kappa)$ of $\rH^\bullet_{\CC^*}( \mathrm{pt} ; \QQ) = \rH^\bullet (\rB \CC^*; \QQ) = \QQ[\kappa]$, where $\kappa$ is the equivariant parameter given by the first Chern class of the line bundle $\cO(1)$ on $\CC \PP^\infty = \rB \CC^*$. \emph{$\cE$-twisted Gromov--Witten invariant} are defined by
\begin{equation*}
\left\langle \al_1 \psi^{k_1}, \dots, \al_n \psi^{k_n} \right\rangle_{0,n,d}^{\mathrm{tw}} := \int_{[\cY_{0,n,d}]^\mathrm{vir}} \bfe (\cE_{0,n,d}) \cup \prod_{i=1}^n \left( \mathrm{ev}^\star_i(\al_i) \cup \psi_i^{k_i} \right) \in \QQ(\kappa),
\end{equation*}
for $\al_1, \dots, \al_n \in \chenruan(\cY)$ and non-negative integers $k_1, \dots, k_n$.
The inertia stack $\rI \cE$ of the total space of the vector bundle $\cE \to \cY$ is a vector bundle over $\rI \cY$: the fibre of $\rI \cE$ over the point $(y,g)$ is the $g$-fixed subspace of the fibre of $\cE$ over $y$. One can define the twisted Poincar\'e pairing
\begin{equation*}
(\al , \beta)^\mathrm{tw}_\mathrm{CR} := \int_{\rI \cX} \al \cup \mathrm{inv}^\star \beta \cup \bfe (\rI \cE) \in \QQ[\kappa], \qquad \al, \beta \in \rH^\bullet_\mathrm{CR}(\cY)
\end{equation*}
and the twisted symplectic form
\begin{equation}
\Omega^\mathrm{tw} (f,g) := - \mathrm{Res}_{z= \infty} \big( f(-z), g(z) \big)^\mathrm{tw}_\mathrm{CR} \rmd z
\end{equation}
on $\cH_\cY \otimes_\QQ \QQ(\kappa)$, where $\cH_\cY$ is defined in \eqref{eq:symplectic_vector_space_H}. In the symplectic vector space $(\cH_\cY \otimes_\QQ \QQ(\kappa), \Omega^\mathrm{tw})$ there is a Lagrangian submanifold, which is a formal germ of a cone with vertex at the origin and which encodes all genus-zero Euler-twisted Gromov--Witten invariants of $\cX$: it is called the \emph{twisted Givental cone} and is denoted by $\cL_\cY^\mathrm{tw}$. We will not give a precise definition of $\cL_\cY^\mathrm{tw}$ here, referring the reader to \cite{tseng_orbifold_quantum_riemann_roch}, \cite{computing_twisted}. $\cL_\cY^\mathrm{tw}$ determines and is determined by Givental's \emph{twisted J-function}:
\begin{equation}
J_\cY^\mathrm{tw} (\gamma,z) = z + \gamma + \sum_{d \in \mathrm{Eff}(\cX)} \sum_{n=0}^\infty \sum_{k=0}^\infty \sum_{\epsilon = 1}^N \frac{Q^d}{n!}
\left\langle \gamma, \dots, \gamma, \phi^\epsilon_\bfe \psi^k \right\rangle_{0, n+1, d}^\mathrm{tw} \phi_\epsilon z^{-k-1},
\end{equation}
where:
\begin{itemize}
\item $\gamma$ runs in the even part $\rH^\bullet_{\rC \rR} (\cY)$ of the Chen-Ruan orbifold cohomology  of $\cY$;
\item $\{ \phi_1, \dots \phi_N \}$ and $\{ \phi^1_\bfe, \dots \phi^N_\bfe \}$ are homogeneous bases of the $\QQ(\kappa)$-vector space $\chenruan(\cY) \otimes_\QQ \QQ(\kappa)$ which are dual with respect to the twisted Poincar\'e pairing $( \cdot, \cdot )_\mathrm{CR}^\mathrm{tw}$.
\end{itemize}
The cone $\cL_\cY^\mathrm{tw}$ determines the twisted J-function because $J_\cY^\mathrm{tw}(\gamma, -z)$ is the unique point on $\cL_\cY^\mathrm{tw}$ of the form $-z + \gamma + \rO(z \inv)$.
Actually $\cL_\cY^\mathrm{tw}$ is a family of cones in $\cH_\cY$ and $J_\cY^\mathrm{tw}$ is a family of elements in $\cH_\cY$. Both families are parameterised by the equivariant parameter $\kappa$ in some open set of $\AA^1_\CC$.

\medskip

Now we are going to say what is the relationship between $\cE$-twisted Gromov--Witten invariants of $\cY$ and the ordinary Gromov--Witten invariants of the zero locus $\cX$ of a generic global section of $\cE$. Before doing that, we introduce the class of convex vector bundles.

The vector bundle $\cE$ over $\cY$ is \emph{convex}  if $\rH^1(\cC, f^\star \cE) = 0$ for all genus-zero $n$-pointed stable maps $f \colon \cC \to \cY$, for any $n$. If $\cE$ is convex, then $\rR^1 \pi_\star \mathrm{ev}^\star \cE = 0$  and by cohomology and base change $\cE_{0,n,d}$ is the class of the vector bundle  $\pi_\star \mathrm{ev}^\star \cE$ over $\cY_{0,n,d}$, for all $n \in \NN$ and $d \in \mathrm{Eff}(\cY)$. Therefore, every $\cE$-twisted Gromov--Witten invariant lies in $\QQ[\kappa]$.
A line bundle $\cE$ on $\cY$ is convex if and only if it is the pull-back of a nef line bundle from the coarse moduli space $Y$ (see \cite{quantum_lefschetz_can_fail}).

Now we consider the following setup:
\begin{itemize}
\item[($\dagger$)] $\cY$ is a proper smooth Deligne--Mumford stack over $\CC$ with projective coarse moduli space; $\cE$ is a vector bundle on $\cY$ and $i \colon \cX \into \cY$ is the closed substack defined by a regular section of $\cE$;  $\iota^\star \colon \chenruan(\cY) \to \chenruan(\cX)$ is the pull-back defined by the inclusion $\iota \colon \rI \cX \into \rI \cY$; $J^\mathrm{tw}_\cY$ is the $\cE$-twisted J-function of $\cY$ and $J_\cX$ is the non-twisted J-function of $\cX$; $\cL_\cY^\mathrm{tw}$ is the $\cE$-twisted Givental cone of $\cY$ and $\cL_\cX$ is the non-twisted Givental cone of $\cX$.
\end{itemize}

Under the hypothesis that the vector bundle $\cE$ is convex, the following theorem relates the Gromov--Witten invariants of the complete intersection to the twisted invariants of the ambient.

\begin{theorem}[\cite{iritani_periods, quantum_lefschetz_tom}] \label{thm:tom_quantum_lefschetz}
Let $\cY, \cE, \cX$ be as in $(\dagger)$. If $\cE$ is convex, then the non-equivariant limit $\lim_{\kappa \to 0} J^\mathrm{tw}_\cY $ is well-defined and satisfies:
\begin{equation*}
\iota^\star \left( \lim_{\kappa \to 0} J_\cY^\mathrm{tw} (\gamma) \right) = J_\cX (\iota^\star \gamma)
\end{equation*}
for all $\gamma \in \chenruan(\cY)$. Moreover, if $I^\mathrm{tw}$ is a point of $\cL^\mathrm{tw}_\cY$, then the non-equivariant limit $I_{\cX, \cY} := \iota^\star \left( \lim_{\kappa \to 0} I^\mathrm{tw} \right)$ is well-defined and lies in $\cL_\cX$.
\end{theorem}

Without the convexity hypothesis, it is conjectured that there is some relation between invariants of the complete intersection and twisted invariants of the ambient.

\begin{conjecture}[Coates--Corti--Iritani--Tseng \cite{private_communication}, cf. \cite{quantum_lefschetz_can_fail}] \label{conj:tom_quantum_lefschetz}
Let $\cY, \cE, \cX$ be as in $(\dagger)$. Let $I^\mathrm{tw}$ be a point of $\cL_\cY^\mathrm{tw}$ such that the non-equivariant limit $I_{\cX, \cY} := \lim_{\kappa \to 0} \iota^\star I^\mathrm{tw}$ is well-defined. Then $I_{\cX, \cY}$ lies in $\cL_\cX$.
\end{conjecture}

\begin{remark}
In Theorem \ref{thm:tom_quantum_lefschetz} and Conjecture \ref{conj:tom_quantum_lefschetz} we have applied the homomorphism $Q^\delta \mapsto Q^{i_\star \delta}$ to the Novikov ring of $\cX$.
\end{remark}

\subsection{Quantum Lefschetz for toric orbifolds} \label{sec:Lefschetz_toric_orbifolds}

Here we discuss twisted Gromov--Witten invariants of toric orbifolds. We maintain all the notations we used in \S \ref{sec:stacky_fans}. In particular, we assume that $\cY$ is a toric well-formed orbifold coming from the stacky fan $(N, \Sigma, \rho)$, where $N$ is a finitely generated free abelian group, $\Sigma$ is a rational simplicial fan in $N_\RR$ and $\rho$ is the ray map of $\Sigma$. We denote by $Y$ the toric variety that is the coarse moduli space of $\cY$. 
We also use the formalism of $S$-extended stacky fans introduced  in \S\ref{sec:stacky_fans}, where $S$ is  a finite set with a map $S \to N$.

Let $\cE_1, \dots, \cE_r$ be line bundles on $\cY$. Consider the vector bundle $\cE = \cE_1 \oplus \cdots \oplus \cE_r$ on $\cY$ and  choose $\epsi_1, \dots, \epsi_r \in \LL^{S \vee}$ such that their images $E_1, \dots, E_r$ in $\LL^\vee$ are the first Chern classes of $\cE_1, \dots, \cE_r$.
The \emph{$S$-extended $\cE$-twisted I-function} (see \cite{some_applications}) of $\cY$ is:
\begin{gather} \label{eq:twisted_I_function}
I^S_\cE(\tau, \xi; z) := z \rme^{\sum_{i=1}^n u_i \tau_i / z} \times \\\sum_{\lambda \in \Lambda \rE^S} \tilde{Q}^\lambda \rme^{\lambda t} \left( \prod_{i=1}^{n+m}  \ambientfactor{\lambda_i}{u_i} \right) \left( \prod_{j=1}^r 
\frac{
\prod\limits_{\substack{a \leq \epsi_j \cdot \lambda \\ \langle a \rangle = \langle \epsi_j \cdot \lambda \rangle}} (\kappa + E_j + az)
}{
\prod\limits_{\substack{a \leq 0 \\ \langle a \rangle = \langle \epsi_j \cdot \lambda \rangle}} (\kappa + E_j + az)
} 
\right)\bfone_{v^S(\lambda)}, \nonumber
\end{gather}
where:
\begin{itemize}
\item $\kappa$ is the equivariant parameter;
\item $\tau = (\tau_1, \dots, \tau_n)$ are formal variables;
\item $\xi = (\xi_1, \dots, \xi_m)$ are formal variables;
\item for $1 \leq i \leq n$, $u_i \in \rH^2(\cY;\QQ)$ is the first Chern class of the the line bundle corresponding to the $i$th toric divisor $D_i$;
\item for $n+1 \leq i \leq n+m$, $u_i$ is defined to be zero;
\item for $\lambda \in \Lambda \rE^S$,
\begin{equation*}
\tilde{Q}^\lambda = Q^d \xi_1^{k_1} \cdots \xi_m^{k_m} \in \boldsymbol{\Lambda} [ \! [ \xi_1, \dots, \xi_m ] \! ],
\end{equation*}
where $d \in \LL \otimes_\ZZ \QQ$ and $k \in \NN^m$ are such that that $\lambda$ corresponds to $(d,k)$ via \eqref{eq:isomorphism_splitting} and $Q^d$ denotes the representative of $d \in \mathrm{Eff}(\cY)$ in the Novikov ring $\boldsymbol{\Lambda}$;
\item for $\lambda \in \Lambda \rE^S$, $\rme^{\lambda t} := \prod_{i=1}^n \rme^{(p_i \cdot d) \tau_i}$;
\item for $\lambda \in \Lambda \rE^S$, $\mathbf{1}_{v^S(\lambda)} \in \rH^0 (\cY_{v^S(\lambda)} ; \QQ) \subseteq \rH^{2 \mathrm{age}(v^S(\lambda))}_\mathrm{CR}(\cY)$ is the unit class supported on the component of inertia associated to $v^S(\lambda) \in \mathrm{Box}(\Sigma)$.
\end{itemize}

Note that $I^S_\cE$ depends on the choice of the liftings $\epsi_j$ of $E_j = c_1(\cE_j) \in \LL^\vee$ to $\LL^{S \vee}$.

\begin{remark} \label{rmk:version_twisted_I_function}
For the twisted I-function \eqref{eq:twisted_I_function} we use the same substitutions as in Remark \ref{rmk:version_I_function}.
\end{remark}

\begin{theorem}[Twisted mirror theorem for toric stacks \cite{some_applications}] \label{thm:twisted_mirror_theorem}
Let $Y$ be a projective simplicial toric variety, associated to the fan $\Sigma$ in the lattice $N$, and let $\cY$ be the corresponding toric well-formed orbifold. Let $S$ be a finite set equipped with a map to $N$. Let $\cE_1, \dots, \cE_r \in \Pic(\cY)$ be line bundles and let $\cE = \cE_1 \oplus \cdots \oplus \cE_r$. 

For any choice of the liftings of $E_j = c_1 (\cE_j)$ to $\LL^{\vee S}$,  the $S$-extended $\cE$-twisted I-function $I^S_\cE(\tau, \xi, -z)$ lies in the $\cE$-twisted Givental cone $\cL_\cY^\mathrm{tw}$ for all values of the parameters $\tau$ and $\xi$.
\end{theorem}

\begin{remark}
The theorem above is useful when computing Gromov--Witten invariants of complete intersections in toric orbifolds. 
Assume we are in the situation of Theorem \ref{thm:twisted_mirror_theorem}. Let $i \colon \cX \into \cY$ be the zero locus of a regular section of $\cE$ and let $\iota^\star \colon \chenruan(\cY) \to \chenruan(\cX)$ be the pull-back defined by the inclusion $\iota \colon \rI \cX \into \rI \cY$. 
\begin{enumerate}
\item[(i)] Suppose that $I^S_\cE (\tau, \xi, z) = z \bfone_0 + F(\tau, \xi) + \rO(z \inv)$ and the line bundles $\cE_1, \dots, \cE_r$ are convex. Then $I^S_\cE$ determines the $\cE$-twisted J-function of $\cY$. We may apply Theorem \ref{thm:tom_quantum_lefschetz} to obtain the J-function of $\cX$. 
\item[(ii)] Assume that the non-equivariant limit $I_{\cX, \cY} := \lim_{\kappa \to 0} \iota^\star I^S_\cE$ is well-defined. If Conjecture \ref{conj:tom_quantum_lefschetz} holds and $I_{\cX, \cY} = z \bfone_0 + F(\tau, \xi) + \rO(z \inv)$, then $I_{\cX, \cY}$ determines the J-function of $\cX$.
\end{enumerate}
In \S\ref{subsec:example_2} and \S\ref{subsec:example_3} we give two examples of this.
\end{remark}

\subsection{Example of a toric complete intersection: $X_{2, 8/3}$} \label{subsec:example_2}

In the lattice $N = \ZZ^3$ consider the polytope such that its vertices are the columns of the matrix
\begin{equation*}
\rho = \begin{bmatrix}
 1 & 0 & 0 & -1 & 3 \\
 0 & 1 & 0 & -1 & 3 \\
 0 & 0 & 1 & -1 & 2 \\
\end{bmatrix}
\end{equation*}
and consider its spanning fan $\Sigma$. It contains six $3$-dimensional cones: $\sigma_{235}$, $\sigma_{234}$, 
$\sigma_{135}$, $\sigma_{134}$, $\sigma_{125}$, $\sigma_{124}$, where $\sigma_{ijk}$ is the cone spanned by the the $i$th, the $j$th and the $k$th columns of $\rho$. Let $\cY$ be the toric orbifold associated to the stacky fan $(N, \Sigma, \rho)$.

The rows of the matrix
\begin{equation*}
D = \begin{bmatrix}
 1 & 1 & 1 & 1 & 0  \\
  0 & 0 & 1 & 3 & 1  \\
\end{bmatrix}.
\end{equation*}
constitute a basis of $\LL = \ker(\rho \colon \ZZ^5 \to N = \ZZ^3)$. Therefore, the fan sequence \eqref{eq:fan_sequence_general} is
\begin{equation*}
0 \longrightarrow \LL \simeq \ZZ^2 \overset{^t \! D}\longrightarrow \ZZ^5 \overset{\rho}\longrightarrow N = \ZZ^3 \longrightarrow 0
\end{equation*}
and the divisor sequence \eqref{eq:divisor_sequence_general} is
\begin{equation*}
0 \longrightarrow M = \ZZ^3 \overset{ ^t \! \rho}\longrightarrow \ZZ^5 \overset{D}\longrightarrow \LL^\vee \simeq \ZZ^2 \longrightarrow 0.
\end{equation*}
Let $\{ p_1, p_2 \}$ be the basis of $\LL^\vee$ coming from the isomorphism $\LL \simeq \ZZ^2$ chosen above. The nef cone of $Y$ is $\Nef(Y) = \mathrm{cone} \langle p_1 + p_2, p_1 + 3 p_2 \rangle$. We see that $Y$ is a Fano 3-fold.
If we use $x_0, x_1, y, z, t$ as coordinates on $\AA^5$, the irrelevant ideal is $\mathrm{Irr}_\Sigma = (x_0, x_1, y) \cdot (z,t)$. Considering the open set 
$U_{\Sigma}=\AA^5_{\CC} \smallsetminus {\rm V}({\rm Irr}_{\Sigma})$,
the toric variety $Y$ is given by the quotient $U_{\Sigma}/ \GG_\rmm^2$, under the action of $\GG_\rmm^2$ on $\AA^5$ induced by the matrix $D$, and the toric orbifold $\cY$ is the stack-theoretic quotient $[U_{\Sigma}/ \GG_\rmm^2]$.

The singular locus of $Y$ has two components: a rational curve $C$, corresponding to the cone $\sigma_{35}$ and  made up of the points $[x_0 : x_1 : 0 : 1 :0]$, and the point $P = [0 : 0 : 1 : 1 : 0]$, corresponding to the cone $\sigma_{125}$. A neighbourhood of every point of $C$ in $Y$ is isomorphic to $\frac{1}{3}(1,1) \times \AA^1$.  One can see that a neighbourhood of $P$ in $Y$ is isomorphic to $\frac{1}{2}(1,1,1)$. The connected components of the inertia stack $\rI \cY$ are indexed by $\BBox(\Sigma) = \{ (0,0,0), (1,1,1), (2,2,2), (2,2,1) \}$, with ages $0, \frac{2}{3}, \frac{4}{3}, \frac{3}{2}$ respectively.

Let $\cX \into \cY$ be the hypersurface defined by a generic section of the line bundle $\cE$ on $\cY$ with $c_1(\cE) = 3 p_1 + 3 p_2$. Such a generic section is of the form
\begin{equation*}
f(x_0, x_1, y, z , t) = f_3(x_0, x_1) t^3 + f_2(x_0, x_1) (ayt^2 + bz) + f_1(x_0, x_1) y^2 t + cy^3,
\end{equation*}
where $a,b,c \in \CC$ and $f_i(x_0, x_1)$ denotes a homogeneous polynomial of degree $i$ in the variables $x_0, x_1$. Since $f(0,0,1,1,0) = c$, we see that a generic choice for $f$ implies $P \notin X$. Moreover, since $f(x_0, x_1, 0, 1, 0) = f_2(x_0, x_1) b$, we see that the surface $X$ intersects the curve $C$ in two points. For each of these two points there is a neighbourhood in $X$ that is analytically isomorphic to $\frac{1}{3}(1,1)$.

By adjunction $-K_X = (-K_Y - E) \vert_X = (4p_1 + 5p_2 - 3p_1 - 3p_2) \vert_X = (p_1 + 2p_2) \vert_X$, which is ample. Therefore $X$ is a del Pezzo surface with two singular points of type $\frac{1}{3}(1,1)$. Using the relations $(p_1 + 3p_2)p_2 =0$, $p_1^2(p_1+p_2) =0$ and $p_1^2 p_2 = \frac{1}{2}$ that hold in $\rH^\bullet(Y, \QQ)$, one can show that the degree of $X$ is $K_X^2 = (p_1 + 2p_2)^2(3 p_1 + 3 p_2) = \frac{8}{3}$. Hence $X = X_{2, 8/3}$.

Now we `extend' with the vector $(1,1,1)$. In other words, we consider the map $S = \{ 1 \} \to N$ with $s_1 = (1,1,1) \in N$ and the corresponding stacky fan, which is the one with extended ray map \eqref{eq:extended_fan_map_general}
\begin{equation*}
\rho^S = \begin{bmatrix}
               1 & 0 & 0 & -1 & 3 & 1 \\
               0 & 1 & 0 & -1 & 3 & 1 \\
               0 & 0 & 1 & -1 & 2 & 1 \\
              \end{bmatrix}.
\end{equation*} 
A basis of $\LL^S = \ker (\rho^S \colon \ZZ^6 \to N)$ is given by the rows of the matrix
\begin{equation*}
D^S = \begin{bmatrix}
             1 & 1 & 1 & 1 & 0 & 0 \\
             0 & 0 & 1 & 3 & 1 & 0 \\
             0 & 0 & 0 & 1 & 0 & 1 \\
            \end{bmatrix}.
\end{equation*}
We use this basis to identify $\LL^S$ with $\ZZ^3$. We call $l_1, l_2, k$ the coordinates
with respect to this basis; thus, the inclusion $\LL^S \otimes_\ZZ \RR \hookrightarrow \RR^6$ is given by
\begin{eqnarray*}
 (l_1,l_2,k) & \mapsto & \begin{pmatrix}
                                  l_1 \\ l_1 \\ l_1+l_2 \\l_1+3l_2+k \\ l_2 \\ k \\
                                 \end{pmatrix}.
\end{eqnarray*}
The extended fan sequence \eqref{eq:extended_fan_sequence_general} is 
\begin{equation*}
0 \longrightarrow \LL^S \simeq \ZZ^3 \overset{^t \! D^S}\longrightarrow \ZZ^6 \overset{\rho^S} \longrightarrow N = \ZZ^3 \longrightarrow 0
\end{equation*}
and the extended divisor sequence \eqref{eq:extended_divisor_sequence_general} is
\begin{equation*}
0 \longrightarrow M = \ZZ^3 \overset{^t \! \rho^S} \longrightarrow \ZZ^6 \overset{D^S}\longrightarrow \LL^{S \vee} \longrightarrow 0.
\end{equation*}
The extended nef cone is 
\begin{equation*}
\Nef^S(\cY) = \mathrm{cone}\left\langle\begin{bmatrix}
                           1 \\ 3 \\ 1 \\
                          \end{bmatrix},
                          \begin{bmatrix}
                           3 \\ 3 \\ 1 \\
                          \end{bmatrix},
                          \begin{bmatrix}
                           0 \\ 0 \\ 1 \\
                          \end{bmatrix}
 \right\rangle.
\end{equation*}
One can check that
\begin{equation*}
\Lambda\rE^S=\left\{\begin{pmatrix}
                              l_1 \\ l_1 \\ l_1+l_2 \\l_1+3l_2+k \\ l_2 \\ k \\
                             \end{pmatrix} \in \QQ^6
 \right|\left. \begin{matrix} l_1 + 3 l_2 + k \geq 0, 3 l_1 + 3 l_2 + k \geq 0, \\
  k \in \NN,  \\ 
  (l_1 \in \ZZ, 3l_2\in\ZZ) ~{\rm or}~ (l_1+l_2 \in \ZZ, 2l_2 \in \ZZ) \\ \\ \end{matrix}\right\}.
\end{equation*}
The extended reduction function $v^S \colon \Lambda\rE^S \rightarrow \BBox(\Sigma)$ is given by 
\begin{align*}
 v^S(l_1,l_2,k) &= \lceil l_1\rceil \begin{pmatrix}
                                           1 \\ 0 \\ 0 \\
                                          \end{pmatrix}+
                        \lceil l_1\rceil \begin{pmatrix}
                                          0 \\ 1 \\ 0 \\
                                         \end{pmatrix}+
                        \lceil l_1+l_2 \rceil \begin{pmatrix}
                                               0 \\ 0 \\ 1 \\
                                              \end{pmatrix}+ \\
                      & \qquad \qquad + \lceil l_1+3l_2+k\rceil \begin{pmatrix}
                                                 -1 \\ -1 \\ -1 \\
                                                \end{pmatrix}+
                        \lceil l_2 \rceil \begin{pmatrix}
                                            3 \\ 3 \\ 2 \\
                                          \end{pmatrix}+
                        \lceil k\rceil \begin{pmatrix}
                                        1 \\ 1 \\ 1 \\
                                       \end{pmatrix} \\
                     &= \begin{pmatrix}
                         \lceil l_1\rceil -\lceil l_1+3l_2\rceil +3\lceil l_2\rceil \\
                         \lceil l_1\rceil -\lceil l_1+3l_2\rceil +3\lceil l_2\rceil \\
                         \lceil l_1+l_2\rceil -\lceil l_1+3l_2\rceil +2\lceil l_2\rceil \\
                        \end{pmatrix}.
\end{align*}
We get:
\begin{itemize}
 \item $v^S(l_1, l_2, k) = (0,0,0)$ if $l_1, l_2 \in \ZZ$,
 \item $v^S(l_1, l_2, k) = (2,2,2)$ if $l_1 \in \ZZ$, $l_2 \in \frac{1}{3} + \ZZ$,
 \item $v^S(l_1, l_2, k) = (1,1,1)$ if $l_1 \in \ZZ$, $l_2 \in \frac{2}{3} + \ZZ$
 \item $v^S(l_1, l_2, k) = (2,2,1)$ if $l_1, l_2 \in \frac{1}{2} + \ZZ$.
\end{itemize}
Since $s_1 = \frac{1}{3} \rho_3 + \frac{1}{3} \rho_5$, the image of $(l_1, l_2, k) \in \LL^S \otimes_\ZZ \QQ$ in $\LL \otimes_\ZZ \QQ$ via the splitting \eqref{eq:isomorphism_splitting} is
\begin{equation}
d = \begin{bmatrix}
l_1 \\
l_1 \\
l_1 + l_2 + \frac{k}{3}\\
l_1 + 3 l_2 + k \\
l_2 + \frac{k}{3}
\end{bmatrix}
=
l_1 \begin{bmatrix}
1 \\ 1 \\ 1 \\ 1 \\ 0
\end{bmatrix}
+ \left( l_2 + \frac{k}{3} \right) \begin{bmatrix}
0 \\ 0 \\1 \\ 3 \\ 1
\end{bmatrix}.
\end{equation}
Therefore for $\lambda = (l_1, l_2, k) \in \Lambda \rE^S$ we have $\tilde{Q}^\lambda = Q^{(l_1, l_2 + \frac{k}{3})} \xi^k$.

We take $\epsi = (3,3,1)$ as a lifting of the line bundle $E = 3p_1 + 3p_2 \in \LL^\vee$ to $\LL^{S \vee}$. 
If we denote by $\kappa$ the equivariant parameter, the $S$-extended $\cE$-twisted I-function \eqref{eq:twisted_I_function} is
\begin{align*}
I_\cE^S (\tau_1, \tau_2, \xi; z) &= z \rme^{(\tau_1 p_1 + \tau_2 p_2)/z} \sum_{(l_1, l_2, k) \in \Lambda \rE^S} 
\tilde{Q}^{(l_1, l_2,k)} \rme^{\tau_1 l_1 + \tau_2 \left( l_2 + \frac{k}{3} \right)} ~~ \times \\
&\times~~ \left(\ambientfactor{l_1}{p_1}\right)^2 
\ambientfactor{l_1+l_2}{p_1+p_2} ~~\times \\
&\times~~ \ambientfactor{l_1+3l_2+k}{p_1+3p_2} 
\ambientfactor{l_2}{p_2} \frac{1}{k! z^k} ~~\times \\
& \times~~ \prod_{\substack{0\leq a\leq 3l_1+3l_2+k \\ \langle a\rangle=\langle 3l_1+3l_2+k\rangle}}(3p_1+3p_2+ \kappa +az) \mathbf{1}_{v^S(l_1,l_2,k)}.
\end{align*}
The degree of the summand corresponding to $\lambda \in \Lambda \rE^S$ with respect to $z$ is not smaller than $-1$ if and only if $\lambda \in \left\{ (0,0,0), (1,0,0), (0,0,1), (0, - \frac{1}{3}, 1), (1, - \frac{1}{3}, 0) \right\}$. Therefore\begin{align*}
I_\cE^S(\tau_1, \tau_2, \xi; z)= & z \bfone_0+ \tau_1 p_1 + \tau_2 p_2 + 
\left( 6 Q^{(1,0)} \rme^{\tau_1} + Q^{(0,\frac{1}{3})} \xi \rme^{\frac{\tau_2}{3}} \right) \bfone_0+ \\
&+ \left( \xi +3 Q^{(1,- \frac{1}{3})} \rme^{\tau_1 - \frac{\tau_2}{3}} \right) \bfone_{(1,1,1)}+O(z \inv).
\end{align*}
By the mirror theorem, the $\cE$-twisted J-function of $\cY$ is such that 
\begin{equation*}
J^\mathrm{tw}_\cY \left( \left( \xi +3 Q^{(1,- \frac{1}{3})} \right) \bfone_{(1,1,1)}, z \right) = \exp\left( - z \inv \left( 6 Q^{(1,0)} + Q^{(0,\frac{1}{3})} \xi \right) \right) \cdot I^S_\cE(0,0, \xi ;z).
\end{equation*}
By Conjecture \ref{conj:tom_quantum_lefschetz},
\begin{equation*}
I_{\cX,\cY} \left( \! \left( \xi +3 Q^{(1,- \frac{1}{3})} \right) \iota^\star \bfone_{(1,1,1)}, z \right) := \lim_{\kappa \to 0} \iota^\star J^\mathrm{tw}_\cY \left( \left( \xi +3 Q^{(1,- \frac{1}{3})} \right) \bfone_{(1,1,1)}, z \right)
\end{equation*}
lies on the Givental cone of $\cX$. Therefore
\begin{align*}
I_{\cX,\cY} \left( \eta \iota^\star \bfone_{(1,1,1)}, z \right) &= \exp\left( - z \inv \left( 3 Q^{(1,0)} + Q^{(0,\frac{1}{3})} \eta \right) \! \right) \times \\
& \times \lim_{\kappa \to 0} I^S_\cE \left( 0,0, \eta - 3 Q^{(1, - \frac{1}{3})} ;z \right)
\end{align*}

Let $\bfone_1$, $\bfone_2$ denote the two identity classes of the components of $\rI \cX$ with age equal to $2/3$. Now we compute a specialisation of the quantum period $G_X(x_1, x_2; t) \in \QQ [ x_1, x_2 ] [ \! [ t ] \! ]$ of $X$.
Since $\iota^\star \bfone_{(1,1,1)} = \bfone_1 + \bfone_2$, setting $z = 1$ and $\eta = x t^{\frac{1}{3}}$, replacing $Q^{(\al_1 , \al_2)} \mapsto t^{\al_1 + 2 \al_2}$ (and consequently $\tilde{Q}^{(l_1, l_2, k)} \mapsto t^{l_1 + 2 l_2 + k} (x-3)^k$), and considering the component along $\bfone_0$ only, we get
\begin{equation*}
G_X (x, x ; t) = \rme^{-xt - 3t} \! \sum_{\substack{l_1,l_2,k\in\NN}}
\frac{(3l_1+3l_2+k)!}{(l_1!)^2l_2!(l_1+l_2)!(l_1+3l_2+k)!k!} (x-3)^k t^{l_1+2l_2+k}.
\end{equation*}
Since the two singular points of $X$ lie in the same component of the singular locus of $Y$, we cannot distinguish the two points, but we are able to compute $G_X(x_1, x_2; t)$ only for $x_1 = x_2 $. It is possible that if we had used another model of $X$ as a complete intersection in a toric orbifold, we could have been able to compute the whole quantum period of $X$.

The regularised quantum period is such that
\begin{align*}
\widehat{G}_X(x,x;t) &= 1 + (12 x + 20)t^2 + (6x^2 + 108x + 168) t^3 \\
&+ (396 x^2 + 1800 x + 2220) t^4 \\
&+ (360 x^3 + 7980 x^2 + 26640 x + 27600) t^5
+ \cdots
\end{align*}
and matches with the classical period of a $1$-dimensional subspace of the $2$-parameter family of maximally mutable Laurent 
polynomials associated to the polygon n.13.


\subsection{Another example of a toric complete intersection: $B_{1,16/3}$.} \label{subsec:example_3}

Let $\cX$ be a general quartic in $\cY = \PP(1,1,1,3)$.
In this example we apply the Quantum Lefschetz technique, as in the \S\ref{subsec:example_2}, to compute the quantum period of $\cX$. Nevertheless, here it is crucial to use Conjecture \ref{conj:tom_quantum_lefschetz} by applying $\iota^\star$ firstly and then considering the limit for $\kappa \to 0$.

The reason is that, since the toric ambient $\cY$ is `extended weak Fano', it is impossible to choose a lifting of $\cE = \cO_\cY(4)$ to the extended Picard group in such a way that the extended twisted I-function $I^S_\cE$ has both a good asymptotic behaviour with respect to $z$ and a well-defined non-equivariant limit for $\kappa \to 0$.
So we will choose a lifting of $\cE$ such that $I^S_\cE$ has a good asymptotic behaviour, but $\lim_{\kappa \to 0} I^S_\cE$ does not exist. Fortunately, even though $I^S_\cE$ does not have a well-defined limit as $\kappa \to 0$,   $\iota^\star I^S_\cE$ does:  $\iota^\star I^S_\cE \to I_{\cX,\cY}$ as $\kappa \to 0$. Having a good asymptotic behaviour, $I_{\cX,\cY}$ gives information about $J_\cX$. We will be more precise below.

It is easy to see that $[0:0:0:1]$ is the unique singular point of $Y$ and is of type $\frac{1}{3}(1,1,1)$. The inertia stack $\rI \cY$ has three connected components: one isomorphic
to $\cY$ and two non-trivial components which are both isomorphic to $\rB \bmu_3$.  Since $-K_X = \cO_X(2)$, $X$ is a del Pezzo surface with Fano index $2$ and degree $K_X^2 = 2 \cdot 2 \cdot 4 \cdot \frac{1}{3} = \frac{16}{3}$. Moreover $[0:0:0:1]$ is the unique singular point of $X$ and is of type $\frac{1}{3}(1,1)$. Therefore $X$ has been called $B_{1, 16/3}$ in \cite{alessio_liana}.

The fan sequence \eqref{eq:fan_sequence_general} of $\cY$ is
$$
0 \longrightarrow \LL \simeq \ZZ \longrightarrow \ZZ^4 \overset{\rho}\longrightarrow \ZZ^3 = N \longrightarrow 0,
$$
where
$$\rho=\begin{bmatrix}
              -1 & 1 & 0 & 0 \\
              -1 & 0 & 1 & 0 \\
              -3 & 0 & 0 & 1 \\
             \end{bmatrix}.
$$
The divisor matrix is simply 
$D=\begin{bmatrix} 1 & 1 & 1 & 3\end{bmatrix}$. We use the transpose of $D$ as a basis of $\LL$. One can check that $\BBox(\cY) = \left\{ (0,0,0), (0,0,-1), (0,0,-2) \right\}$, with ages $0, 1, 2$ respectively.

Now we extend with $(0,0,-1)$. The extended fan sequence \eqref{eq:extended_fan_sequence_general} is 
\begin{equation*}
0 \longrightarrow \LL^S \simeq \ZZ^2 \overset{^t \! D^S}\longrightarrow \ZZ^5 \overset{\rho^S}\longrightarrow \ZZ^3 = N \longrightarrow 0,
\end{equation*}
where
\begin{equation*}
\rho^S = \begin{bmatrix}
                 -1 & 1 & 0 & 0 & 0 \\
                 -1 & 0 & 1 & 0 & 0 \\
                 -3 & 0 & 0 & 1 & -1
                \end{bmatrix}
\end{equation*}
and
\begin{equation*}
D^S=\begin{bmatrix}
                 1 & 1 & 1 & 3 & 0 \\
                 0 & 0 & 0 & 1 & 1 \\
                \end{bmatrix}.
\end{equation*}
The extended nef cone is 
$$
\Nef^S(\cY)= \mathrm{cone} \left\langle \begin{pmatrix}
                                       3 \\ 1 \\
                                      \end{pmatrix},\begin{pmatrix}
                                      0 \\ 1 \\
                                      \end{pmatrix} \right\rangle.
$$
Let us use coordinates $(l,k)$ on $\LL^S$ given by the basis made up of the rows of $D^S$. One can check that
$$
\Lambda\rE^S(\cY) = \left\{ \left. \begin{pmatrix} l \\ l \\ l \\ 3l+k \\ k\end{pmatrix} \in \QQ^5 \right|
\begin{matrix} l \in \frac{1}{3}\ZZ, ~ k \in \ZZ ~~\text{ s.t. } \\ 3l +k \geq 0,~k\geq 0 
 \end{matrix} \right\}.
$$
The reduction function $v^S \colon \Lambda \rE^S \to \BBox(\cY)$ is given by $v^S(l,k) = 3 \langle - l \rangle (0,0,-1)$. Since $s_1 = \frac{1}{3} (\rho_1 + \rho_2 + \rho_3)$, the projection in $\LL \otimes_\ZZ \QQ$ of $(l,k) \in \LL^S \otimes_\ZZ \QQ$ via the splitting \eqref{eq:isomorphism_splitting} is $d = l + \frac{k}{3}$.

Let $p$ be the first Chern class of $\cO_\cY(1)$.
In order to write down a twisted $S$-extended I-function, we have to choose a lifting $(4,\al)$ of $4p \in \LL^\vee$ to $\LL^{S \vee}$. One can check that $I^S_\cE$ has a good asymptotic behaviour if and only if $(2, 2-\al) \in \Nef^S(\cY)$, i.e. $\al \leq 1$. On the other hand, $\lim_{\kappa \to 0} I^S_\cE$ exists if and only if $(4, \al) \in \Nef^S(\cY)$, i.e. $\al \geq 2$. Therefore, it is impossible to find an $\al \in \ZZ$ such that $I^S_\cE$ has a good asymptotic behaviour and that the non-equivariant limit of $I^S_\cE$ exists. This is related to the fact that the extended anticanonical class $(6,2)$ is not in  the interior of $\Nef^S(\cY)$, i.e. $\cY$ is not `extended Fano', but only `extended weak Fano'.

Now we fix $\al = 1$. Consider the summand 
\begin{equation*}
\square_{l,k} =\left(\ambientfactor{l}{p}\right)^{\! \! 3} \ambientfactorwithoutnumerator{3l+k}{3p}
\frac{1}{k! z^k} \bundlefactorwithdenominator{4l+k}{4p + \kappa}
\end{equation*}
of $I^S_\cE$ corresponding to $(l,k) \in \Lambda \rE^S$.
We see that the degree of $\square_{l,k}$ with respect to $z$ is the following:
\begin{align*}
\deg \square_{l,k} &= -3 \left(
\begin{cases}
 \lceil l\rceil+1 & \text{ if }l\in\ZZ,~l<0 \\
 \lceil l\rceil & \text{ if }l\notin\ZZ \text{ or }l\geq 0 \\
\end{cases} \right) -\lceil 3l+k\rceil-k+\\
&+ \left( \begin{cases}
  \lceil 4l+ k \rceil +1 & \text{ if } 4l+ k\in\ZZ, 4l+ k< 0 \\
  \lceil 4l+ k\rceil & \text{ if }4l+ k\notin\ZZ \text{ or } 4l+ k\geq 0  \\
 \end{cases} \right) \\
 &= \begin{cases}
  -2\lceil l\rceil -k & \text{ if } l \notin \ZZ \text{ or } l \geq 0; \\
  -2 l  -k-3 & \text{ if }l \in \ZZ, - \frac{k}{4} \leq l < 0; \\
  -2l -k-2 & \text{ if }l \in \ZZ, l < - \frac{k}{4}. \\
 \end{cases}
\end{align*}
It is easy to show that $(-\frac{1}{3}, 1)$ is the only $(l,k) \in \Lambda \rE^S$ such that $\deg \square_{l,k} \geq -1$. 
So the twisted I-function of $\cY$ has the following asymptotic behaviour:
\begin{align*}
 I^S_\cE(\tau,\xi, z)=z{\bf 1}_0+ \tau p+ \tilde{Q}^{(0,1)}{\bf 1}_0+ \xi {\bf 1}_{(0,0,-1)}
 +O(z \inv).
\end{align*}
Since the lifting we have chosen is not in $\Nef^S$, it follows that the non-equivariant limit of $I^S_\cE$
does not exist. 

However, we can study the pull-back $\iota^\star(I^S_\cE)$ more carefully. 
The terms $\square_{l,k}$ that are divisible by $\kappa \inv$, namely the ones that make the limit do not exist,
correspond to $(l,k)$ such that $4l+k\in\ZZ_{<0}$;
in these cases we have that $\square_{l,k}$ is divisible by $p^3$ and then, since $\cX$ is a surface, $\iota^\star(\square_{l,k})=0$. Therefore
the limit
$$
I_{\cX,\cY} := \lim_{\kappa \to 0} \iota^\star I^S_\cE
$$
exists and, according to Conjecture \ref{conj:tom_quantum_lefschetz}, lies in $\cL_\cX$.

Thus, the J-function of $\cX$ is such that
$$
J_{\cX}(\tau p + \xi {\bf 1}_{1/3};z)=\exp\left( -\frac{ \tilde{Q}^{(0,1)}{\bf 1}_0}{z}\right)
I_{\cX,\cY}.
$$

After applying the change of variables $\tilde{Q}^{(l,k)} \mapsto x^k t^{2l+k}$, we get
$$
 G_X(x;t)=\exp(-xt)\sum_{l,k\in\NN}\frac{(4l+k)!}{l!^3 (3l+k)! k!}x^kt^{2l+k}.
$$
The regularised quantum period is
$$
 \widehat{G}_X(x;t)=1+8t^2+6xt^3+168t^4+240xt^5+(4440+90x)t^6+9240xt^7+\ldots.
$$
and matches with the classical period of a $1$-dimensional subspace of the $2$-parameter family of maximally mutable Laurent
polynomials associated to the polygon n.12.

\section{The Abelian/non-Abelian Correspondence}
\label{sec:abelian_nonabelian}

\subsection{Theoretical background}
\label{sec:setup_abelian}

Let $G$ be a reductive group over $\CC$ acting on a smooth affine variety $A$. Let $T$ be a maximal torus in $G$. We consider the stack-theoretic GIT quotients $[A \git G]$ and $[A \git T]$. Let $E$ be a representation of $G$ and let $\cE_G$ and $\cE_T$ be the induced vector bundles on $[A \git G]$ and $[A \git T]$, respectively. We assume that $[ A \git G]$ and $[A \git T]$ are proper Deligne--Mumford stacks with projective coarse moduli spaces. Moreover, we assume that there are no strictly semi-stable points and the unstable locus has codimension at least $2$, for both the actions of $G$ and $T$.

The Abelian/non-Abelian Correspondence of Bertram, Ciocan-Fontanine, Kim and Sabbah \cite{betram_cf_kim_abelian,abelian_nonabelian} relates\footnote{The results of \cite{abelian_nonabelian} have a projective hypothesis on $A$, but their arguments apply verbatim to the case where $A$ is affine, as here.} genus-zero Gromov--Witten invariants of $[A \git G]$, twisted by $\cE_G$, to the Gromov--Witten invariants of $[A \git T]$, twisted by $\cE_T$. We will be more precise below.

Let $W = \rN(T) / T$ be the Weyl group and $\Phi = \Phi_+ \cup \Phi_-$ be the root system with decomposition into positive and negative roots. The adjoint $T$-representation $\frakg$ splits as $\frakg = \frakt \oplus \bigoplus_{\al \in \Phi} \frakg_\al$. For every $\al \in \Phi$, the one-dimensional $T$-representation $\frakg_\al$ induces a line bundle $L_\al$ on $[A \git T]$. Let $p_\al = c_1(L_\al)$. 
Consider the cohomology class
\begin{equation*}
\omega := \prod_{\al \in \Phi_+} p_\al.
\end{equation*}
It is the fundamental $W$-anti-invariant class in the cohomology of $[A \git T]$. We recall that the $W$-invariant part of the cohomology of $[A \git T]$ may be identified with the cohomology of $[A \git G]$.

It is easy to see that there are homomorphisms $\Pic([A \git G]) \into \Pic([A \git T])$ and $\rho \colon \mathrm{Eff}([A \git T]) \onto \mathrm{Eff}([A \git G])$. The homomorphism $\epsi \colon \mathrm{Eff}([A \git G]) \to \QQ$ sends a curve class $\beta$ into $\sum_{\al \in \Phi_+} L_\al \cdot \tilde{\beta}$, where $\tilde{\beta} \in \mathrm{Eff}([A \git T])$ is a preimage of $\beta$. We consider the homomorphism on the Novikov rings $p \colon \boldsymbol{\Lambda}_{[A \git T]} \to \boldsymbol{\Lambda}_{[A \git G]}$ defined by $p(Q^d) = (-1)^{\epsi(\rho(d))} Q^{\rho(d)}$, for every $d \in \mathrm{Eff}([A \git T])$.\footnote{This actually depends on the choice of an $m$th root of $-1$, where $m$ is the least common multiple of the exponents of the automorphism group of geometric points of $[A \git G]$.}

\begin{conjecture}[Abelian/non-Abelian Correspondence] \label{conj:abelian}
Let $J^{\cE_G}$ and $J^{\cE_T}$ be the J-functions for the corresponding twisted Gromov--Witten theories of $[A \git G]$ and $[A \git T]$ (as in \S\ref{sec:quantum_Lef_principle}), respectively. Consider the differential operator $\cD = z \partial_\omega$. Let $\tilde{J}$ be the $W$-invariant function such that $\cD J^{\cE_T} = \omega \cup \tilde{J}$.

Then $\tilde{J}$ coincides with $J^{\cE_G}$, after:
\begin{itemize}
\item identifying the $W$-invariant part of the cohomology of $[A \git T]$ with the cohomology of $[A \git G]$;
\item applying the homomorphism $p$ on the Novikov ring of $[A \git T]$;
\item applying a suitable mirror map $\phi$ on the parameters:
\begin{equation*}
\left. \cD J^{\cE_T} (\gamma; z) \right\vert_{Q^d \mapsto p(Q^d)} = \omega \cup \tilde{J}(\phi(\gamma); z).
\end{equation*}
\end{itemize}
\end{conjecture}

In \cite[Theorem 6.1.2]{abelian_nonabelian} Ciocan-Fontanine, Kim and Sabbah state Conjecture \ref{conj:abelian} under the additional assumption that $[A \git T]$ and $[A \git G]$ are smooth varieties and show that it is a consequence of a conjecture about Frobenius structures. Moreover, in \cite[Theorem 4.1.1]{abelian_nonabelian} they show that Conjecture \ref{conj:abelian} holds when $[A \git G]$ is a flag manifold.

In \S\ref{subsec:sample_abelian_nonabelian} we show how to use Conjecture \ref{conj:abelian} to compute the quantum period of a del Pezzo surface.

\subsection{Example of orbifold Abelian/non-Abelian correspondence: $X_{1,7/3}$} \label{subsec:sample_abelian_nonabelian}

Let $A$ be the space of $2 \times 5$ matrices, which are denoted by
\begin{equation*}
\begin{bmatrix}
a_1 & a_2 & a_3 & a_4 & a_5 \\
b_1 & b_2 & b_3 & b_4 & b_5
\end{bmatrix}.
\end{equation*}
The group $\rS \rL_2$ acts on $A$ via left multiplication and the group $\GG_\rmm$ acts on $A$ via
\begin{equation*}
\mu \cdot \begin{bmatrix}
a_1 & a_2 & a_3 & a_4 & a_5 \\
b_1 & b_2 & b_3 & b_4 & b_5
\end{bmatrix} =
\begin{bmatrix}
\mu a_1 & \mu a_2 & \mu a_3 & \mu^3 a_4 & \mu^3 a_5 \\
\mu b_1 & \mu b_2 & \mu b_3 & \mu^3 b_4 & \mu^3 b_5
\end{bmatrix}.
\end{equation*}
We get an action of $\rS \rL_2 \times \GG_\rmm$ on $A$, which induces a faithful action of the affine reductive group
\begin{equation*}
G := \frac{\rS \rL_2 \times \GG_\rmm}{
\left\{ \left. \left( \! \! \begin{pmatrix}
\lambda & 0 \\ 0 & \lambda
\end{pmatrix}, \lambda \right) \right\vert \lambda \in \bmu_2 \right\}
}
\end{equation*}
on $A$. Following \cite[Example 2.6]{weighted_grassmannians}, the stack-theoretic GIT quotient $\cF := [A \git G]$ is the weighted Grassmannian $\mathrm{wGr}(2,5)$ with weights $\frac
1 2, \frac 1 2, \frac 1 2, \frac 3 2, \frac 3 2$. By using
\begin{equation*}
c_{ij} = \det \begin{pmatrix}
a_i & a_j \\ b_i & b_j
\end{pmatrix}, \qquad 1 \leq i < j \leq 5
\end{equation*}
as coordinates, we get a closed embedding of $\cF$ into the weighted projective space $\PP = \PP(1^3, 2^6,3)$. The pulling-back homomorphism $\ZZ \simeq \mathrm{Pic}(\PP) \to \mathrm{Pic}(\cF)$ maps $\cO_\PP(1)$ into the line bundle $\cO_\cF(1)$ on $\cF$ associated to the character of $G$ induced by the composite
\begin{equation*}
\rS \rL_2 \times \GG_\rmm \overset{\mathrm{pr_2}}{\longrightarrow} \GG_\rmm \overset{( \cdot)^2}{\longrightarrow} \GG_\rmm.
\end{equation*}
Let $\cX \into \cF$ be the zero locus of a generic section of $\cE_G = \cO_\cF(2)^{\oplus 4}$. The coarse moduli space $X$ of $\cX$ is a del Pezzo surface with one $\frac{1}{3}(1,1)$ and degree $K^2_X = \frac{7}{3}$ (see \cite{alessio_liana}). To compute the Gromov--Witten invariants of $\cX$ we need to compute the $\cE_G$-twisted Gromov--Witten invariants of $\cF$. This can be done by using Conjecture \ref{conj:abelian}.

The maximal subtorus of $G$ is
\begin{equation*}
T := \frac{
\left\{ \left. \left( \! \! \begin{pmatrix}
\lambda & 0 \\ 0 & \lambda \inv
\end{pmatrix}, \mu \right) \right\vert \lambda, \mu \in \GG_\rmm \right\}
}{
\left\{ \left. \left( \! \! \begin{pmatrix}
\lambda & 0 \\ 0 & \lambda
\end{pmatrix}, \lambda \right) \right\vert \lambda \in \bmu_2 \right\}
}.
\end{equation*}
Indeed, $T$ is isomorphic to $\GG_\rmm^2$ via 
\begin{equation}
\left( \! \! \begin{pmatrix}
\lambda & 0 \\ 0 & \lambda \inv
\end{pmatrix}, \mu \right) \mapsto (\lambda \mu, \lambda \inv \mu).
\end{equation}
Therefore the toric Fano orbifold $\cY := [A \git T]$ is the stack-theoretic GIT quotient of $A \simeq \AA^{10}$ with respect to the action of $\GG_\rmm^2$ given by the following matrix.
\begin{equation*}
\begin{matrix}
a_1 & a_2 & a_3 & a_4 & a_5 & b_1 & b_2 & b_3 & b_4 & b_5 \\
1 & 1 & 1 & 2 & 2 & 0 & 0 & 0 & 1 & 1 \\
0 & 0 & 0 & 1 & 1 & 1 & 1 & 1 & 2 & 2
\end{matrix}
\end{equation*}
Let $Y$ be the coarse moduli space of $\cY$.
We denote by $p_1, p_2 \in \rH^2(Y;\QQ)$ the first Chern classes of the line bundles on $\cY$ induced by the characters of $T$ given by
\begin{align*}
\left( \! \! \begin{pmatrix}
\lambda & 0 \\ 0 & \lambda \inv
\end{pmatrix}, \mu \right) &\mapsto \lambda \mu, \\
\left( \! \! \begin{pmatrix}
\lambda & 0 \\ 0 & \lambda \inv
\end{pmatrix}, \mu \right) &\mapsto \lambda \inv \mu.
\end{align*}
The nef cone of $Y$ is $\Nef(Y) = \mathrm{cone}\left\langle 2p_1 + p_2, p_1 + 2 p_2 \right\rangle$.

Let $\cO_\cY(1)$ be the line bundle on $\cY$ such that its restriction to $[A^\rms(G)/T]$ is the pull-back of $\cO_\cF(1) $ from $\cF = [A^\rms(G)/G]$. It corresponds to the character of $T$ defined by
\begin{equation*}
\left( \! \! \begin{pmatrix}
\lambda & 0 \\ 0 & \lambda \inv
\end{pmatrix}, \mu \right) \mapsto \mu^2.
\end{equation*}
Therefore $c_1(\cO_\cY(1)) = p_1 + p_2$, which is ample. Consider the vector bundle $\cE_T = \cO_\cY(2)^{\oplus 4}$.

Since $\cY$ is a toric orbifold, we may construct its $\cE_T$-twisted I-function:
\begin{equation*}
\begin{aligned}
I^{\cE_T} (\tau_1, \tau_2; z) &= z \rme^{(\tau_1 p_1 + \tau_2 p_2)/z} \sum_{(l_1, l_2) \in \Lambda \rE} 
Q^{(l_1, l_2)} \rme^{\tau_1 l_1 + \tau_2 l_2}  ~~ \times \\
&\times~~ \left(\ambientfactor{l_1}{p_1}\right)^3 
\left( \ambientfactor{2l_1+l_2}{2p_1+p_2} \right)^2~~\times \\
&\times~~ \left( \ambientfactor{l_2}{p_2} \right)^3
\left( \ambientfactor{l_1+2l_2}{p_1+2p_2} \right)^2  ~~\times \\
&\times \left( \bundlefactorwithoutdenominator{2l_1 + 2l_2}{2p_1 + 2p_2} \right)^4 \bfone_{v(l_1,l_2)}.
\end{aligned}
\end{equation*}
More compactly, we may write
\begin{equation*}
I^{\cE_T} = \sum_{(l_1, l_2) \in \Lambda \rE} z Q^{(l_1, l_2)} \mathrm{exp}\left( \! \left(l_1 + \frac{p_1}{z} \right) \tau_1 + \left(l_2 + \frac{p_2}{z} \right) \tau_2 \right) \square_{l_1, l_2} \bfone_{v(l_1, l_2)}.
\end{equation*}
One may prove that
\begin{align*}
I^{\cE_T} &= z \bfone_0 + 4 \left( Q^{(1,0)}\rme^{\tau_1} + Q^{(0,1)}\rme^{\tau_2} \right) \bfone_0 + \tau_1 p_1 + \tau_2 p_2 + \\
& + \frac{2}{3} \left( Q^{(-\frac{1}{3}, \frac{2}{3})}\rme^{\frac{-\tau_1+2\tau_2}{3}} + Q^{(\frac{2}{3}, -\frac{1}{3})} \rme^\frac{2 \tau_1 - \tau_2}{3} \right) \bfone_b + \rO(z \inv),
\end{align*}
for some component $b \in \BBox(\cY)$ with $\mathrm{age}(b) = 2$.
Therefore, the $\cE_T$-twisted J-function of $\cY$ is such that
\begin{equation} \label{eq:J_function_abelian_specialised}
J^{\cE_T} (\tau_1 p_1 + \tau_2 p_2 + \phiv(\tau_1, \tau_2) \bfone_b; z) = \exp \left( -4 \frac{Q^{(1,0)}\rme^{\tau_1} + Q^{(0,1)}\rme^{\tau_2}}{z} \right) \cdot I^{\cE_T},
\end{equation}
where
\begin{equation*}
\phiv(\tau_1, \tau_2) = \frac{2}{3} \left( Q^{(- \frac 1 3 , \frac 2 3)}\rme^{\frac{-\tau_1+2 \tau_2}{3}} + Q^{(\frac 2  3, - \frac 1 3)} \rme^\frac{2 \tau_1 - \tau_2}{3} \right).
\end{equation*}

The Weyl group $W = \rN(T)/T$ of $G$ with respect to $T$ is cyclic of order $2$ and is generated by the class of
\begin{equation*}
\left( \! \! \begin{pmatrix}
0 & 1 \\ -1 & 0
\end{pmatrix}, 1 \right).
\end{equation*}
The positive root $\al$ corresponds to the character of $T$ defined by
\begin{equation*}
\left( \! \! \begin{pmatrix}
\lambda & 0 \\ 0 & \lambda \inv
\end{pmatrix}, \mu \right) \mapsto \lambda^2.
\end{equation*}
This character of $T$ induces the line bundle $L_\al$ on $\cY$. Let $\omega = c_1(L_\al) = p_1 - p_2$.

Consider the differential operator $\cD = z \partial_{p_1 - p_2}$. Since $J^{\cE_T}$ is $W$-invariant, $\cD J^{\cE_T}$ is $W$-anti-invariant. Therefore it must be divisible by $\omega = p_1 - p_2$:
\begin{equation*}
\cD J^{\cE_T} = (p_1 - p_2) \cup \tilde{J}.
\end{equation*}
The Abelian/non-Abelian Correspondence (Conjecture \ref{conj:abelian}) relates $\tilde{J}$ with a lifting of the $\cE_G$-twisted J-function of $\cF$, up to the ring homomorphism on the Novikov rings $Q^{(l_1, l_2)} \mapsto (-q)^{l_1 + l_2}$.

 Unfortunately we do not know $J^{\cE_T}$, but by \eqref{eq:J_function_abelian_specialised} we only know $J^{\cE_T} \circ \vartheta$, where
$
\vartheta \colon (\tau_1, \tau_2) \mapsto \tau_1 p_1 + \tau_2 p_2 + \phiv(\tau_1, \tau_2) \bfone_b
$.
Now consider the differential operator $
\overline{\cD} = z (\partial_{\tau_1} - \partial_{\tau_2}).
$

By the chain rule
\begin{align*}
(\cD J^{\cE_T})(\vartheta(\tau_1, \tau_2)) = \overline{\cD} (J^{\cE_T} \circ \vartheta) - z \partial_{\bfone_b} J^{\cE_T}(\vartheta(\tau_1, \tau_2)) \cdot (\partial_{\tau_1} \varphi - \partial_{\tau_2} \varphi),
\end{align*}
where
\begin{align*}
\partial_{\tau_1} \phiv - \partial_{\tau_2} \phiv = \frac{2}{3} \left( - Q^{(- \frac 1 3 , \frac 2 3)}\rme^{\frac{-\tau_1+2 \tau_2}{3}} + Q^{(\frac 2  3, - \frac 1 3)} \rme^\frac{2 \tau_1 - \tau_2}{3} \right)
\end{align*}
From \eqref{eq:J_function_abelian_specialised} we get 
\begin{align*}
\overline{\cD} (J^{\cE_T} \circ \vartheta) &= \exp \left( -4 \frac{Q^{(1,0)}\rme^{\tau_1} + Q^{(0,1)}\rme^{\tau_2} }{z} \right) \times \\
&\times \left[  -4(Q^{(1,0)} \rme^{\tau_1} - Q^{(0,1)} \rme^{\tau_2} ) I^{\cE_T} + \overline{\cD} I^{\cE_T} \right],
\end{align*}
where
\begin{align*}
\overline{\cD} I^{\cE_T} &= \sum_{(l_1, l_2) \in \Lambda \rE} z Q^{(l_1, l_2)} \mathrm{exp}\left( \! \left(l_1 + \frac{p_1}{z} \right) \tau_1 + \left(l_2 + \frac{p_2}{z} \right) \tau_2 \right) \times \\
&\times \left( z l_1 + p_1 - z l_2 - p_2 \right) \square_{l_1, l_2} \bfone_{v(l_1, l_2)}.
\end{align*}

If we set $Q^{(l_1, l_2)} = (-q)^{l_1 + l_2}$, we get
\begin{align*}
\cD J^{\cE_T} \left( \vartheta(0,0) \right) &= \overline{\cD} (J^{\cE_T} \circ \vartheta) \vert_{\tau_1 =  \tau_2 = 0} \\
&= \rme^{8 q z \inv} \sum_{(l_1, l_2) \in \Lambda \rE} z (-q)^{l_1 + l_2} (z l_1 + p_1 - z l_2 - p_2) \square_{l_1, l_2} \bfone_{v(l_1, l_2)},
\end{align*}
whose asymptotic behaviour is 
\begin{equation*}
 \cD J^{\cE_T}  \left( \vartheta(0,0) \right) = (p_1-p_2) \left(z + 16q \bfone_0 + \frac{25}{9}(-q)^{\frac{1}{3}}\bfone_{b}+ O(z \inv)\right).
\end{equation*}

Hence, Conjecture \ref{conj:abelian} implies that a specialisation of $J^{\cE_G}$ coincides with 
$e^{-16q z \inv} \tilde{J}$, via the string equation. Its component along  the identity class $\bfone_0$ is
\begin{align*}
\rme^{-8q z \inv}  \sum_{l_1, l_2 \in \NN} A_{l_1, l_2}(q,z) \left(1 + \frac{l_1 - l_2}{2} \left( -3 H_{l_1} + 3 H_{l_2} - 2 H_{2 l_1 + l_2} + 2 H_{2 l_2 + l_1} \right) \! \right)
\end{align*}
where $H_l := \sum_{i = 1}^l \frac{1}{i}$ is the $l$th harmonic number ($H_0 := 0$) and 
\begin{equation*}
A_{l_1, l_2}(q,z) := (-q)^{l_1+l_2} \frac{(2l_1 + 2 l_2)!^4}{l_1!^3 l_2!^3 (2l_1 + l_2)!^2 (l_1 + 2 l_2)!^2 z^{l_1 + l_2 -1}}.
\end{equation*}
Therefore a specialisation of the quantum period of $\cX$ is 
\begin{align*}
G(t) &= \exp(-8t) \times \\
&\times \sum_{l_1, l_2 \in \NN} A_{l_1, l_2}(t,1) \left( 1 + \frac{l_1 - l_2}{2} \left( -3 H_{l_1} + 3 H_{l_2} - 2 H_{2 l_1 + l_2} + 2 H_{2 l_2 + l_1} \right) \right),
\end{align*}
whose regularization is
\begin{equation*}
\widehat{G}(t) = 1 + 112t^2 + 1650 t^3 + 48048 t^4 + \cdots
\end{equation*}
and matches with the classical period of a maximally mutable Laurent
polynomial associated to the polygon n.10.


\section{Results} \label{sec:results}

In this section we summarise our conclusions. In Table \ref{our_list}, for each family of del~Pezzo surfaces with $\frac{1}{3}(1,1)$ singularities we give:
\begin{itemize}
\item its name according to \cite{alessio_liana}, i.e. $X_{k,d}$ or $B_{k,d}$, where $k$ denotes the number of singular points and $d$ is the degree;
\item the method used to compute the quantum period;
\item the name of the corresponding mirror Fano polygon from the list in \cite{minimal_polygons};
\item the dimension of the subspace of the $k$-dimensional space of maximally mutable polynomials along which we matched the classical and quantum periods, as described in Theorem \ref{thm:main_result}.
\end{itemize}

 Later, we explicitly describe our results. For each family we give:
\begin{itemize}
\item model used for our computations, taken from \cite{alessio_liana}: $p_i$'s denote a basis of the Picard group either of the surface, when is toric, or the toric ambient space, when the surface is a complete intersection;
\item asymptotic behaviour of the (extended) $I$-function (see \S\ref{sec:mirror_theorem} and \S\ref{sec:Lefschetz_toric_orbifolds}). Here $\xi$'s, $\tau$'s and $z$ are formal variables. We use the notation $\tau p$ instead of $\tau_1 p_1 + \ldots + \tau_\ell p_\ell$ and we use ${\bf 1}_\al$ to express an identity class supported on a non-trivial component of the inertia stack;
\item (specialized) quantum period and its regularization, as explained in \S\ref{sec:quantum_period_Fano_orbifold};
\item mirror polygon from the list in \cite{minimal_polygons} and the corresponding Laurent polynomial with the specialisation required, if needed. If the specialisation is over a subspace of positive dimension, we use parameters $a_i$'s: in these cases, the matching betweend quantum period and classic period is attained by identifying the parameters $x_i$'s in the quantum period with the $a_i$'s in the classical period.
\end{itemize}

Our computations rely on the use of the computer. Indeed, we have implemented the machinery described in \S\ref{sec:stacky_fans}, \S\ref{sec:mirror_theorem}, \S\ref{sec:Lefschetz_toric_orbifolds} in the language of computer algebra software MAGMA \cite{magma}.

\begin{longtable}{lp{5.5cm}cc}
\caption{Summary of results.} \label{our_list} \\

\toprule 
\multicolumn{1}{c}{{\sc Name}}&
\multicolumn{1}{c}{{\sc Method}} & 
\multicolumn{1}{c}{{\sc Mirror}} & 
\multicolumn{1}{c}{{\sc Result}}
\\ \midrule 
\endfirsthead

\multicolumn{4}{c}{{\tablename\ \thetable{}: Summary of results -- continued from previous page}} \\ \addlinespace[1.7ex] \midrule
\multicolumn{1}{c}{{\sc Name}}&
\multicolumn{1}{c}{{\sc Method}} & 
\multicolumn{1}{c}{{\sc Mirror}} & 
\multicolumn{1}{c}{{\sc Result}}
\\ \midrule \endhead

\midrule \multicolumn{4}{c}{{Continued on next page}} \endfoot

\bottomrule \endlastfoot

\textbf{1.} \hyperref[subsec:P113]{$\PP(1,1,3)$} & Toric  & n.26 & $1$-dimensional \\ \addlinespace[1.3ex] 

\rowcolor[gray]{0.95} 
\textbf{2.} \hyperref[subsec:B_1,16/3]{$B_{1,16/3}$} & Quantum Lefschetz  & n.21 & $1$-dimensional \\ \addlinespace[1.3ex] 

\textbf{3.} \hyperref[subsec:B_2,8/3]{$B_{2,8/3}$} & Quantum Lefschetz  & n.12 & $1$-dimensional \\ \addlinespace[1.3ex] 

\rowcolor[gray]{0.95} 
\textbf{4.} \hyperref[subsec:X_1,22/3]{$X_{1,22/3}$} & Toric  & n.25 & $1$-dimensional \\ \addlinespace[1.3ex] 

\textbf{5.} \hyperref[subsec:X_1,19/3]{$X_{1,19/3}$} & Toric  & n.24 & $1$-dimensional \\ \addlinespace[1.3ex] 

\rowcolor[gray]{0.95} 
\textbf{6.} \hyperref[subsec:X_1,16/3]{$X_{1,16/3}$} & Quantum Lefschetz  & n.22 & $1$-dimensional \\ \addlinespace[1.3ex] 

\textbf{7.} \hyperref[subsec:X_1,13/3]{$X_{1,13/3}$} & Quantum Lefschetz  & n.18 & $0$-dimensional \\ \addlinespace[1.3ex] 

\rowcolor[gray]{0.95} 
\textbf{8.} \hyperref[subsec:X_1,10/3]{$X_{1,10/3}$} & Quantum Lefschetz & n.15 & $0$-dimensional \\ \addlinespace[1.3ex] 

\textbf{9.} \hyperref[subsec:X_1,7/3]{$X_{1,7/3}$} & Abelian/non-Abelian & n.10 &  $0$-dimensional \\ \addlinespace[1.3ex] 

\rowcolor[gray]{0.95} 
\textbf{10.} \hyperref[subsec:X_1,4/3]{$X_{1,4/3}$} & Quantum Lefschetz & n.4 & $0$-dimensional \\ \addlinespace[1.3ex] 

\textbf{11.} \hyperref[subsec:X_1,1/3]{$X_{1,1/3}$} & Quantum Lefschetz & n.1 & $0$-dimensional \\ \addlinespace[1.3ex] 

\rowcolor[gray]{0.95} 
\textbf{12.} \hyperref[subsec:X_2,17/3]{$X_{2,17/3}$} & Quantum Lefschetz & n.23 & $0$-dimensional \\ \addlinespace[1.3ex] 

\textbf{13.} \hyperref[subsec:X_2,14/3]{$X_{2,14/3}$} & Quantum Lefschetz & n.19 & $0$-dimensional \\ \addlinespace[1.3ex] 

\rowcolor[gray]{0.95} 
\textbf{14.} \hyperref[subsec:X_2,11/3]{$X_{2,11/3}$} & Quantum Lefschetz & n.16 & $0$-dimensional \\ \addlinespace[1.3ex] 

\textbf{15.} \hyperref[subsec:X_2,8/3]{$X_{2,8/3}$} & Quantum Lefschetz & n.13 & $1$-dimensional \\ \addlinespace[1.3ex] 

\rowcolor[gray]{0.95} 
\textbf{16.} \hyperref[subsec:X_2,5/3]{$X_{2,5/3}$} & Quantum Lefschetz & n.6 & $1$-dimensional \\ \addlinespace[1.3ex] 

\textbf{17.} \hyperref[subsec:X_2,2/3]{$X_{2,2/3}$} & Quantum Lefschetz & n.2 & $0$-dimensional \\ \addlinespace[1.3ex] 

\rowcolor[gray]{0.95} 
\textbf{18.} \hyperref[subsec:X_3,5]{$X_{3,5}$} & Toric & n.20 & $3$-dimensional \\ \addlinespace[1.3ex] 

\textbf{19.} \hyperref[subsec:X_3,4]{$X_{3,4}$} & Toric & n.17 & $3$-dimensional \\ \addlinespace[1.3ex] 

\rowcolor[gray]{0.95} 
\textbf{20.} \hyperref[subsec:X_3,3]{$X_{3,3}$} & Quantum Lefschetz & n.14 & $0$-dimensional \\ \addlinespace[1.3ex] 

\textbf{21.} \hyperref[subsec:X_3,2]{$X_{3,2}$} & Quantum Lefschetz & n.8 & $0$-dimensional \\ \addlinespace[1.3ex] 

\rowcolor[gray]{0.95} 
\textbf{22.} \hyperref[subsec:X_3,1]{$X_{3,1}$} & Quantum Lefschetz & n.3 & $0$-dimensional \\ \addlinespace[1.3ex] 

\textbf{23.} \hyperref[subsec:X_4,7/3]{$X_{4,7/3}$} & Quantum Lefschetz & n.11 & $2$-dimensional \\ \addlinespace[1.3ex] 

\rowcolor[gray]{0.95} 
\textbf{24.} \hyperref[subsec:X_4,4/3]{$X_{4,4/3}$} & Quantum Lefschetz & n.5 & $1$-dimensional \\ \addlinespace[1.3ex] 

\textbf{25.} \hyperref[subsec:X_4,1/3]{$X_{4,1/3}$} & Quantum Lefschetz & & no $\QQ$-Gor. toric deg. \\ \addlinespace[1.3ex] 

\rowcolor[gray]{0.95} 
\textbf{26.} \hyperref[subsec:X_5,5/3]{$X_{5,5/3}$} & missing good model & & \\ \addlinespace[1.3ex] 

\textbf{27.} \hyperref[subsec:X_5,2/3]{$X_{5,2/3}$} & missing good model & & no $\QQ$-Gor. toric deg. \\ \addlinespace[1.3ex] 

\rowcolor[gray]{0.95} 
\textbf{28.} \hyperref[subsec:X_6,2]{$X_{6,2}$} & Toric & n.9 & $0$-dimensional \\ \addlinespace[1.3ex] 

\textbf{29.} \hyperref[subsec:X_6,1]{$X_{6,1}$} & missing good model & & no $\QQ$-Gor. toric deg. \\ \addlinespace[1.3ex] 
\end{longtable}

\subsection{$\PP(1,1,3)$}

\begin{family}\label{subsec:P113}
		 {\sc Name:} $\PP(1,1,3)$
\end{family}
	{\sc Model:} toric
	$$\begin{matrix}
	   1 & 1 & 3 & & p\\
	  \end{matrix}$$

	{\sc Extended I-function:} $$I^{S}(\tau,\xi,z)=z+\tau p + \xi \bfone_\al + O(z\inv);$$

	{\sc Quantum period:} $$G_{\PP(1,1,3)}(x;t)=\sum_{l, k \geq 0} \frac{1}{l!^2 (3 l + k)! k!} x^k t^{5 l+2 k};$$

	{\sc Regularization:} $$\widehat{G}_{\PP(1,1,3)}(x;t)=1 + 2 x t^2 + 6 x^2 t^4 + 20 t^5 + 20 x^3 t^6 + 210 x t^7+
	\ldots$$

\noindent\begin{minipage}{10.5cm}
	{\sc Mirror:} (with parameter $a_{[-1,1]}=a$) $$f(a;x,y)=\frac{y^2}{x}+\frac{y}{x^2}+\frac{x}{y}+a\frac{y}{x}.$$
\end{minipage}
\begin{minipage}{2cm}
 \includegraphics[scale=0.6]{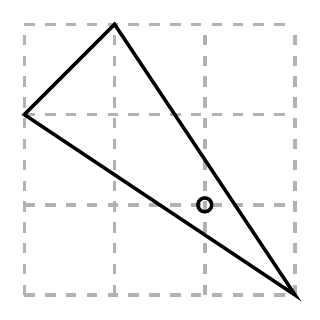}
 \begin{center} Polygon n.26 \end{center}
\end{minipage}

\subsection{Surfaces with Fano index 2}

\begin{family}\label{subsec:B_1,16/3}
		 {\sc Name:} $B_{1,16/3}$
\end{family}
	{\sc Model:} degree $4$ hypersurface in $\PP(1,1,1,3)$ (see \S\ref{subsec:example_3})
		$$\left.\begin{matrix}
			1 & 1 & 1 & 3 & 
		 \end{matrix}\right|
		\begin{matrix}
			& 4 & & p
		\end{matrix}$$

	{\sc Extended twisted I-function:}\begin{align*}
 I^S_\cE(\tau,\xi, z)=z{\bf 1}_0+ \tau p+ q^{\frac 1 3} \xi {\bf 1}_0+ \xi {\bf 1}_\al
 +O(z \inv).
\end{align*}

	{\sc Quantum period:}  $$G_X(x;t)=\exp(-xt)\sum_{l,k\in\NN}\frac{(4l+k)!}{l!^3 (3l+k)! k!}x^kt^{2l+k}$$

	{\sc Regularization:} $$
 \widehat{G}_{B_{1, 16/3}}(x;t)=1+8t^2+6xt^3+168t^4+240xt^5+(4440+90x)t^6+ \ldots.
$$
\noindent\begin{minipage}{10.5cm}

	{\sc Mirror:}  $$f(0;x,y)=\frac{(1+y)^4}{xy^2}+yx+ay.$$
\end{minipage}
\begin{minipage}{2cm}
 \includegraphics[scale=0.6]{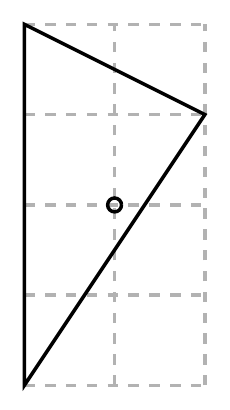}
 \begin{center} Polygon n.21 \end{center}
\end{minipage}

\begin{family}\label{subsec:B_2,8/3}
		 {\sc Name:} $B_{2,8/3}$
\end{family}
	{\sc Model:} degree $6$ hypersurface in $\PP(1,1,3,3)$
		$$\left.\begin{matrix}
			1 & 1 & 3 & 3 & 
		 \end{matrix}\right|
		\begin{matrix}
			& 6 & & p
		\end{matrix}$$

	{\sc Extended twisted I-function:} 
	$$I^{S}(\xi,\tau,z)=z + \tau p + 2\xi \mathbf{1}_0 + \xi q^{-\frac 1 3} \rme^{- \frac t 3} \mathbf{1}_\al + O(z \inv);$$

	{\sc Specialised quantum period:} $$G_{B_{2,8/3}}(x;t)=\exp(-2 x t) \cdot \sum_{l, k \geq 0} 
	\frac{(6 l + 2 k)!}{l!^2 (3 l + k)!^2 k!} x^k t^{2 l + k};$$

	{\sc Regularization:} $$\begin{aligned}\widehat{G}_{B_{2,8/3}}&(x;t)=1 + (40 + 2 x^2) t^2 + 180 x t^3 + 
	(5544 + 624 x^2 + 6 x^4)t^4 + \\ &+(47520 x + 1840 x^3) t^5 + (972400 + 255420 x^2 + 5040 x^4 + 20 x^6)
	t^6+\ldots \\
	\end{aligned}$$
\noindent\begin{minipage}{10cm}

	{\sc Mirror:} (for $a_{[-1,1]}=a_{[1,-1]}=a$) $$f(a,a;x,y)=\frac{y^2}{x}\left(1+\frac{x}{y}\right)^3+\frac{y}{x^2}
	\left(1+\frac{x}{y}\right)^3+a\left(\frac{x}{y}+\frac{y}{x}\right).$$
\end{minipage}
\begin{minipage}{2cm}
 \includegraphics[scale=0.6]{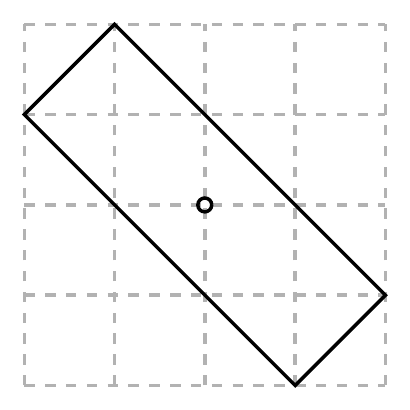}
 \begin{center} Polygon n.12 \end{center}
\end{minipage}

\begin{family}\label{subsec:X_1,22/3}
		 {\sc Name:} $X_{1,22/3}$
\end{family}
	{\sc Model:} blow-up of $\PP(1,1,3)$ in one point, toric (see \S\ref{subsec:example_1})
		$$\begin{matrix}
			1 & 1 & 2 & 0 & & p_1 \\
			0 & 1 & 3 & 1 & & p_2 \\
		 \end{matrix}$$

	{\sc Extended I-function:} $$I^{S}(\tau, \xi,z)=z + \tau p + \xi \bfone_\al + O(z \inv);$$

	{\sc Quantum period:} $$G_{X_{1,22/3}}(x;t)=\sum_{l_1, l_2, k \geq 0} 
	\frac{1}{l_1 ! l_2 ! (l_1 + l_2)! (2l_1 +3 l_2 + k)! k!}x^k t^{4 l_1 + 5 l_2 + 2 k};$$

	{\sc Regularization:} $$\widehat{G}_{X_{1,22/3}}(x;t)=1 + 2 x t^2 + (12 + 6 x^2) t^4 + 20 t^5 + 
	(120 x + 20 x^3) t^6 +\ldots$$
\noindent\begin{minipage}{10cm}

	{\sc Mirror:} (with parameter $a_{[-1,1]}=a$) $$f(a;x,y)=\frac{y^2}{x}+\frac{y}{x^2}+\frac{x}{y}+y+a\frac{y}{x}.$$
\end{minipage}
\begin{minipage}{2cm}
 \includegraphics[scale=0.6]{25.pdf}
 \begin{center} Polygon n.25 \end{center}
\end{minipage}

\begin{family}\label{subsec:X_1,19/3}
		 {\sc Name:} $X_{1,19/3}$
\end{family}
	{\sc Model:} blow-up of $\PP(1,1,3)$ in two points, toric
		$$\begin{matrix}
			1 & 3 & 3 & 0 & 0 & & p_1 \\
			0 & 2 & 1 & 1 & 0 & & p_2 \\
			1 & 2 & 0 & 0 & 1 & & p_3 \\
		 \end{matrix}$$

	{\sc Extended I-function:} $$I^{S}(\tau, \xi,z)=z + \tau p + \xi \bfone_\al + O(z \inv);$$

	{\sc Quantum period:} $$\begin{aligned}&G_{X_{1,19/3}}(x;t)= \\
	                                  =\sum_{\substack{l_1, l_2, l_3, k \geq 0 \\ l_1 - l_2 - 2 l_3 - k \geq 0 \\ -l_1 + 2 l_2 + 3 l_3 + k \geq 0}}
&\frac{x^k t^{l_1 + 2 l_2 + 2 l_3 + k}}{l_1 ! l_2 ! l_3 ! (l_1 - l_2 - 2 l_3 - k)! (-l_1 + 2 l_2 + 3 l_3 + k)! k!}; \\
	                                 \end{aligned}$$

	{\sc Regularization:} $$\begin{aligned}\widehat{G}_{X_{1,19/3}}&(x;t)=1 + 2 x t^2 + 
 6 t^3 + (24 + 6 x^2) t^4 + (20 + 60 x) t^5 + \\ &+ (90 + 240 x + 
    20 x^3) t^6 + (840 + 210 x + 420 x^2) t^7 + \ldots\end{aligned}$$
\noindent\begin{minipage}{10cm}

	{\sc Mirror:} (with parameter $a_{[-1,1]}=a$) $$f(a;x,y)=\frac{y^2}{x}+\frac{y}{x^2}+\frac{x}{y}+y+\frac{1}{x}+a\frac{y}{x}.$$
\end{minipage}
\begin{minipage}{2cm}
 \includegraphics[scale=0.6]{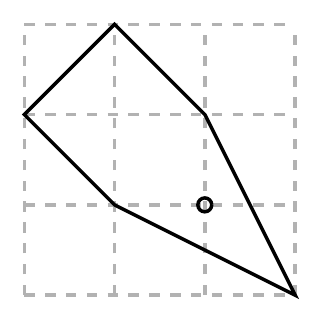}
 \begin{center} Polygon n.24 \end{center}
\end{minipage}

\begin{family}\label{subsec:X_1,16/3}
		 {\sc Name:} $X_{1,16/3}$
\end{family}
	{\sc Model:} blow-up of $\PP(1,1,3)$ in three general points
		$$\begin{matrix}
			1 & 1 & 0 & 0 & 0 & \\
			0 & 0 & 1 & 1 & 3 & \\
		\end{matrix}
		\left|
		\begin{matrix}
		 & 1 & & p_1 \\& 3 & & p_2 \\
		\end{matrix}\right.$$

	{\sc Extended twisted I-function:} 
	$$I^{S}(\xi,\tau,z)=z + \tau p +  (\xi + \rme^{\tau_1} q_1) \mathbf{1}_0 + \xi q_2^{-1/3} \rme^{-\tau_2/3} 
	\mathbf{1}_\al + O(z \inv);$$

	{\sc Quantum period:} 
	$$G_{X_{1,16/3}}(x;t)=\exp(-xt-t) \cdot \sum_{l_1, l_2, k \geq 0} 
	\frac{(l_1 + 3l_2 + k)!}{l_1 !^2 l_2! ^2 (3 l_2 + k)! k!} x^k t^{l_1 + 2 l_2 + k};$$

	{\sc Regularization:} $$\begin{aligned}\widehat{G}_{X_{1,16/3}}&(x;t)=1 + (2 + 2 x) t^2 + 
 18 t^3 + (42 + 24 x + 6 x^2) t^4 + (200 + 180 x) t^5 + \\ &+(1370 + 
    540 x + 180 x^2 + 20 x^3) t^6 + (5460 + 3990 x + 1260 x^2) t^7 + \ldots \end{aligned}$$
\noindent\begin{minipage}{10cm}

	{\sc Mirror:} (with parameter $a_{[0,1]}=a$) $$f(a;x,y)=\frac{(1+y)^3}{xy}+\frac{1}{y}+xy+a y.$$
\end{minipage}
\begin{minipage}{2cm}
 \includegraphics[scale=0.6]{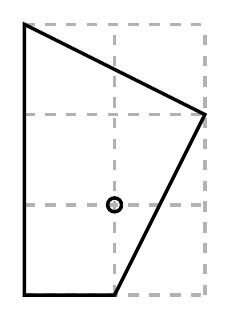}
 \begin{center} Polygon n.22 \end{center}
\end{minipage}

\begin{family}\label{subsec:X_1,13/3}
		 {\sc Name:} $X_{1,13/3}$
\end{family}
	{\sc Model:} blow-up of $\PP(1,1,3)$ in four general points
		$$\begin{matrix}
1 & 1 & 3 & 1 & 0  & \\
0 & 0 & 0 & 1 & 1 & \\
\end{matrix}\left|
\begin{matrix}
 & 4 & & p_1 \\ & 1 & & p_2 \\
\end{matrix}\right.
$$

	{\sc Non-extended twisted I-function:} 
	$$I(\tau,z)=z + \tau p + \rme^{\tau_2} q_2 \mathbf{1}_0 + O(z \inv);$$

	{\sc Specialised quantum period:} 
	$$G_{X_{1,13/3}}(t)=\exp(-t) \cdot \sum_{l_1, l_2 \geq 0} \frac{(4 l_1 + l_2)!}{(l_1 !)^2 (3 l_2) ! (l_1 + l_2)! l_2 !} t^{2l_1 + l_2};$$

	{\sc Regularization:} $$\widehat{G}_{X_{1,13/3}}(t)=1 + 8 t^2 + 36 t^3 + 216 t^4 + 1700 t^5 + 10700 t^6 + 81060 t^7 + \ldots$$
\noindent\begin{minipage}{10cm}

	{\sc Mirror:} (for $a_{[-1,1]}=0$)$$f(0;x,y)=\left(\frac{y}{x^2}+
	\frac{y^2}{x}\right)\left(1+\frac{x}{y}\right)^2+\left(1+\frac{x}{y}\right)-1$$
\end{minipage}
\begin{minipage}{2cm}
 \includegraphics[scale=0.6]{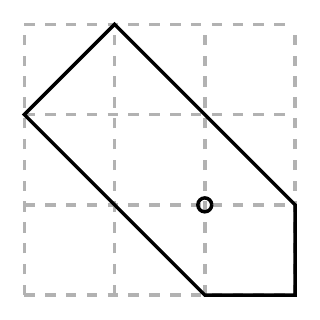}
 \begin{center} Polygon n.18 \end{center}
\end{minipage}

\begin{family}\label{subsec:X_1,10/3}
		 {\sc Name:} $X_{1,10/3}$
\end{family}
	{\sc Model:}  blow-up of $\PP(1,1,3)$ in five general points
		$$\begin{matrix}
1 & 1 & 2 & 1 & 0 & 0 & \\
0 & 0 & 1 & 2 & 1 & 1 & \\
\end{matrix}\left|
\begin{matrix}
& 2 & 2 & & p_1 \\ & 2 & 2 & & p_2 \\ 
\end{matrix}
\right.
$$

	{\sc Non-extended twisted I-function:} 
	$$I(\tau,z)=z + \tau p + 2(q_1  + q_2) \mathbf{1}_0 + \left( e^{\frac{-\tau_1+2\tau_2}{3}}q_1^{-\frac{1}{3}} q_2^{\frac{2}{3}} 
+ e^{\frac{2\tau_1-\tau_2}{3}}q_1^{\frac{2}{3}} q_2^{-\frac{1}{3}} \right)\mathbf{1}_\al + O(z \inv);$$

	{\sc Specialised quantum period:} 
	
	$$G_{X_{1,10/3}}(t)=\exp(-4t) \cdot \sum_{l_1, l_2 \geq 0} \frac{(2 l_1 + 2 l_2)!^2}{(l_1 !)^2 (2 l_1 + l_2)! (l_1 + 2 l_2)! (l_2!)^2} t^{l_1 + l_2};$$

	{\sc Regularization:} $$\widehat{G}_{X_{1,10/3}}(t)=1 + 28 t^2 + 180 t^3 + 2604 t^4 + 29680 t^5 + 384700 t^6 + 
 4944240 t^7 +\ldots$$

\noindent\begin{minipage}{10cm}
	{\sc Mirror:} (for $a_{[0,1]}=2$)$$f(2;x,y)=\frac{(1+x)^2(1+y)^2}{xy}+\frac{(1+y)^2}{x}+2 y-4,$$
\end{minipage}
\begin{minipage}{2cm}
 \includegraphics[scale=0.6]{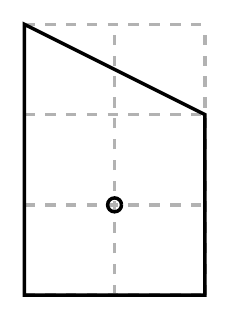}
 \begin{center} Polygon n.15 \end{center}
\end{minipage}

\begin{family}\label{subsec:X_1,7/3}
		 {\sc Name:} $X_{1,7/3}$
\end{family}

	{\sc Model:} complete intersection of type $(2,2,2,2)$ in the weighted Grassmannian $\mathrm{wGr}(2,5)$ with weights $(\frac{1}{2}, \frac{1}{2}, \frac{1}{2}, \frac{3}{2}, \frac{3}{2})$.

\medskip
{\sc Specialised quantum period:} \begin{align*}
G(t) &= \exp(-8t) \times \\
&\times \sum_{l_1, l_2 \in \NN} A_{l_1, l_2}(t) \left( 1 + \frac{l_1 - l_2}{2} \left( -3 H_{l_1} + 3 H_{l_2} - 2 H_{2 l_1 + l_2} + 2 H_{2 l_2 + l_1} \right) \right),
\end{align*}
where 
\begin{equation*}
H_l := \begin{cases}
0 &\quad \text{if } l=0; \\
\sum_{i = 1}^l \frac{1}{i} &\quad \text{if } l > 0.
\end{cases}
\end{equation*}
and
\begin{equation*}
A_{l_1, l_2}(t) := (-t)^{l_1+l_2} \frac{(2l_1 + 2 l_2)!^4}{l_1!^3 l_2!^3 (2l_1 + l_2)!^2 (l_1 + 2 l_2)!^2}.
\end{equation*}

{\sc Regularization:} $\widehat{G}_{X_{1,7/3}}(t) = 1 + 112t^2 + 1650 t^3 + 48048 t^4 + \cdots$

\noindent\begin{minipage}{9.5cm}
	{\sc Mirror:} (for $a_{[0,1]}=3$)$$f(3;x,y)=3y+\frac{x}{y^2}(1+y)^3+\frac{1}{xy^2}(1+y)^4+\frac{7}{y}+\frac{2}{y^2},$$
\end{minipage}
\begin{minipage}{2cm}
 \includegraphics[scale=0.6]{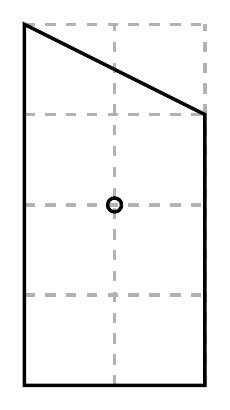}
 \begin{center} Polygon n.10 \end{center}
\end{minipage}

\begin{family}\label{subsec:X_1,4/3}
		 {\sc Name:} $X_{1,4/3}$
\end{family}
	{\sc Model:} surface $X_{4,4}\subset\PP(1,1,2,2,3)$,
	$$
	\begin{matrix}
	 1 & 1 & 2 & 2 & 3 &  
	\end{matrix}\left|
	\begin{matrix}
	 & 4 & 4 & & p
	\end{matrix}\right.
	$$

	{\sc Non-extended twisted I-function:} 
	$$I(\tau,z)=z + \tau p + 24 \rme^{\tau} q \mathbf{1}_0 + 4 \rme^{\frac \tau 3} q^\frac{1}{3} \mathbf{1}_{\al} + O(z \inv)$$

	{\sc Specialised quantum period:} 
	$$G_{X_{1,4/3}}(t)=\exp(-24t) \cdot \sum_{l \geq 0} \frac{(4 l)!^2}{(l !)^2 (2 l)!^2 (3 l)!} t^{l}$$
	
	{\sc Regularization:}
	$$\begin{aligned}\widehat{G}_{X_{1,4/3}}(t)=1 + 1384 t^2 + 89808 t^3& + 9686952 t^4 + 985441920 t^5 + \\
 &+106460790640 t^6 + 11728528875840 t^7+\ldots\end{aligned}$$
	
	\noindent\begin{minipage}{9cm}

	{\sc Mirror:} (with parameter $a_{[-1,1]}=8$)
	$$f(8;x,y)=\left(1+\frac{1}{x}+\frac{1}{y}\right)^4\left(\frac{y}{x^2}+\frac{y^2}{x}\right)-24,$$
\end{minipage}
\begin{minipage}{2cm}
 \includegraphics[scale=0.6]{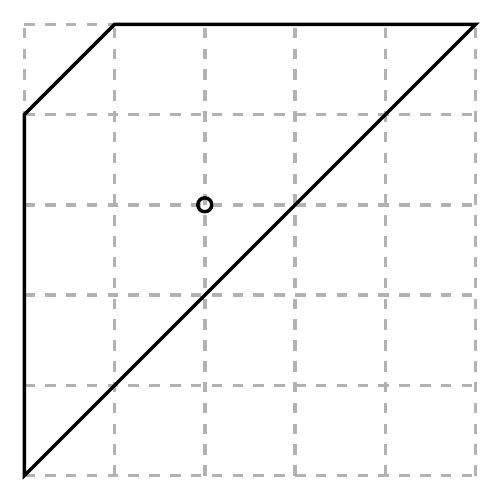}
 \begin{center} Polygon n.4 \end{center}
\end{minipage}


\begin{family}\label{subsec:X_1,1/3}
		 {\sc Name:} $X_{1,1/3}$
\end{family}
	{\sc Model:} hypersurface $X_{10}\subset\PP(1,2,3,5)$
	$$
	\begin{matrix}
	 1 & 2 & 3 & 5 &  
	\end{matrix}\left|
	\begin{matrix}
	 & 10 & & p
	\end{matrix}\right.
	$$

	{\sc Non-extended twisted I-function:} 
	$$I(\tau,z)=z + \tau P + 2520 \rme^{\tau} q \mathbf{1}_0 + 14 \rme^{\frac \tau 3} q^\frac{1}{3} \mathbf{1}_{\al} + O(z \inv)$$

	{\sc Specialised quantum period:}
	$$G_{X_{1,1/3}}(t)=\exp(-2520t) \sum_{l\ge 0} \frac{10l!}{l! 2l! 3l! 5l!} t^{l}$$
	
	{\sc Regularization:}
	$$\begin{aligned}&\widehat{G}_{X_{1,1/3}}(t)=1 + 32448360 t^2 + 515050578000 t^3 + 10896129436991400 t^4 + \\
 &+ 237912591939390587520 t^5 + 5409193242794544150150000 t^6+\ldots\end{aligned}$$
	
	\noindent\begin{minipage}{9cm}

	{\sc Mirror:} (with parameter $a_{[4,3]}=360$)
	$$f(360;x,y)=\frac{y^5}{x^3}\left(1+x+\frac{1}{y}\right)^{10}-2520.$$
\end{minipage}
\begin{minipage}{2cm}
 \includegraphics[scale=0.35]{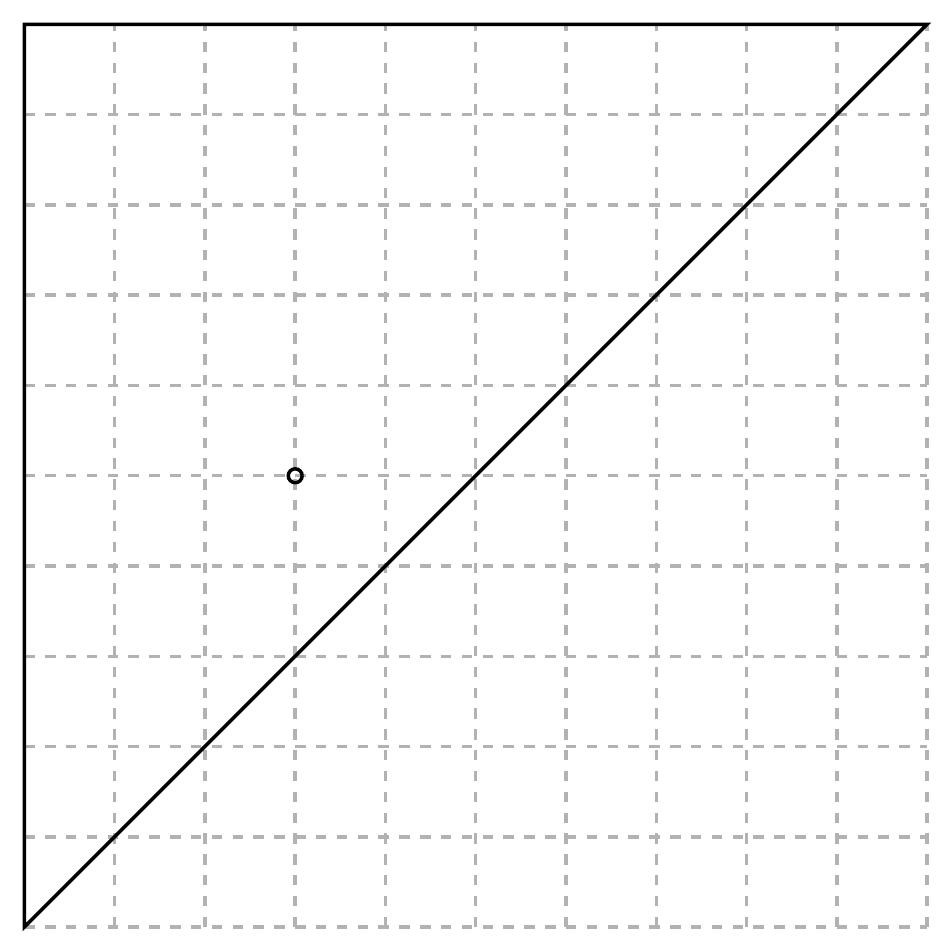}
 \begin{center} Polygon n.1 \end{center}
\end{minipage}

\subsection{Surfaces with Fano index 1 and 2 singular points}

\begin{family}\label{subsec:X_2,17/3}
		 {\sc Name:} $X_{2,17/3}$
\end{family}
	{\sc Model:} hypersurface in toric threefold
$$
\begin{matrix}
 1 & 1 & 2 & 3 & 0 & \\
-1 & 0 & 1 & 3 & 1 & \\
\end{matrix}\left|
\begin{matrix} 
 & 4 & p_1 \\ & 2 & p_2 \\
\end{matrix}\right.
$$

	{\sc Non-extended twisted I-function:} 
	$$I(\tau,z)=z + \tau p + O(z \inv)$$

	{\sc Specialised quantum period:} 
	$$ G_{X_{2,17/3}}(t)=\sum_{l_1 \geq l_2 \geq 0} \frac{(4 l_1 + 2l_1)!}{(l_1-l_2)!l_1 ! (2l_1 +  l_2)! (3 l_1 + 3 l_2)! l_2! } t^{3 l_1 + 2 l_2};$$
	
	{\sc Regularization:}
$$
 \widehat{G}_{X_{2,17/3}}(t)=1 + 12 t^3 + 20 t^5 + 420 t^6 + 1680 t^8 + 18480 t^9+\ldots
$$
	
	\noindent\begin{minipage}{10cm}

	{\sc Mirror:} (for $a_{[1,0]}=a_{[0,1]}=0$)
 $$
  f(0,0;x,y)=\frac{(1+x)^2}{y} + \frac{y^2}{x} + xy.
 $$ 
\end{minipage}
\begin{minipage}{2cm}
 \includegraphics[scale=0.6]{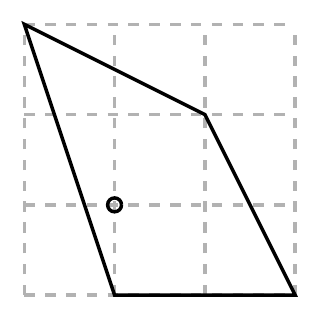}
 \begin{center} Polygon n.23 \end{center}
\end{minipage}

\newpage 

\begin{family}\label{subsec:X_2,14/3}
		 {\sc Name:} $X_{2,14/3}$
\end{family}
	{\sc Model:} hypersurface in toric threefold
$$
\begin{matrix}
 1 & 0 & 3 & 1 & 2 & \\
 0 & 1 & 3 & 1 & 1 & \\
\end{matrix}\left|
\begin{matrix}
 & 4 & p_1 \\ & 4 & p_2 \\
\end{matrix}
\right.
$$

	{\sc Non-extended twisted I-function:} 
	$$I(\tau,z)=z + \tau p + 2\rme^{\frac{-\tau_1+2\tau_2}{3}}q_1^{-\frac{1}{3}}q_2^{\frac{2}{3}}{\bf 1}_{\alpha} + O(z \inv);$$

	{\sc Specialised quantum period:}
	$$G_{X_{2,14/3}}(t)=\sum_{l_1,l_2\geq 0} \frac{(4l_1+4l_2)!}{l_1! l_2! (3l_1+3l_2)! (l_1+l_2)! (2l_1+l_2)!} t^{3l_1+2l_2}$$
	
	{\sc Regularization:}
		$$\widehat{G}_{X_{2,14/3}}(t)=1 + 8 t^2 + 12 t^3 + 168 t^4 + 560 t^5 + 4820 t^6 + 23100 t^7+\ldots$$
	
	\noindent\begin{minipage}{10cm}

	{\sc Mirror:} (with parameters $a_{[1,0]}=0$ and $a_{[0,1]}=2$) 
	$$f(2,0;x,y)=\frac{(1+x)^2}{y}+\frac{y(1+x)}{x}+\frac{y^2}{x}.$$
\end{minipage}
\begin{minipage}{2cm}
 \includegraphics[scale=0.7]{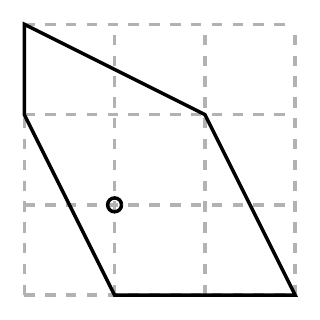}
 
 \begin{center}Polygon n.19\end{center}
\end{minipage}

\begin{family}\label{subsec:X_2,11/3}
		 {\sc Name:} $X_{2,11/3}$
\end{family}
	{\sc Model:} hypersurface in toric threefold
$$
\begin{matrix}
 1 & 0 & 0 & 1 & 0 & 1 &\\
 0 & 1 & 0  & 0 & 1 & 1 & \\
 0 & 0 & 1  & 1 & 1 & 4 & \\
\end{matrix}\left|
\begin{matrix}
 & 2 & p_1 \\ & 2 & p_2 \\ & 4 & p_3 \\
\end{matrix}\right.
$$

	{\sc Non-extended twisted I-function:} 
	$$I(\tau,z)=z+\tau p+2\rme^{\tau_1+\tau_2}(q_1+q_2){\bf 1}_0+\text{(orbifold classes)}+O(z \inv)$$

	{\sc Specialised quantum period:}
	$$G_{X_{2,11/3}}(t)=\exp(-4t)\sum_{l_1, l_2, l_3 \geq 0}\frac{(2l_1+2l_2+4l_3)!}{l_1! l_2! l_3! (l_1+l_3)! (l_2+l_3)! (l_1+l_2+4l_3)!}t^{l_1+l_2+3l_3}$$
	
	{\sc Regularization:}
		$$\widehat{G}_{X_{2,11/3}}(t)=1 + 20 t^2 + 102 t^3 + 1236 t^4 + 11440 t^5 + 121610 t^6 + 1278060 t^7 + \ldots$$
	
	\noindent\begin{minipage}{10cm}

	{\sc Mirror:}  (with parameters $a_{[1,0]}=a_{[0,1]}=3$)$$
  f(3,3;x,y)=\frac{(1+x+y)^2}{x} + \frac{(1+x+y)^2}{y} + xy - 4.
 $$ 
\end{minipage}
\begin{minipage}{2cm}
 \includegraphics[scale=0.7]{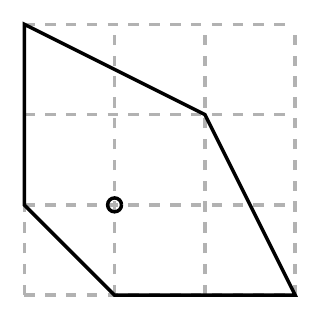}
 
 \begin{center}Polygon n.16\end{center}
\end{minipage}

\newpage

\begin{family}\label{subsec:X_2,8/3}
		 {\sc Name:} $X_{2,8/3}$
\end{family}
	{\sc Model:} hypersurface in toric threefold
$$
\begin{matrix}
 1 & 1 & 1 & 1 & 0 & \\
 0 & 0 & 1 & 3 & 1 & \\
\end{matrix}\left|
\begin{matrix}
 & 3 & p_1 \\ & 3 & p_2 \\
\end{matrix}\right.
$$

	{\sc Extended twisted I-function:} 
$$I^{S}(\xi,\tau,z)=z + \tau p + (6q_1 + \xi) \rme^{\tau_1} \mathbf{1}_0 + (3q_1+\xi)q_2^{-\frac 1 3} \rme^{\tau_1 - \frac{\tau_2}3} \mathbf{1}_{\al} + O(z \inv)
$$

	{\sc Specialised quantum period:}
$$G_{X_{2,8/3}}(x;t)=\exp(-(6+x)t)\sum_{\substack{l_1,~l_2,~k\in\ZZ \\ l_1,~l_2,~k \geq0}}
\frac{(3l_1+3l_2+k)!}{(l_1!)^2l_2!(l_1+l_2)!(l_1+3l_2+k)!k!}x^kt^{l_1+2l_2+k}
$$	
	{\sc Regularization:}
$$\begin{aligned}\widehat{G}_{X_{2,8/3}}&(x;t)=1 + (12x + 56)t^2 + (6x^2 + 144x + 546)t^3 + \\ 
&+ (396x^2 + 4176x + 11184)t^4 + (360x^3 + 11220x^2 + 84240x + 189060)t^5 + \\
&+ (90x^4 + 20640x^3 + 339480x^2 + 1900800x + 3560870)t^6 + \ldots \end{aligned},$$
	
	\noindent\begin{minipage}{10cm}

	{\sc Mirror:} (with parameter $a_{[1,0]}=a_{[0,1]}=a+3$) $$
  f(a,a;x,y)=\frac{(1+x+y)^3}{xy} + xy + a(x+y) -6.
 $$ 
\end{minipage}
\begin{minipage}{2cm}
 \includegraphics[scale=0.7]{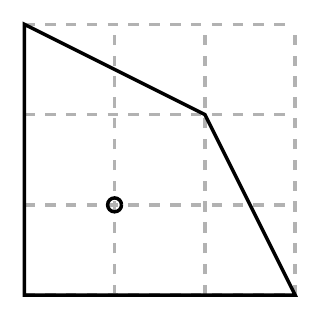}
 
 \begin{center}Polygon n.13\end{center}
\end{minipage}

\begin{family}\label{subsec:X_2,5/3}
		 {\sc Name:} $X_{2,5/3}$
\end{family}
	{\sc Model:} hypersurface in toric threefold
$$
\begin{matrix}
 1 & 1 & 2 & 1 & 0 & \\
 0 & 1 & 3 & 3 & 1 & \\
 \end{matrix}\left|
\begin{matrix}
 & 4 & p_1 \\ & 6 & p_2 \\
\end{matrix}\right.
$$

	{\sc Extended twisted I-function:} 
$$I(\xi,\tau,z)=z+\left(12e^{\tau_1}q_1+2\xi \right){\bf 1}_0+\left(3e^{\tau_1-\frac{1}{3}}q_1q_2^{-\frac{1}{3}}+
\xi q_2^{-\frac{1}{3}}\right){\bf 1}_{\al}+O(z \inv)$$

	{\sc Specialised quantum period:} 
$$
 G_{X_{2,5/3}}(x;t)=\exp(-(12+x)t)\sum_{l_1, l_2, k \geq 0}\frac{(4l_1+6l_2+2k)! ~~~~ x^k t^{l_1+2l_2+2k}}{l_1! (l_1+l_2)!  (2l_1+3l_2+k)! (l_1+3l_2+k)!l_2!  k!}
$$
	{\sc Regularization:}
$$
\begin{aligned}
\widehat{G}_{X_{2,5/3}}&(x;t)=1+(2x^2 + 72x + 316)t^2+(192x^2 + 2628x + 9156)t^3+ \\ 
    &+(6x^4 + 816x^3 + 22440x^2 + 180576x + 461484)t^4 +\\
    &+(1920x^4 + 110800x^3 + 1690800x^2 + 10130400x + 21425520)t^5+\ldots
\end{aligned}
$$
	
	\noindent\begin{minipage}{9cm}

	{\sc Mirror:} 
	$1$-parameter family of maximally mutable \\ polynomials with coefficients of the interior points:
	$$\begin{aligned}
	 a_{[1,1]} &=12+a; \\
	 a_{[0,1]} &=-13+a; \\
	 a_{[-1,1]}&=6+a; \\
	 a_{[-1,0]}&=8+2a; \\
	 a_{[-1,-1]} &= 4+a; \\
	\end{aligned}$$
\end{minipage}
\begin{minipage}{2cm}
 \includegraphics[scale=0.6]{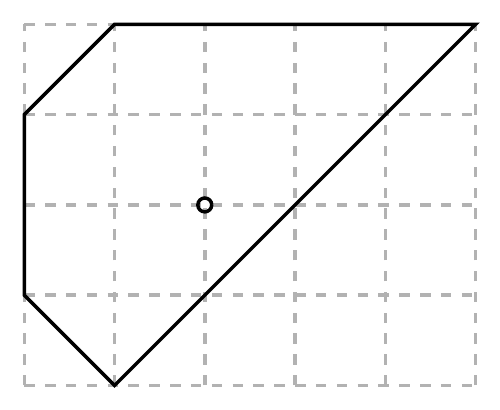}
 
 \begin{center}Polygon n.6\end{center}
\end{minipage}

\begin{family}\label{subsec:X_2,2/3}
		 {\sc Name:} $X_{2,2/3}$
\end{family}
	{\sc Model:} complete intersection $S_{4,6}$ in $\PP(1,2,2,3,3)$
$$
\begin{matrix}
 1 & 2 & 2 & 3 & 3 & \\
 \end{matrix}\left|
\begin{matrix}
 & 4 & 6 & & p \\
\end{matrix}\right.
$$

	{\sc Non-extended twisted I-function:} 
$$
I(\tau,z)=z \mathbf{1}_0 + \tau p + 120 \rme^{\tau} q \mathbf{1}_0 + 6 \rme^{\frac \tau 3} q^{\frac 1 3} \mathbf{1}_{\alpha} + O(z \inv)
$$

	{\sc Specialised quantum period:} 
$$
 G_{X_{2,2/3}}(t)=\exp(-120t)\sum_{l\geq 0}\frac{(4l!)^2}{l! (2l!)^2 (3l!)^2}t^{l}
$$

	{\sc Regularization:}
$$
 \begin{aligned}
 \widehat{G}_{X_{2,2/3}}&(t)=1 + 50280 t^2 + 25096080 t^3 + 18204817320 t^4 + \\ &+13228445013120 t^5 + 
 10057163200940400 t^6  \ldots
 \end{aligned}
$$

	\noindent\begin{minipage}{9cm}

	{\sc Mirror:} (with parameters $a_{[2,-1]}=a_{[-2,1]}=24$)
 $$
  f(24,24;x,y)=\frac{(1+x)^6(1+y)^4}{x^3y^2}.
 $$
\end{minipage}
\begin{minipage}{2cm}
 \includegraphics[scale=0.5]{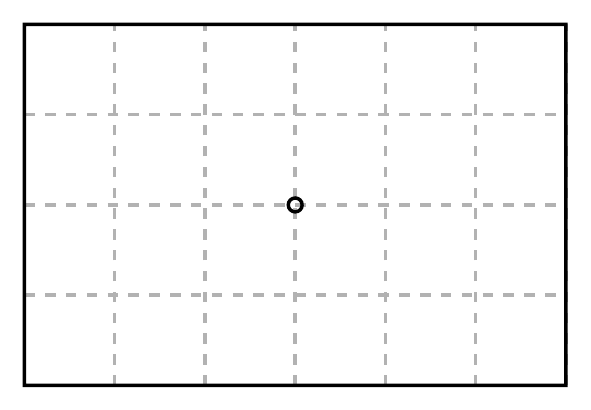}
 
 \begin{center}Polygon n.2\end{center}
\end{minipage}


\subsection{Surfaces with Fano index 1 and 3 singular points}

\begin{family}\label{subsec:X_3,5}
		 {\sc Name:} $X_{3,5}$
\end{family}
	{\sc Model:} toric
	$$
\begin{matrix}
 1 & 0 & 0 & 1 & 0 \\
 0 & 1 & 0 & 2 & 3 \\
 0 & 0 & 1 & 1 & 3 \\
\end{matrix}
\begin{matrix}
  & & p_1 \\
  & & p_2 \\
  & & p_3 \\
\end{matrix}
$$

	{\sc Extended twisted I-function:} 
$$
I(\xi,\tau,z)^{{\rm S}}=z+\tau p + \xi_1q_3^{-\frac{1}{3}}{\bf 1}_{\alpha_1}+\xi_2q_2^{-\frac{1}{3}}q_3^{\frac{1}{3}}{\bf 1}_{\al_2}+\xi_3q_1^{-\frac{1}{3}}q_2^{-\frac{1}{3}}{\bf 1}_{\al_3}+O(z \inv)
$$

	{\sc Quantum period:} 
$$
\begin{aligned}
 &G_{X_{3,5}}(x_1,x_2,x_3;t)=\sum_{\substack{l_1,l_2,l_3 \geq 0 \\ k_1,k_2,k_3 \geq 0}}
 \frac{1}{l_1!l_2!l_3!k_1!k_2!k_3!} \times \\ 
 &\times\frac{x^{k_1}x_2^{k_2}x_3^{k_3}t^{2l_1+6l_2+5l_3+2k_1+4k_2+3k_3}}{(l_1+2l_2+l_3+k_2+k_3)!(3l_2+3l_3+k_1+2k_2+k_3)!}&
 \end{aligned}
$$

{\sc Regularization:}
$$
 \begin{aligned}
  \widehat{G}_{X_{3,5}}(x_1,x_2,x_3;t)&=1 + (2 + 2 x_1) t^2 + 
 6 x_3 t^3 + \\ & +(6 + 12 x_2 + 24 x_1 +6 x_1^2) t^4 + (20 + 60 x_3 + 
    60 x_3 x_1) t^5 + \ldots \\
   \end{aligned}
$$

	\noindent\begin{minipage}{9cm}

	{\sc Mirror:}  (with parameters $a_{[0,1]}=a_1,~a_{[-1,1]}=a_2$ and $a_{[-1,0]}=a_3$)
 $$
  f(a_1,a_2,a_3;x,y)= \frac{1}{y} + \frac{1}{x y} + \frac{y}{x^2} + x y + \frac{y^2}{x} + a_1 y +a_2\frac{y}{x} + a_3\frac{1}{x}
 $$
\end{minipage}
\begin{minipage}{2cm}
 \includegraphics[scale=0.8]{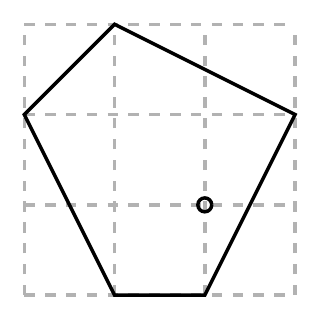}
 
 \begin{center}Polygon n.20\end{center}
\end{minipage}

\begin{family}\label{subsec:X_3,4}
		 {\sc Name:} $X_{3,4}$
\end{family}
	{\sc Model:} toric
$$
 \begin{matrix}
  1 & 0 & 0 & 0 & 1 & -1 \\
  0 & 1 & 0 & 0 & -1 & 1 \\
  0 & 0 & 1 & 0 & 2 & 1 \\
  0 & 0 & 0 & 1 & 1 & 2 \\
 \end{matrix}
\begin{matrix}
  & & p_1 \\
  & & p_2 \\
  & & p_3 \\
  & & p_4
\end{matrix}
$$

	{\sc Extended twisted I-function:} 
$$
I(\xi,\tau,z)^S=z+\tau p+\xi_1q_1^{-\frac{1}{3}}q_4^{-\frac{1}{3}}{\bf 1}_{\al_1}+\xi_2q_3^{-\frac{1}{3}}q_4^{\frac{1}{3}}{\bf 1}_{\al_2}+
\xi_3q_1^{-\frac{1}{3}}q_3^{-\frac{1}{3}}{\bf 1}_{\al_3}+O(z \inv)
$$

	{\sc  Quantum period:}
$$
\begin{aligned}
 G_{X_{3,4}}(x_1,x_2,x_3;t)&=\sum_{\substack{l_1,l_2,l_3,l_4,k_1,k_2,k_3 \geq 0 \\ l_1-l_2+2l_3+l_4+k_2+k_3 \geq 0 \\ -l_1+l_2+l_3+2l_4+k_1+k_2 \geq 0}}
 \frac{1}{l_1!l_2!l_3!l_4!k_1!k_2!k_3!} \times \\ 
 &\times\frac{x^{k_1}x_2^{k_2}x_3^{k_3}t^{l_1+l_2+4l_3+4l_4+2k_1+3k_2+2k_3}}{(l_1-l_2+2l_3+l_4+k_2+k_3)!(-l_1+l_2+l_3+2l_4+k_1+k_2)!}&
 \end{aligned}
$$

{\sc Regularization:}
$$
 \begin{aligned}
 &\widehat{G}_{X_{3,4}}(x_1,x_2,x_3;t)=1 + (2 + 2 x_1 + 2 x_3)t^2 + (6 x_1 + 6 x_2 + 6 x_3)t^3 + \\
 & + (30 + 24 x_1 + 6 x_1^2 + 24 x_2 + 24 x_3 + 24 x_1 x_3 + 
    6 x_3^2) t^4 + \\ 
 &+ (160 + 60 x_1 + 60 x_1^2 + 120 x_2 + 60 x_1 x_2 + 60 x_3 + 
    120 x_1 x_3 + 60 x_2 x_3 + 60 x_3^2) t^5 + \ldots \\
    \end{aligned}
$$

	\noindent\begin{minipage}{9cm}

	{\sc Mirror:} (with parameters $a_{[0,1]}=a_1,~a_{[-1,1]}=a_2$ and $a_{[-1,0]}=a_3$)
 $$
  f(a_1,a_2a_3;x,y)= x+ \frac{1}{y} + \frac{1}{x y} + \frac{y}{x^2} + x y + \frac{y^2}{x} + a_1 y+ a_2\frac{y}{x}+ a_3\frac{1}{x}
 $$
\end{minipage}
\begin{minipage}{2cm}
 \includegraphics[scale=0.8]{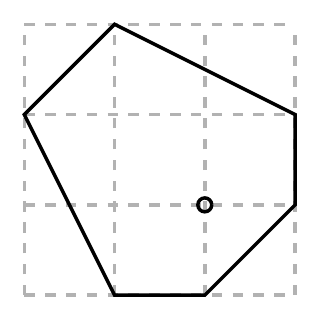}
 
 \begin{center}Polygon n.17\end{center}
\end{minipage}

\begin{family}\label{subsec:X_3,3}
		 {\sc Name:} $X_{3,3}$
\end{family}
	{\sc Model:} hypersurface in toric threefold
$$\begin{matrix}
   1 & 1 & 1 & 0 & 0 & \\
   0 & 0 & 1 & 1 & 3 & \\
  \end{matrix}\left|
 \begin{matrix}
  & 2 & & p_1 \\ & 3 & & p_2 \\
 \end{matrix}\right.
$$

	{\sc Non-extended twisted I-function:} 
$$I(\tau,z)= z + \tau p + 2 \rme^{\tau_1} q_1 \mathbf{1}_0 + \rme^{\frac{-\tau_1 + \tau_2}{3}} q_1^{- \frac{\tau_1}{3}} q_2^{\frac{\tau_2}{3}} \mathbf{1}_{\al} + O(z \inv).
$$

	{\sc Specialised quantum period:} 
$$
\begin{aligned}
 G_{X_{3,3}}(t) = \exp(-2t) \cdot \sum_{l_1,l_2 \geq 0} \frac{(2l_1+3l_2)! }{(l_1 !)^2 (l_1 + l_2)! l_2! (3 l_2)!} t^{l_1 + 2 l_2}
 \end{aligned}
$$

{\sc Regularization:}
$$
 \widehat{G}_{X_{3,3}}(t) = 1 + 4 t^2 + 48 t^3 + 420 t^4 + 2740 t^5 + 22360 t^6 + 209370 t^7 + 
 1856820 t^8+\ldots
$$

	\noindent\begin{minipage}{10cm}

	{\sc Mirror:} (with parameter $a_{[0,1]}=a_{[-1,0]}=0$ and $a_{[-1,1]}=1$)
 $$
  f(0,1,0;x,y)=\frac{x}{y}\left( 1+\frac{1}{x}+y\right)^2+\frac{y}{x}\left(1+\frac{1}{x}+y\right)-2.
 $$
\end{minipage}
\begin{minipage}{2cm}
 \includegraphics[scale=0.8]{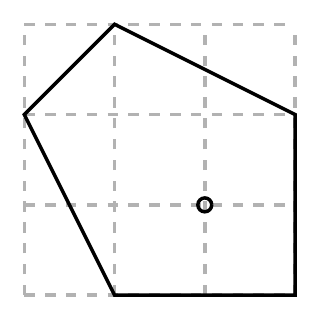}
 
 \begin{center}Polygon n.14\end{center}
\end{minipage}


\begin{family}\label{subsec:X_3,2}
		 {\sc Name:} $X_{3,2}$
\end{family}
	{\sc Model:} hypersurface in toric threefold
$$\begin{matrix}
   1 & 3 & 2 & 0 & -1 & \\
   0 & 0 & 1 & 1 & 1 & \\
  \end{matrix}\left|
  \begin{matrix}
   & 4 & p_1 \\ & 2 & p_2 \\
  \end{matrix}\right.
$$

	{\sc Non-extended twisted I-function:} 
$$
I(\tau,z)=z + \tau p + 2 \rme^{\tau_2} q_2 \mathbf{1}_0 + 2 \rme^{\frac{\tau_1}{3}} q_1^\frac{\tau_1}{3} \mathbf{1}_{\al} + O(z \inv).
$$

	{\sc Specialised quantum period:}
$$
 G_{X_{3,2}}(t) = \exp(-2t)\sum_{l_2 \geq l_1 \geq 0} \frac{(4l_1+2l_2)!}{l_1! (3l_1)! (2l_1+l_2)! l_2! (l_2-l_1)!}t^{l_1+l_2}
$$

{\sc Regularization:} $$ \widehat{G}_{X_{3,2}}(t) = 1 + 42 t^2 + 600 t^3 + 9870 t^4 + 206800 t^5 + 3919500 t^6 + 80106600 t^7 +\ldots$$

	\noindent\begin{minipage}{10cm}

	{\sc Mirror:} coefficients in the interior points,
	$$
		\begin{aligned}
		 a_{[1,1]} & =1; \\
		 a_{[0,1]} & = 3; \\
		 a_{[-1,1]} & = 4; \\
		 a_{[-1,0]} & = 3; \\
		 a_{[-1,-1]} & = 1. \\
		\end{aligned}
	$$
\end{minipage}
\begin{minipage}{2cm}
 \includegraphics[scale=0.6]{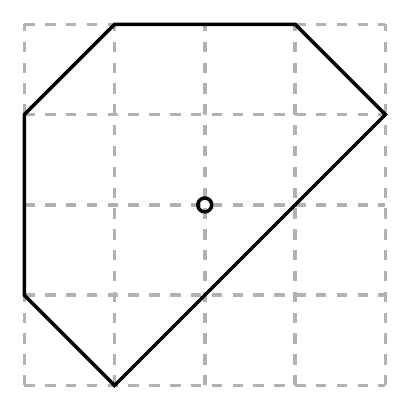}
 
 \begin{center}Polygon n.8\end{center}
\end{minipage}

\begin{family}\label{subsec:X_3,1}
		 {\sc Name:} $X_{3,1}$
\end{family}
	{\sc Model:} hypersurface in toric threefold
$$ \begin{matrix}
  1 & 0 & 0 & 1 & 1 & 2 & \\
  0 & 1 & 0 & 1 & 2 & 1 & \\
  0 & 0 & 1 & 2 & 1 & 1 & \\
 \end{matrix}\left|
 \begin{matrix}
  & 4 & & p_1 \\ & 4 & & p_2 \\ & 4 & & p_3 \\
 \end{matrix}
 \right.
$$

	{\sc Non-extended twisted I-function:} 
$$
I(\tau,z)=z + \tau p+12\rme^{\tau_1+\tau_2+\tau_3}\left(q_1+q_2+q_3\right){\bf 1}_0+(\text{orbifold classes})+O(z \inv)
$$

	{\sc Specialised quantum period:}
$$
 G_{X_{3,1}}(t) = \exp(-36t)\sum_{l_1, l_2, l_3 \geq 0}\frac{(4l_1+4l_2+4l_3)!~~~~~t^{l_1+l_2+l_3}}{l_1!l_2!l_3!(l_1+l_2+2l_3)!(l_1+2l_2+l_3)!(2l_1+l_2+l_3)!}
$$

{\sc Regularization:}
$$
 \widehat{G}_{X_{3,1}}(t) = 1 + 3324 t^2 + 356652 t^3 + 61331148 t^4 + 10136532960 t^5 + 
 1770572214660 t^6+\ldots
$$

	\noindent\begin{minipage}{9cm}

	{\sc Mirror:} (with parameters $a_{[2,1]}=a_{[1,1]}=a_{[-1,-2]}=9$)
 $$
  f(9,9,9;x,y)=\left(1+x+\frac{1}{y}\right)^4\left(\frac{y}{x}+\frac{y}{x^2}+\frac{y^2}{x}\right)-36.
 $$
\end{minipage}
\begin{minipage}{2cm}
 \includegraphics[scale=0.6]{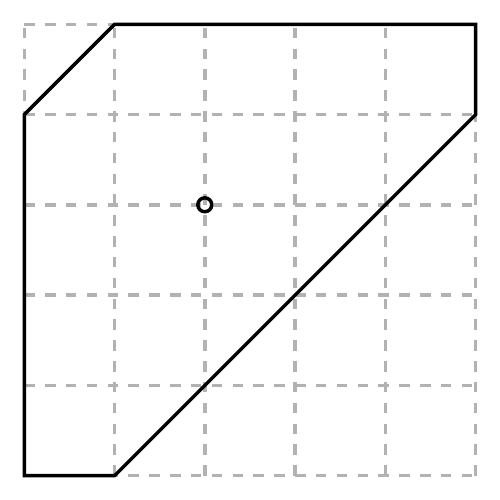}
 
 \begin{center}Polygon n.3\end{center}
\end{minipage}


\subsection{Surfaces with Fano index 1 and 4 singular points}
\begin{family}\label{subsec:X_4,7/3}
		 {\sc Name:} $X_{4,7/3}$
\end{family}
	{\sc Model:} hypersurface in toric threefold
$$
 \begin{matrix}
  1 & 0 & 3 & 3 & 2 & \\
  0 & 1 & 3 & 3 & 1 & \\
 \end{matrix}
 \left|\begin{matrix}
  & 6 & & p_1 \\ & 6 & & p_2 \\
 \end{matrix}
 \right.
$$

	{\sc Extended twisted I-function:} 
$$
I^S(\xi,\tau,z)=z + \tau p + 2\xi_2{\bf 1}_0 + (\text{orbifold classes})+ O(z \inv).
$$

	{\sc Specialised quantum period:} 
$$
 G_{X_{4,7/3}}(x;t) = {\rm exp}(-2x_2t)\sum_{l_1, l_2, k_1, k_2 \geq 0} \frac{(6l_1+6l_2+4k_1+2k_2)!~~~~x_1^{k_1}x_2^{k_2} t^{3l_1+2l_2+2k_1+k_2}}{l_1 ! l_2! k_1! k_2!(3 l_1+3l_2+2k_1+k_2)!^2 (2l_1+l_2+k_1)!} 
$$

{\sc Regularization:}
$$
\begin{aligned}
    &\widehat{G}_{X_{4,7/3}}(x;t) = 1 +(12x_1 + 2x_2^2 + 40)t^2 + (48x_1x_2 + 180x_2 + 60)t^3+ \\ &+ (420x_1^2 + 168x_1x_2^2 + 
    3024x_1 + 6x_2^4 + 624x_2^2 + 360x_2 + 5544)t^4 + \\
    &+(3360x_1^2x_2 + 480x_1x_2^3 + 25200x_1x_2 + 5040x_1 + 1840x_2^3 + 1560x_2^2 + 47520x_2 + 18480)t^5+\ldots \\
\end{aligned}
$$

	\noindent\begin{minipage}{9cm}

	{\sc Mirror:} coefficients of the interior points,
	$$
		\begin{aligned}
		 a_{[0,1]} & =a_{[-1,1]}=a+3; \\
		 a_{[1,0]} & =a_{[-1,0]}=b.
		\end{aligned}
	$$
	
\end{minipage}
\begin{minipage}{2cm}
 \includegraphics[scale=0.8]{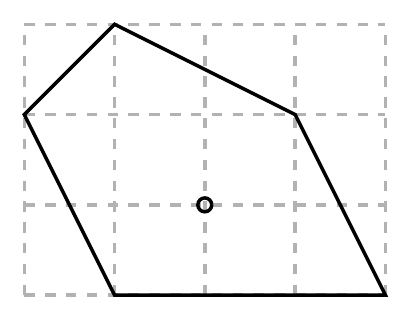}
 
 \begin{center}Polygon n.11\end{center}
\end{minipage}
 
 \begin{family}\label{subsec:X_4,4/3}
		 {\sc Name:} $X_{4,4/3}$
\end{family}
	{\sc Model:} complete intersection in toric fourfold
$$
 \begin{matrix}
 1 & 0 & 2 & 2 & 1 & 1 & \\
 0 & 1 & 1 & 1 & 2 & 2 & \\
 \end{matrix}
 \left|\begin{matrix}
  & 2 & 4 & & p_1 \\ & 4 & 2 & & p_2 \\
 \end{matrix}
 \right.
$$

	{\sc Extended twisted I-function:} 
$$
\begin{aligned}
I^{S}&(\xi,\tau,z)=z  + \tau p + \left(12 e^{\tau_1}q_1 + 12 e^{\tau_2}q_2 + 4\xi\right) \mathbf{1}_0 \\
&+ \left(3\rme^{\frac{-\tau_1 + 2 \tau_2}{3}} q_1^{- \frac 1 3} q_2^{\frac 2 3} + 
3\rme^{\frac{2 \tau_1 - \tau_2}{3}} q_1^\frac{2}{3} q_2^{- \frac 1 3} +xq_1^{-\frac 1 3}q_2^{-\frac 1 3} \right) \mathbf{1}_{\al} + O ( z \inv).
\end{aligned}
$$

	{\sc Specialised quantum period:} 
$$
 G_{X_{4,4/3}}(x;t) = \exp\left(-24t - 4 x t\right) \sum_{l_1, l_2, k \geq 0} \frac{(2l_1+4l_2+2k)!(4l_1+2l_2+2k)!}{l_1!l_2!(2l_1+l_2+k)!^2(l_1+2l_2+k)!^2k!}x^kt^{l_1+l_2+k}
$$

{\sc Regularization:}
$$
\begin{aligned}
    &\widehat{G}_{X_{4,4/3}}(x;t) = 1 + (20x^2 + 288x + 1064)t^2 + (96x^3 + 2352x^2 + 19224x + 52368)t^3 + \\
    &+(1188x^4 + 37440x^3 + 446496x^2 + 2384640x + 4807080)t^4+\\
    &+(10560x^5 + 431040x^4 + 7055680x^3 + 57885600x^2 + 237988800x +
        392223360)t^6+\ldots \\
\end{aligned}
$$

	\noindent\begin{minipage}{9cm}

	{\sc Mirror:} coefficients of the interior points
	$$
	\begin{aligned}
	a_{[1,1]} = a_{[1,-1]} = a_{[-1,-1]} = a_{[-1,1]} & = a+8; \\
	a_{[1,0]} = a_{[0,-1]} = a_{[-1,0]} = a_{[0,1]} & = 2a+14. \\
	\end{aligned}
	$$
	
\end{minipage}
\begin{minipage}{2cm}
 \includegraphics[scale=0.6]{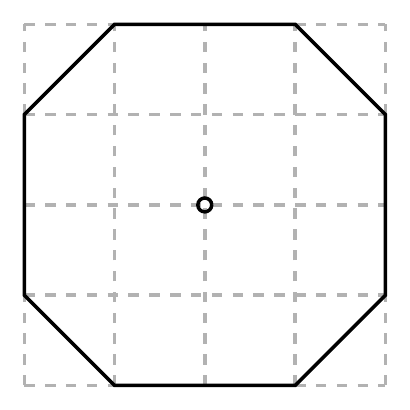}
 
 \begin{center}Polygon n.5\end{center}
\end{minipage}

 \begin{family}\label{subsec:X_4,1/3}
		 {\sc Name:} $X_{4,1/3}$
\end{family}

	{\sc Model:} complete intersection in toric fourfold
$$
 \begin{matrix}
 2 & 2 & 3 & 3 & 3 & 
 \end{matrix}
 \left|\begin{matrix}
  & 6 & 6 & & p
 \end{matrix}
 \right.
$$

	{\sc Non-extended twisted I-function:} 
$$
 I(\tau,z) = z  + \tau p + 600 \rme^{\tau}q \mathbf{1}_0 + 9 \rme^{\frac \tau 3} q^{\frac 1 3} \mathbf{1}_{\al} + O ( z \inv).
$$

	{\sc Specialised quantum period:} 
$$
  G_{X_{4,1/3}}(t) = \exp(-600t)\sum_{l\geq 0}\frac{6l!^2}{2l!^2 3l!^3}t^l.
$$

	{\sc Mirror:} this family does not admit a toric $\QQ$-Gorenstein degeneration.





\subsection{Surfaces with Fano index 1 and 6 singular points}
\setcounter{family}{27}
 \begin{family}\label{subsec:X_6,2}
		 {\sc Name:} $X_{6,2}$
\end{family}

	{\sc Model:} toric
$$
\begin{matrix}
 1 & 0 & 0 & 0 & 1 & -1 \\
 0 & 1 & 0 & 0 & 1 & 0 \\
 0 & 0 & 1 & 0 & 0 & 1 \\
 0 & 0 & 0 & 1 & -1 & 1 \\
\end{matrix}
 \begin{matrix}
  & & p_1 \\
  & & p_2 \\
  & & p_3 \\
  & & p_4 \\
 \end{matrix}
$$

	{\sc Non-extended twisted I-function:} 
$$
 I(\tau,z) = z +\tau p+ O ( z \inv).
$$

	{\sc Specialised quantum period:}
$$
 G_{X_{6,2}}(t)=\sum_{\substack{l_1,l_2,l_3,l_4 \geq 0 \\ l_1+l_2-l_4 \geq 0 \\ -l_1+l_3+l_4 \geq 0}}
 \frac{1}{l_1! l_2! l_3! l_4! (l_1+l_2-l_4)! (-l_1+l_3+l_4)!}t^{l_1+2l_2+2l_3+l_4}
$$

{\sc Regularization:}
$$
\begin{aligned}
 \widehat{G}_{X_{6,2}}(t)=1 + 6 t^2 + & 12 t^3 + 90 t^4 + 360 t^5 + \\ &+2040 t^6 + 10080 t^7 + 
 54810 t^8 + 290640 t^9 + \ldots \\
\end{aligned}
$$

	\noindent\begin{minipage}{9cm}

	{\sc Mirror:} 
	$$ f(0,\ldots,0;x,y)=\frac{x}{y^2} + \frac{1}{x y} + \frac{x^2}{y} + \frac{y}{x^2} + \frac{x}{y} + \frac{y^2}{x}
	$$
	
\end{minipage}
\begin{minipage}{2cm}
 \includegraphics[scale=0.6]{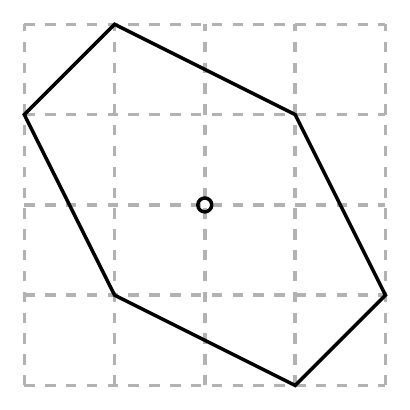}
 
 \begin{center}Polygon n.9\end{center}
\end{minipage}


	
\bibliography{biblio}

\end{document}